\documentclass[a4paper,reqno,11pt]{amsart} 
\usepackage[english]{babel}
\usepackage{tabls}
\usepackage{slashed}
\usepackage{tipa}
\usepackage{color}
\usepackage[all]{xy}
\usepackage{eucal}
\allowdisplaybreaks[2]

\usepackage{lmodern}
\usepackage{bm}
\usepackage{sansmath}

\usepackage[linktocpage=true]{hyperref}

\newcommand{\eps}{\epsilon}

\newcommand{\R}{\mathbb{R}}
\newcommand{\C}{\mathbb{C}}

\newcommand{\ud}{\text{d}}
\newcommand{\snabla}{\slashed \nabla}
\def\sdiv{\text{\texthtd}}
\def\scurl{\text{\texthtc}}
\def\stwist{\text{\texthtt}}

\newcommand{\coder}{\ud^*}
\newcommand{\FF}{\mathcal F}
\newcommand{\EE}{\mathcal E}
\newcommand{\HH}{\mathcal H} 
\newcommand{\KK}{\mathcal K} 
\newcommand{\LL}{\mathcal L}

\numberwithin{equation}{section}

\newcommand{\lAngle}{\mathopen{<}}
\newcommand{\rAngle}{\mathclose{>}}

\DeclareMathOperator*{\diverg}{div}

\def\clap#1{\hbox to 0pt{\hss#1\hss}}
\def\mathclap{\mathpalette\mathclapinternal}
\def\mathclapinternal#1#2{\clap{$\mathsurround=0pt#1{#2}$}}

\newtheorem{theorem}{Theorem}[section]  
\newtheorem{remark}[theorem]{Remark}  
\newtheorem{lemma}[theorem]{Lemma}

\newtheorem{definition}[theorem]{Definition}  
\newtheorem{proposition}[theorem]{Proposition}  

\title[Hertz potentials and spin-$s$ fields]{Hertz potentials and asymptotic 
properties of massless fields}

\author[L. Andersson]{Lars Andersson}
\email{laan@aei.mpg.de}
\address{Albert Einstein Institute, Am M\"uhlenberg 1, D-14476 Potsdam,
  Germany}

\author[T. B\"ackdahl]{Thomas B\"ackdahl}
\email{thomas.backdahl@aei.mpg.de}
\address{Albert Einstein Institute, Am M\"uhlenberg 1, D-14476 Potsdam,
  Germany}

\author[J. Joudioux]{J\'er\'emie Joudioux}
\email{jeremie.joudioux@aei.mpg.de}
\address{Albert Einstein Institute, Am M\"uhlenberg 1, D-14476 Potsdam,
  Germany}

\begin{document}

\date{\today \ {\em AEI ref: AEI-2013-161}}

\begin{abstract}
In this paper we analyze Hertz potentials for free massless spin-$s$ fields
on the Minkowski spacetime, with data in weighted Sobolev spaces. 
We prove existence and pointwise estimates for the Hertz potentials using a 
weighted estimate for the wave
equation. This is then applied to give 
weighted estimates for the
solutions of the spin-$s$ field equations, for arbitrary half-integer $s$. In
particular, the peeling properties of the free massless spin-$s$ fields are
analyzed for initial data in weighted
Sobolev spaces with arbitrary, non-integer weights. 
\end{abstract}

\maketitle


\section{Introduction}
The analysis by Christodoulou and Klainerman of the decay of massless fields
of spins 1 and 2 on Minkowski space \cite{Christodoulou:1990dd} served as an
important preliminary for their proof of the non-linear stability of
Minkowski space \cite{Christodoulou:1993vma}. The method used in
\cite{Christodoulou:1990dd} was based on energy estimates using the vector
fields method, see \cite{Klainerman:1985wn}. 
This approach was extended to fields of arbitrary spin by Shu
\cite{Shu:1991ta}. The approach of \cite{Christodoulou:1993vma} to the problem of
nonlinear stability of Minkowski space was later extended by Klainerman and
Nicolo \cite{2003CQGra..20.3215K}
to
give the full peeling behavior for the Weyl tensor 
at null infinity.

The vector fields method makes use of the conformal symmetries of Minkowski
space to derive conservation laws for higher order energies, which then
via the Klainerman Sobolev inequality \cite{Klainerman:1987hz}
give pointwise estimates for the solution of the wave equation. An analogous
procedure is used for the higher spin fields in the papers cited above. 
This procedure  gives pointwise  
decay estimates for the solution of the Cauchy problem of the wave equation and
the spin-$s$ equation, for 
initial data of one particular falloff at spatial infinity. The conditions on the initial data
originate in the
growth properties of the conformal Killing vector fields on Minkowski space, which are used in the energy estimates. 

Let $H^j_\delta$ be the weighted $L^2$ Sobolev spaces on $\R^3$, that is to say the space of functions $\phi$ for which
$$
\sum_{k=0}^j\Vert \lAngle r\rAngle^{-\left(3/2 + \delta\right) +k}D^k \phi  \Vert_{L^2(\mathbb{R}^3)}^2<\infty, \quad  \text{ where }\lAngle r\rAngle = (1+r^2)^{1/2}.
$$  We use the
conventions\footnote{The spaces $H^j_\delta$ are in \cite{Bartnik:1986dq}
denoted by
$W^{j,2}_{\delta}$.}
of Bartnik \cite{Bartnik:1986dq}. 
Since we shall use the 2-spinor formalism, as defined in Section \ref{sec:conventions}, 
we work here and throughout the paper on Minkowski space with signature ${+}\,{-}\,{-}\,{-}$. 
Consider the Cauchy problem for
the wave equation 
\begin{align}
\square \phi &= 0, \label{eq:wave-intro} \\ 
\phi \big{|}_{t=0} = f\in H^j_{-3/2}, \quad &  \partial_t \phi \big{|}_{t=0} = g\in H^{j-1}_{-5/2}.\nonumber
\end{align} 
Then, for $j \geq 2$, 
one has the
estimate 
\cite{Klainerman:1985wn} 
\begin{equation}\label{eq:CKwave}
 |\phi(x,t)|
\leq C\lAngle u \rAngle^{-1/2} \lAngle v \rAngle^{-1}
(\Vert f\Vert_{j,-3/2} + \Vert g\Vert_{j-1,-5/2}),
\end{equation} 
where $\lAngle u\rAngle = (1 + u^2)^{1/2}$, $u=t-r$ and $v=t+r$. 
On the other hand, if one considers the wave equation \eqref{eq:wave-intro}
on the flat 3+1 dimensional Minkowski spacetime 
as a special case of the conformally covariant form of the
wave equation
$$
(\nabla^a \nabla_a +  R/6) \phi = 0,
$$
the condition on the initial data which is compatible with regular conformal
compactification is 
\begin{equation}\label{eq:confdata}
\partial^\ell f = \mathcal{O}(r^{-2-\ell}), \quad \partial^{\ell} g = \mathcal{O}(r^{-3-\ell}).
\end{equation} 
Making use of standard energy estimates in the conformal compactification
of Minkowski space one arrives, after undoing the conformal compactification, 
at 
\begin{equation}\label{eq:confdecay}
|\phi(x,t)| = \mathcal{O}\left (\lAngle u \rAngle^{-1} \lAngle v \rAngle^{-1} \right ) 
\end{equation} 
see the discussion in \cite[\S 6.7]{MR1466700}. 
In particular, there is an extra $r^{-1/2}$ falloff in the condition
\eqref{eq:confdata} on the initial data compared to \eqref{eq:wave-intro} as well as an 
additional factor $\lAngle u \rAngle^{-1/2}$ decay
in the retarded time coordinate $u$ in \eqref{eq:confdecay} 
compared to \eqref{eq:CKwave}. 

Let us now consider the case of higher spin fields. Let $2s$ be a positive
integer and let $\phi_{A\dots F}$ be a totally symmetric spinor field of
spin-$s$, i.e with $2s$ indices. 
The Cauchy problem for a massless
spin-$s$ field is 
\begin{align*} 
\nabla_{A'}{}^A \phi_{A \dots F} &= 0 , \\ 
\phi_{A \dots F} \big{|}_{t=0} &= \varphi_{A \dots F} .
\end{align*} 
For $s \geq 1$,  the Cauchy datum $\varphi_{A\dots F}$ must satisfy the constraint equation 
$$
D^{AB} \varphi_{AB\dots F} = 0,
$$
where $D_{AB}$ is the intrinsic space spinor derivative on $\Sigma$, see
Section~\ref{sec:conventions}. The spin $1/2$ case does not have constraints. 

One of the main differences in the asymptotic behavior between a massless
scalar field satisfying a wave equation and a massless higher spin field is
the existence of a hierarchy of decay rates for the different null components
of the field along the outgoing null directions. This property, known as peeling, was first pointed out
by Sachs in 1961 \cite{Sachs:1961tw}.

Let $o_A, \iota_A$ be a spin dyad aligned with the outgoing and ingoing null
directions $\partial_v, \partial_u$, and let $\phi_{\bm{i}}$ be the scalars
of $\phi_{A\dots F}$ defined by 
$$
\phi_{\bm{i}}  = \phi_{A_1 \dots A_{\bm{i}} B_{\bm{i}+1} \dots B_{2s}} \iota^{A_1} \cdots
\iota^{A_{\bm{i}}} 
 o^{B_{\bm{i} +1}} \cdots o^{B_{2s}}.
$$
One says that $\phi_{A\dots F}$ satisfies the peeling property if the
components $\phi_{\bm{i}}$ satisfy 
$$
\phi_{\bm{i}} = {\mathcal O} ( r^{\bm{i} - 2s -1} ) , 
$$ along
the outgoing null geodesics with affine parameter $r$.  

In \cite{MR0175590}, Penrose gave two arguments for peeling of
massless fields on Minkowski space. The first, cf. 
\cite[\S 4]{MR0175590}, makes use of a 
representation of the field in terms of a Hertz
potential of order $2s$, \emph{i.e.} the field is written as a
derivative of order $2s$ of a potential satisfying a wave equation. 
Penrose assumed that the Hertz potential 
decays at a specific rate along outgoing
null rays. 
He then infered the peeling property from this decay assumption.

The second approach presented by Penrose, cf. \cite[\S 13]{MR0175590}, 
is based on the just mentioned fact together with the conformal invariance of
the spin-$s$ field equation. Solving the Cauchy problem in the conformally
compactified picture, as was discussed for the wave equation in \cite[\S
  6.7]{MR1466700}, and taking into account the effect of the conformal
rescaling, one recovers the peeling property for the solution of the massless
spin-$s$ equation on Minkowski space. 
Based on this analysis, Penrose conjectured that the peeling of massless fields at null infinity should be a generic property of asymptotically simple space-times.

The estimate proved in \cite{Christodoulou:1990dd} for the spin-1 or Maxwell
field can be stated in the present notation as 
$$
|\phi_{\bm{i}}(t,x)| \leq C \lAngle u \rAngle^{1/2 - \bm{i}} \lAngle v
\rAngle^{\bm{i}-3} \Vert \varphi_{AB} \Vert_{j,-5/2}, \quad
\text{ for } \bm{i}= 1,2,\quad  j\geq 2,
$$
while for the component $\phi_0$ one has 
$$
|\phi_0(t,x)| \leq C r^{-5/2} \Vert \varphi_{AB} \Vert_{j,-5/2},
$$
along outgoing null rays. Thus, this result does not give the peeling 
property for all components of $\phi_{AB}$, 
which is due to the fact that the norm $\Vert \varphi_{AB}
\Vert_{j,-5/2}$ is not compatible with the conformal compactification of
Minkowski space. Similarly for the spin-2 case, the result in
\cite{Christodoulou:1990dd} gives peeling for $\phi_{\bm{i}}$, $\bm{i}
= 2,3,4$ for initial data in $H^j_{-7/2}$, 
while peeling fails to hold for $\phi_{\bm{i}}$, $\bm{i} =
0,1$. 
On the other hand, the condition on the initial datum which is compatible with
a regular conformal compactification, and also gives peeling, is for a spin-$s$ field
\begin{equation}\label{eq:confcond}
\partial^\ell \varphi_{A\dots F} = \mathcal{O}(r^{-2s-2-\ell}).
\end{equation} 

In this paper we shall follow an approach outlined by Penrose in
\cite[\S 6]{MR0175590} 
to give a weighted decay estimate for spin-$s$ fields of arbitrary, half-integer
spin-$s$. The result proved here admits conditions on the initial data which 
include the ones 
considered in \cite{Christodoulou:1990dd,Shu:1991ta}, as well as 
conditions which
are compatible with peeling, but also general weights. 
The results of this paper clarify
the relation between the condition on the initial
datum and the peeling property of the solution of the spin-$s$ field equation. 
In this paper we shall make use of some estimates for elliptic equations in
weighted Sobolev spaces, and for technical reasons these are not compatible
with the integer
powers or $r$ as in \eqref{eq:confcond}.

The method we shall use is based on the notion of Hertz potentials.
The reader can refer for background to Stewart \cite{stewart:1979}, Fayos \emph{et al.} 
\cite{fayos:llanta:llosa:1985} Benn \emph{et al.} \cite{benn:charlton:kress:1997} and references therein. Since Minkowski space is
topologically trivial, there is no obstruction to representing 
a Maxwell field on Minkowski space in terms of a Hertz
potential. However, this general fact does not provide estimates for the
potential.  In this
paper we prove the necessary estimates not only for the Maxwell field but for
fields with general half-integer spins. 

To introduce the method we here consider the spin-1 case, \emph{i.e.} the
Maxwell field on 3+1 dimensional Minkowski space. 
With our choice of signature, the metric on the spatial slices is negative definite. 

The Maxwell field, in the absence of sources, is a real differential 2-form $F_{ab}$ 
which is closed and
divergence free. 
For convenience we consider the complex anti-self-dual form 
$$
\FF_{ab} = F_{ab} + i * F_{ab}, 
$$
which corresponds to a symmetric 2-spinor $\phi_{AB}$ via
\begin{equation} \label{eq:F=phi}
\FF_{ab} = \phi_{AB} \eps_{A'B'} .
\end{equation} 
In terms of $\FF_{ab}$, the Maxwell equation
is simply 
\begin{equation}\label{eq:maxwellinforms} 
(\ud \FF)_{abc} = 0.
\end{equation} 
\newcommand{\ii}{\bm{i}} 
\newcommand{\norm}{\xi} 
Let $\norm^a = (\partial_t)^a$ 
be the unit normal to the Cauchy surface $\Sigma = \{ t = 0\}$.
Given a complex 1-form $\EE_a$ on $\Sigma$,
satisfying the Maxwell constraint equation
\begin{equation}\label{eq:maxwellconstr}
\ud^*\EE=0,
\end{equation}
there is a unique solution of the Maxwell equation such that 
$$
(\FF_{ab} \norm^b)\big{|}_{\Sigma} =  \EE_a.
$$
Now, let $\HH_{ab}$ be a self-dual
2-form which solves the wave equation 
\begin{equation}\label{eq:H-wave}
\square 
\HH_{ab} = 0,
\end{equation} 
where 
$\square = \ud \coder + \coder \ud$ 
is the Hodge wave operator, and
$\coder = * \ud *$ is the exterior co-derivative in dimension 4. 
Defining the form $\FF_{ab}$ by 
\begin{equation}\label{eq:F=H}
\FF_{ab} = \ud \coder \HH_{ab} ,
\end{equation} 
we have using \eqref{eq:H-wave} that $\FF_{ab}$ is anti-self-dual, and solves the
Maxwell equation. The form $\HH_{ab}$ 
is called a Hertz-potential for $\FF_{ab}$. Since
we are working on Minkowski space, the wave equation \eqref{eq:H-wave} is
just a collection of scalar wave equations for the components of $\HH_{ab}$, and
hence the solution to \eqref{eq:H-wave} for given Cauchy data can be analyzed
using results for the scalar wave equation. Thus, if we are able to relate
the Cauchy data for the Maxwell field $\FF_{ab}$ 
to the Cauchy data for $\HH_{ab}$, we
may use the Hertz potential construction to prove estimates for the solution
of the Maxwell field equation, starting from estimates for the wave
equation. 

Let the complex 
1-form $\KK_a$ be the ``electric field'' corresponding to $\HH_{ab}$, 
$$
\KK_a = \HH_{ab} \norm^b.
$$
A calculation shows that if $\FF_{ab}$ is defined in terms of $\HH_{ab}$ by
\eqref{eq:F=H}, the Cauchy data for \eqref{eq:H-wave}
is related to the
Cauchy data for $\FF_{ab}$ by 
\begin{equation}\label{eq:E=K} 
\EE_a =  - *\ud* \ud \KK_a - i *\ud \partial_t\KK_a,
\end{equation} 
where in the right hand side we restrict $\KK_a$ and $\partial_t \KK_a$ to
$\Sigma$, and 
$\ud, *$ 
act on objects on $\Sigma$\footnote{Recall that, in dimension 3, for $p$-forms, the exterior co-derivative is given by $\ud^\ast = (-1)^p \ast\ud \ast{}$. }. 
The constraint equation $\coder \EE_a = 0$  holds automatically for $\EE_a$
given by \eqref{eq:E=K}. 

Now, in order to prove estimates for the Maxwell equation with data 
$\EE_a \in H^j_\delta$, satisfying $\coder \EE = 0$, 
it is sufficient to show that for any such $\EE_a$, there exists a 1-form
$\LL_a  \in H^{j+1}_{\delta+1}$ such that 
\begin{equation}\label{eq:E=L}
\EE_a = - i *\ud \LL_a .
\end{equation} 
Then taking $\HH_{ab}$ to be a solution of \eqref{eq:H-wave} with Cauchy data 
$$
\HH_{ab} \big{|}_{t=0} = 0, \quad 
\left ( \partial_t \HH_{ab} \norm^b \right ) \big{|}_{t=0}=\LL_a ,
$$
gives a solution to the Maxwell
equation via \eqref{eq:F=H}. Estimates for the wave equation can thus be
applied to give estimates for the solution of the Maxwell field equation.

The operator $*\ud$ acting on 1-forms, which appears in Eqs.~\eqref{eq:E=K} and
\eqref{eq:E=L} is simply the curl operator. The first important thing to notice about Eq.~\eqref{eq:E=L} is the fact this is an overdetermined system of partial differential equations: the electric field has to satisfy constraints in order to ensure the existence of a solution. The integrability of these equations is well-known to be described by the elliptic complex
\begin{equation}\label{eq:derham}
 C^\infty(\R^3,\R)\stackrel{\ud}{\longrightarrow} \Lambda^1\stackrel{\ast \ud}{\longrightarrow} \Lambda^1  \stackrel{\coder}{\longrightarrow}C^\infty(\R^3,\R).
\end{equation}
This complex, which is derived from the standard de Rham complex plays a crucial role in the analysis of Hertz potentials. The geometric constraint $\coder \EE_a =0$ guarantees that Eq.~\eqref{eq:E=L} can be solved, at least formally. This is described in Section~\ref{sec:Integrability}.

In Proposition~\ref{proprepspin1} below, we shall prove  
that for non-integer weights $\delta $, 
$*\ud : H^{j+1}_{\delta+1} \to \ker \coder \cap H^j_{\delta} $
is a surjection, and the estimate 
\begin{equation}\label{estimateLE}
\Vert \LL_a\Vert_{j+1,\delta+1} \leq C \Vert\EE_a\Vert_{j,\delta}
\end{equation}
holds, for some constant $C$. 
It is instructive to consider two special weights in the spin-1 case. 
First let $\EE_a\in H^j_{-5/2}$ be a solution to the Maxwell constraint equation \eqref{eq:maxwellconstr}. This corresponds to the case considered in
\cite{Christodoulou:1990dd}. 
We shall now construct a solution $\LL_a\in H^{j+1}_{-3/2}$ to \eqref{eq:E=L}, which will give us Cauchy data for the Herz potential.
Recall that the Laplacian 
$\Delta = \ud \coder + \coder \ud$ is a surjection $H^{j+2}_{-1/2} \to H^j_{-5/2}$, cf.  Proposition~\ref{prop:elliptictoolbox}. Hence, we can find $\mathcal{Q}_a\in H^{j+2}_{-1/2}$ such that 
$$
\EE_a=\Delta \mathcal{Q}_a.
$$
Making use of the constraint equation \eqref{eq:maxwellconstr}, and the fact that $\Delta$ and $\ud^*$ commute, we find
$$
0=\ud^* \EE=\ud^* \Delta \mathcal{Q}= \Delta \ud^*\mathcal{Q}.
$$
Hence, $\ud^*\mathcal{Q}\in\ker \Delta \cap H^{j+1}_{-3/2}$. Due to injectivity of $\Delta:H^{j+1}_{-3/2} \to H^{j-1}_{-7/2}$, we have $\ud^*\mathcal{Q}=0$, and hence, 
$$\EE_a=\ud \coder\mathcal{Q}_a + \coder \ud\mathcal{Q}_a=\coder \ud\mathcal{Q}_a.$$
Therefore we can take $\LL_a=i (*\ud \mathcal{Q})_a$. In this situation, full peeling does not hold, and the Cauchy data for the Hertz potential is
in $H^{j+2}_{-1/2} \times H^{j+1}_{-3/2}$. Secondly, we consider the case $\EE_a \in \ker \coder  \cap H^j_{-7/2}$ where 
full peeling holds. In this case, the Cauchy data for the Hertz potential is in $H^{j+2}_{-3/2} \times H^{j+1}_{-5/2}$. The relevant fact about the Laplacian is now that $\Delta : H^{j+2}_{-3/2} \to H^j_{-7/2}$ is Fredholm with cokernel consisting of constant forms. Since a constant form $\eta_a$ 
is automatically closed, using the integrability of the complex \eqref{eq:derham}, cf. Remark~\ref{integrability4}, it is also exact, $\eta_a = (\ud f)_a$ for some $f$. This fact is usually referred to as the Poincar\'e Lemma and can as we will see below, be generalized to higher spin. It follows that the cokernel of $\Delta$ is automatically $L^2$-orthogonal to $\ker \coder  \cap H^j_{-7/2}$. Hence, 
$$\Delta:H^{j+2}_{-3/2} \to \ker \ud^*\cap  H^j_{-7/2}$$
is surjective. Using this fact, by the argument above we can find a solution $\LL_a\in H^{j+1}_{-5/2}$ to equation \eqref{eq:E=L}. For the case of general weights, see Proposition~\ref{proprepspin1}.

It is important to notice that, for weights $\delta < -4$, none of the elements in the cokernel have a preimage under the Laplacian. 
However, in Lemma \ref{lemma:orthogonalpreimage}, we prove that one can add an element $\zeta_a$ in the image of $\ast \ud$, so that $E_a-i(\ast \ud\zeta)_a$ is orthogonal to the cokernel of $\Delta$, so that a preimage can be found. This way we can always find a preimage under $\ast \ud$, even if we can not find a preimage under $\Delta$.

In the spin-$s$ case, the argument follows the same outline. However, this requires an extension of 
the complex \eqref{eq:derham} to arbitrary spin, and relating the corresponding operators to an elliptic operator, which in this case will be a power of the Laplacian $\Delta^{\lfloor s \rfloor}$.
Nonetheless, two important preliminary results have to be proved to this end: an extension of the complex 
\eqref{eq:derham} to arbitrary spin and a decay result for the solution of the scalar wave equation with initial 
data in Sobolev spaces of arbitrary non-integer weights and its derivatives.

As far as the authors know, not much work has been performed to extend the Hodge-de Rham theory to arbitrary weighted Sobolev spaces on the one hand and to arbitrary spin on the other hand. Cantor \cite{Cantor:1981bs} proved the first steps of a Hodge decomposition for tensors in weighted Sobolev spaces and Weck and Witsch \cite{Weck:1994ei} gave a Hodge-Helmholtz decomposition for forms in $L^p_{\delta}$ spaces. These two results are nonetheless not sufficient for our purposes.  As far as the extension to arbitrary spin is concerned, for analytic solutions of the massless field equation of arbitrary spin on the Euclidean 4-sphere, Woodhouse \cite[Section 10]{Woodhouse:1985wx}, describes the procedure to construct the intermediate potentials, up to the penultimate, in the form language. Penrose \cite{Penrose:te, Penrose:tk} extended this construction for the spin $3/2$ in a wider context. Using this formalism, a topological condition on the domain of validity of the representation on the sphere  is given.

In order to apply standard elliptic theory, one has to perform a $3+1$ splitting of the equation relating the field and its Hertz potential
\begin{equation}\label{eq:hertz}
\phi_{A\dots F} = \nabla_{AA'}\dots \nabla_{FF'}\chi^{A'\dots F'}.
\end{equation}
Assuming that the initial data for the potential are
$$
\chi^{A'\dots F'}|_{t=0} = 0 \text{ and } \partial_t\chi^{A'\dots F'}|_{t=0} = \sqrt{2}\tau^{AA'} \dots\tau^{FF'} \zeta_{A\dots F},
$$
the $3+1$ splitting (performed in detail in Section~\ref{sec:spacespinorsplit}) of Eq.~\eqref{eq:hertz} gives
\begin{equation}\label{eq:varphi=zeta}
\varphi_{A\dots F} = \left(\mathcal{G}_{2s}\zeta\right)_{A\dots F},
\end{equation}
where $\mathcal{G}_{2s}$ is a differential operator of order $2s-1$ (see Definition 
\ref{def:mathcalg}). In the spin-$1$ case, if one translates the spin formalism into the form language, one recovers Eq.~\eqref{eq:E=L}.

Finding a generalization of the complex~\eqref{eq:derham} to arbitrary spin, which encompasses equation~\eqref{eq:varphi=zeta}, and especially the operator $\mathcal{G}_{2s}$, is then a necessary step in the construction of the initial data of the Hertz potential. A generalization of the de Rham complexes have been introduced in the context of the deformation of conformally flat structures by Gasqui-Goldschmidt \cite{Gasqui:1984vu} and, later generalized by Beig \cite{Beig:1997wo} (see Section~\ref{sec:Integrability}). They obtained differential complexes corresponding to the spin-2 case. We generalize, on $\R^3$, their results to arbitrary spin. One introduces the fundamental operators
\begin{align*}
(\sdiv_{2s}\phi)_{A_1\dots A_{2s-2}}\equiv{}&D^{A_{2s-1}A_{2s}}\phi_{A_1\dots A_{2s}},\\
(\stwist_{2s}\phi)_{A_1\dots A_{2s+2}}\equiv{}&D_{(A_1 A_2}\phi_{A_3\dots A_{2s+2})}.
\end{align*}
We prove that, if $\mathcal{S}_{2s}$ is the space of symmetric space spinor fields on $\R^3$, the sequence
\begin{equation*} 
\mathcal{S}_{2s-2}\stackrel{\stwist_{2s-2}}{\longrightarrow} \mathcal{S}_{2s} \stackrel{\mathcal{G}_{2s}}{\longrightarrow} \mathcal{S}_{2s}\stackrel{\sdiv_{2s}}{\longrightarrow} \mathcal{S}_{2s-2},
\end{equation*}
is an elliptic complex. For $s=1$, one recovers the  complex~\eqref{eq:derham}. In the situation considered by Gasqui-Goldschmidt and Beig for the spin-2 case, the operator $\mathcal{G}_{4}$ is the linearized Cotton-York tensor. See Section~\ref{sec:Integrability} for details.

The existence of a solution to \eqref{eq:E=L} with the estimate \eqref{estimateLE}
is then used together with a weighted 
estimate for the solution of the wave equation with initial $(f,g) \in
H^j_\delta \times H^{j-1}_{\delta-1}$. 
As we have not found a
sufficiently general result in the literature, in particular which covers the
range of weights $\delta > -1$ which we need 
for the applications to the Hertz
potential in the range where full peeling fails to hold (including the
situation considered in \cite{Christodoulou:1990dd}), we prove the required
result in Section~\ref{sec:estsol}. This result consists in a direct estimate for 
the solution of the wave equation, using the representation formula. 
For $\delta < -1$, we have in the exterior region 
$$
 |\phi(t,x)|\leq  C  \lAngle v\rAngle^{-1}\lAngle u\rAngle^{1+\delta}
\left ( \Vert f\Vert_{3,\delta}+\Vert g\Vert_{2,\delta-1} \right ) ,
$$
see Proposition~\ref{prop:decaywave}.

The main result of the paper, stated in Theorem \ref{maintheorem}, combines the
analysis of the Hertz potential Cauchy data in weighted Sobolev spaces with
the weighted estimate for the solution of the wave equation to provide a
weighted estimate for the solution to the massless spin-$s$ field equation. 
The peeling properties of the spin-$s$ field with initial data in weighted
Sobolev spaces are analyzed in detail.

Here it is important to note that the detailed decay estimates for the
components of the \emph{massless spin-$s$ field} $\phi_{A\dots F}$ relies on the
decay of the Hertz potential $\chi^{A' \dots F'}$, for which all components decay as 
solutions to the scalar \emph{wave equation}.
The decay properties of the components of $\phi_{A\dots F}$ are 
due to their relation to the derivatives of $\chi^{A'\dots F'}$ in terms of a
null tetrad. In particular, the falloff condition on the initial datum of the massless 
fields which ensures that the peeling property holds is given explicitly. All the intermediate 
states, where the peeling fails because the initial datum does not falloff sufficiently, such 
as the decay result obtained by Christodoulou-Klainerman \cite{Christodoulou:1990dd}, are 
clearly explained in terms of a decay result for the scalar wave equation.

\subsection*{Overview of this paper} 
 In Section~\ref{sec:prelim} we state our conventions and recall some basic
 facts about elliptic operators on weighted Sobolev spaces. 
In particular we introduce the Stein-Weiss operators 
divergence $\sdiv$, curl $\scurl$ and the twistor 
operator $\stwist$ for higher spin fields, as well as the fundamental higher order
operator $\mathcal G$ originating in the 3+1 splitting of the Hertz potential equation. 
In Section~\ref{sec:Integrability} we use these to introduce a 
generalization of the de Rham complex for spinor fields. 
The problem of constructing initial data
 for the Hertz potential is solved in Section~\ref{Section:ConstructPotentials}. 
The weighted estimate for the wave equation is given in Section~\ref{sec:estsol}, 
and the resulting estimates for spin-$s$ fields generated by potentials is given in 
Section~\ref{sec:estspin}. 
Everything is then tied together in Section~\ref{sec:mainres}, where the 3+1 splitting 
of the potential equation is considered and the potential is constructed. The section is 
then concluded with the estimates for spin-$s$ fields.
Appendix~\ref{sec:algprop} contains some results on the operator $\mathcal G$ used in the 
analysis of the elliptic complex introduced in Section~\ref{sec:Integrability} as 
well as for the construction of the initial data for the Hertz potential in
Section~\ref{Section:ConstructPotentials}.


\section{Preliminaries} \label{sec:prelim}
\subsection{Conventions}\label{sec:conventions}
In this paper, we will only work on Minkowski space time. We will use Cartesian coordinates $(t, x^1, x^2, x^3)$ as well as the corresponding spherical coordinates $(t,r,\theta,\phi)$. The spinor formalism with the conventions of \cite{Penrose:1986fk} is extensively used. For important parts of the paper, 3+1 splittings of spinor expressions are performed. The space spinor formalism as introduced in \cite{Sommers:1980dd} is used for this purpose. In this case, the conventions of \cite{Backdahl:2010ig} are adopted. We will always consider the space spinors on the $\{t=\text{const.}\}$ slices of Minkowski space with normal $\tau_{AA'}=\sqrt{2}\nabla_{AA'}t$. Observe that a \emph{negative} definite metric on these slices is used.

The Minkowski space-time $(\mathbb{R}^4, \eta_{\alpha\beta})$ is endowed with its standard connection $\nabla_a = \nabla_{AA'}$. The time slice $\{t=0\}$ is endowed with the connection $D_a = D_{AB}$ defined by
$$
D_{AB} = \tau_{(A}{}^{A'} \nabla_{B)A'}
$$
where $\tau_{AA'}$ is the timelike vector field defined above. Its relation to the connection of the ambient space-time is given by 
$$
\nabla_{AA'}=\tfrac{1}{\sqrt{2}}\tau_{AA'}\partial_t-\tau^{B}{}_{A'}D_{AB}.
$$

Let $S_k$ denote the vector bundle of symmetric valence $k$ spinors on $\mathbb{R}^3$. Furthermore, let $\mathcal{S}_k$ denote the space of smooth ($C^\infty$) sections of $S_k$.
\begin{definition}
Let $\mathcal{P}^{<\delta}(S_k)$ denote the finite dimensional subspace of $\mathcal{S}_k$ spanned by constant spinors with polynomial coefficients of degree $<\delta$. 
\end{definition}
Observe that with $\delta\leq 0$, $\mathcal{P}^{<\delta}(S_k)$ is just the trivial space $\{0\}$.

\subsection{Analytic framework}

We introduce in this section the analytic framework which is necessary to understand the propagation of the field as well as the geometric constraints. We will use the conventions of Bartnik \cite{Bartnik:1986dq}. Even though Bartnik's paper only gives statements for functions, we can easily extend this to space spinors on Euclidean space.

We recall first the standard norms, coming from the Hermitian space spinor product.
\begin{definition} The Hermitian space spinor product is given by
\begin{equation*}
\langle \zeta_{A\dots F},\phi_{A\dots F}\rangle = \zeta_{A\dots F} \widehat{\phi}^{A\dots F},
\end{equation*}
where  
$
\widehat{\phi}^{A\dots F} = \tau^{AA'}\dots \tau^{FF'} \overline{\phi}_{A'\dots F'}
$ and $\tau_{AA'}=\sqrt{2}\nabla_{AA'}t$.
The pointwise norm of a smooth $\phi_{A\dots F}$ is defined via
$$
\vert\phi_{A\dots F} \vert^2 = \phi_{A\dots F} \widehat{\phi}^{A\dots F}.
$$
The pointwise norm of the derivatives of the smooth spinor $\phi_{A\dots F}$ on $\R^3$ is given by
$$
\vert D_a \phi_{A\dots F}\vert^2 = \delta^{ab} D_{a}\phi_{A\dots F} \widehat{D_b\phi}{}^{A\dots F},$$
where $\delta_{ab}$ is the standard Euclidean metric on $\R^3$. The norm of higher order derivatives is defined similarly.
\end{definition}
\begin{remark}\label{rem:DReal} The identity
$$
D_{AB} \widehat{\phi}_{A\dots F} = - \widehat{D_{AB} \phi}_{A\dots F},
$$
holds, due to the fact that the operator $D_{AB}$ is real.
\end{remark}

\begin{definition}
The $L^2$-norm of a smooth spinor field in $\R^3$ is defined by
$$
\Vert\phi_{A\dots F}\Vert_{2} = \left(\int_{\R^3}\vert\phi_{A\dots F} \vert^2 \ud \mu_{\R^3}\right)^{1/2},
$$
where $\ud \mu_{\R^3}$ is the standard volume form on $\R^3$. 
The $L^2$-norm is also defined for derivatives in the same way using the pointwise definition above.
\end{definition}
If $u$ is a real scalar, its Japanese bracket 
is defined by
$$
\lAngle u\rAngle = (1+u^2)^{1/2}.
$$
We next define the weighted Sobolev norms, which will be used to describe the asymptotic behavior of initial data at space-like infinity.
\begin{definition}[Weighted Sobolev spaces] Let $\delta$ be a real number and $j$ a nonnegative integer.\\
The completion of the space of smooth spinor fields in $\mathcal{S}_{2s}$ with compact support in $\R^3$ endowed with the norm
$$
\Vert \phi_{A\dots F}\Vert^2_{j, \delta} = \sum_{n=0}^j \left\Vert \lAngle r \rAngle^{-(\delta+\frac32) + n} D^n\phi_{A\dots F}\right\Vert^2_{2},
$$
is denoted by $H^j_{\delta}(S_{2s})$.
\end{definition}
For $\delta = - 3/2$, the weighted spaces $H^0_{-3/2}(S_{2s})$ 
are the standard Sobolev spaces $L^2(S_{2s})$.

Many well-known properties can be proved about these spaces -- see for instance \cite[Theorem~1.2]{Bartnik:1986dq} for more details. The only property, crucial to obtain the pointwise estimates, is the following Sobolev embedding (cf. \cite[Theorem~1.2, (iv)]{Bartnik:1986dq}, specialized to dimension 3).
\begin{proposition}\label{prop:decaysob} Let $\delta$ be a real number and $j\geq 2$ an integer. Then, any spinor field in $H^j_\delta(S_{2s})$ is in fact continuous and there exists a constant $C$ such that, for any $\phi_{A\dots F}$ in $H^j_\delta(S_{2s})$
 $$
 \vert\phi_{A\dots F}(x) \vert \leq C \lAngle r \rAngle^\delta \Vert\phi_{A\dots F}\Vert_{2, \delta},
 $$
 and, in fact,
 $$
 \vert\phi_{A\dots F}(x) \vert = o\bigl( r^{\delta}\bigr) \text{ as } r \rightarrow \infty.
  $$
\end{proposition}

We finally recall the following properties of elliptic operators, restricting ourself to powers of the Laplacian. The result stated is a combination of the standard results, see e.g.,   \cite{Bartnik:1986dq,McOwen:1980gz,Cantor:1981bs,Kat80}.
\begin{proposition}\label{prop:elliptictoolbox} Let $j,l$ be non-negative integers such that $j\geq 2l$, $s$ be in $\tfrac{1}{2}\mathbb{N}_{0}$ and $\delta$ be in $\R\setminus\mathbb{Z}$. The formally self-adjoint elliptic operator of order $2l$
$$
\Delta^l_{2s}: H^{j}_{\delta}(S_{2s}) \longrightarrow H^{j-2l}_{\delta-2l}(S_{2s})
$$
is Fredholm and satisfies
\begin{itemize}
\item its kernel is a subspace of $\mathcal{P}^{<\delta}(S_{2s})$; in particular, $\Delta^l_{2s}$ is injective when $\delta<0$;
\item its co-kernel is a subspace of  $\mathcal{P}^{<-3-\delta+2l}(S_{2s})$; in particular, $\Delta^l_{2s}$ is surjective when $\delta>2l-3$.
\end{itemize}
Furthermore, 
there exists a constant $C$ such that, for all spinor fields in $H^{j}_{\delta}(S_{2s})$,
$$
\inf_{\psi_{A\dots F} \in \ker(\Delta_{2s}^l)\cap H^{j}_{\delta}(S_{2s}) } \Vert \phi_{A\dots F} + \psi_{A\dots F} \Vert_{j,\delta} \leq C \Vert\Delta^{l}_{2s}\phi_{A\dots F}\Vert_{j-2s,\delta-2s}.
$$
In fact the infimum is attained, and there exists a $\psi_{A\dots F}$ in $\ker(\Delta_{2s}^l)\cap H^{j}_{\delta}(S_{2s})$ such that $\theta_{A\dots F}=\phi_{A\dots F} + \psi_{A\dots F}$ satisfies
$$
\Vert \theta_{A\dots F} \Vert_{j,\delta} \leq C \Vert\Delta^{l}_{2s}\theta_{A\dots F}\Vert_{j-2s,\delta-2s}.
$$
\end{proposition}
\begin{remark}
\begin{enumerate}
\item In this paper, the set $\mathbb{N}$ is the set of positive integers and $\mathbb{N}_0$ the set of non-negative integers.
\item In the range of weights $-1\leq \delta\leq  0$, the operator $\Delta$ is bijective.
\item The inequality comes from the closed range property (see \cite[Theorem~5.2]{Kat80}). Since this infimum corresponds to the distance to the kernel $\ker(\Delta_{2s}^l)$, which is closed, the infimum is in fact attained.
\item We recall here that 
the co-kernel of $\Delta^l_{2s}$ in
$H^{j-2l}_{\delta-2l}(S_{2s})$ is $L^2$-orthogonal to the kernel of
$\Delta^{l}_{2s}$ in the dual space $H^{-j+2l}_{-3-\delta+2l}(S_{2s})$.
\item The dimension of the spaces can be computed explicitly -- see for instance \cite{Lockhart:1984kt}. However, we will not make explicit use of this.
\end{enumerate}
\end{remark}

\subsection{Fundamental operators} \label{sec:fundop} 

\begin{definition} \label{def:fundop}
Let 
$\phi_{A_1\dots A_k}\in \mathcal{S}_k$, that is $\phi_{A_1\dots A_k}=\phi_{(A_1\dots A_k)}$.
Let $D_{AB}$ be the intrinsic Levi-Civita connection.  Define the operators
$\sdiv_k: \mathcal{S}_k\rightarrow \mathcal{S}_{k-2}$, 
$\scurl_k: \mathcal{S}_k\rightarrow \mathcal{S}_{k}$ and 
$\stwist_k: \mathcal{S}_k\rightarrow \mathcal{S}_{k+2}$ via
\begin{align*}
(\sdiv_k\phi)_{A_1\dots A_{k-2}}\equiv{}&D^{A_{k-1}A_{k}}\phi_{A_1\dots A_k},\\
(\scurl_k\phi)_{A_1\dots A_{k}}\equiv{}&D_{(A_1}{}^{B}\phi_{A_2\dots A_{k})B},\\
(\stwist_k\phi)_{A_1\dots A_{k+2}}\equiv{}&D_{(A_1 A_2}\phi_{A_3\dots A_{k+2})}.
\end{align*}
These operators will be called \emph{divergence}, \emph{curl} and \emph{twistor operator} respectively.
\end{definition}
We suppress the indices of $\phi$ in the left hand sides. The label $k$ indicates its valence.
The importance of these operators comes from the following irreducible decomposition which is valid for any  $k\geq 1$,
\begin{align*}
D_{A_1 A_2}\phi_{A_3\dots A_{k+2}}={}&(\stwist_k\phi)_{A_1\dots A_{k+2}}
-\tfrac{k}{k+2} \epsilon_{A_1(A_3}(\scurl_k\phi)_{A_4\dots A_{k+2})A_2}\nonumber\\
&-\tfrac{k}{k+2} \epsilon_{A_2(A_3}(\scurl_k\phi)_{A_4\dots A_{k+2})A_1}
+\tfrac{1-k}{1+k} \epsilon_{A_1(A_3}(\sdiv_k\phi)_{A_4\dots A_{k+1}}\epsilon_{A_{k+2})A_2}.
\end{align*}
This irreducible decomposition  follows from \cite[Proposition~3.3.54]{Penrose:1986fk}. Contraction with the spin metric $\epsilon^{AB}$ and partial expansion of the symmetries give the coefficients.

We consider now the symbols of these operators
\begin{equation*}
\sigma(\scurl_k) : T^\star M  \to  L(S_k, S_k),\qquad
\sigma(\sdiv_k) : T^\star M  \to  L(S_k, S_{k-2}),\qquad
\sigma(\stwist_k) : T^\star M  \to  L(S_k, S_{k+2}), 
\end{equation*}
where $L(S_k, S_j)$ is the space of bundle maps from $S_k$ into $S_j$. 
When applied to a 1-form $\xi_{AB}$, one denotes the symbols $\sigma_{\xi}$.
\begin{lemma}\label{symbolcurlreal}
When applied to a 1-form $\xi_{AB}$, the symbol $\sigma_{\xi}(\scurl_k) : S_k \rightarrow S_k$ of $\scurl_k$ is Hermitian and has only real eigenvalues. 
\end{lemma}
\begin{proof}
By definition we have
\begin{equation*}
(\sigma_\xi(\scurl_k)\phi)_{A_1\dots A_{k}}\equiv \xi_{(A_1}{}^{B}\phi_{A_2\dots A_{k})B},
\end{equation*}
where $\xi_{AB}$, is real, \emph{i.e.} $\widehat \xi_{AB}=-\xi_{AB}$. For arbitrary $\eta_{A_1\dots A_{k}}$ and $\zeta_{A_1\dots A_{k}}$, we have
\begin{align*}
\langle (\sigma_\xi(\scurl_k)\eta)_{A_1\dots A_{k}},
\zeta_{A_1\dots A_k}\rangle
={}& \xi_{A_1}{}^{B}\eta_{A_2\dots A_{k}B}\hat\zeta^{A_1\dots A_k}
= \eta_{A_2\dots A_{k}B}\widehat \xi^{C_1B}\hat\zeta_{C_1}{}^{A_2\dots A_k}\nonumber\\
={}&\langle \eta_{A_1\dots A_k}, (\sigma_\xi(\scurl_k)\zeta)_{A_1\dots A_{k}}\rangle.
\end{align*}
Hence, the symbol $\sigma_\xi(\scurl_k)$ is for each point Hermitian and, by the spectral theorem, has only real eigenvalues.
\end{proof}

The operators $\sdiv_k$, $\scurl_k$ and $\stwist_k$ are special cases of Stein-Weiss operators. We refer to \cite{Branson1997334} and references therein for general properties of this class of operators. 

We will now consider some operators of general order, 
that will play an important role in this paper.
The most important second order operator is clearly the Laplacian $\Delta\equiv D_{AB}D^{AB}$. Note that here we are using a \emph{negative} definite metric on $\mathbb{R}^3$. When the Laplacian acts on a spinor field $\phi_{A\dots F}$ in $\mathcal{S}_k$, we will often use the notation $(\Delta_k\phi)_{A\dots F}$, where $k$ indicates the valence of $\phi_{A\dots F}$.
\begin{definition}\label{def:mathcalg}
Define the order $k-1$ operators $\mathcal{G}_k:\mathcal{S}_k\rightarrow \mathcal{S}_{k}$ as
\begin{align*}
(\mathcal{G}_k\phi)_{A_1\dots A_k} &\equiv \sum_{n=0}^{\mathclap{\left\lfloor\tfrac{k-1}{2}\right\rfloor}}\binom{k}{2n+1}(-2)^{-n} \underbrace{D_{(A_1}{}^{B_1}\cdots D_{A_{k-2n-1}}{}^{B_{k-2n-1}}}_{k-2n-1} (\Delta^n_{k}\phi)_{A_{k-2n}\dots A_k)B_1\dots B_{k-2n-1}}.
\end{align*}
\end{definition}
The first operators are
\begin{align*}
(\mathcal{G}_1\phi)_{A} &\equiv \phi_{A},\\
(\mathcal{G}_2\phi)_{AB} &\equiv 2 D_{(A}{}^{C}\phi_{B)C}
=2(\scurl_2\phi)_{AB},\\
(\mathcal{G}_3\phi)_{ABC} &\equiv 3 D_{(A}{}^{D}D_{B}{}^{F}\phi_{C)DF}
-  \tfrac{1}{2} (\Delta_{3}\phi)_{ABC} \nonumber\\
&=\tfrac{1}{3}(\stwist_1\sdiv_3\phi)_{ABCD}+4(\scurl_3\scurl_3\phi)_{ABC},\\
(\mathcal{G}_4\phi)_{ABCD} &\equiv4 D_{(A}{}^{F}D_{B}{}^{H}D_{C}{}^{L}\phi_{D)FHL}
- 2 D_{(A}{}^{F}(\Delta_{4}\phi)_{BCD)F}\nonumber\\ &=2(\stwist_2\sdiv_4\scurl_4\phi)_{ABCD}+8(\scurl_4\scurl_4\scurl_4\phi)_{ABCD}.
\end{align*}
These operators  appear naturally in Proposition~\ref{prop:splittingpotential} below. The
most important properties of these operators are 
\begin{equation}\label{eq:divGproperty}
\sdiv_k\mathcal{G}_k=0\quad \text{ and }\quad \mathcal{G}_k\stwist_{k-2}=0,
\end{equation}
which is valid for any $k\geq  2$. The main
idea to prove this is to use that $\sdiv_k\mathcal{G}_k$ and
$\mathcal{G}_k\stwist_{k-2}$ contains derivatives of the kind
$D_{A}{}^{C}D_{BC}=\frac{1}{2}\epsilon_{AB}\Delta$. For a complete proof see
Proposition~\ref{divGproperty}. The operators $\mathcal{G}_k$ also commute
with $\scurl_k$; this is proven in Proposition~\ref{Gcurlexpansion}. 
To connect with the standard elliptic theory, we express appropriate powers of the Laplacian in terms of the operators $\mathcal{G}_k$ as
\begin{subequations}
\begin{align}
(\Delta^{k}_{2k}\phi)_{A_1\dots A_{2k}}={}&(\stwist_{2k-2}\mathcal{F}_{2k-2}\sdiv_{2k}\phi)_{A_1\dots A_{2k}}-(-2)^{1-k}(\mathcal{G}_{2k}\scurl_{2k}\phi)_{A_1\dots A_{2k}},\label{LaplacianAsGeven}\\
 (\Delta^{k}_{2k+1}\phi)_{A_{1}\dots A_{2k+1}}={}&
(\stwist_{2k-1}\mathcal{F}_{2k-1}\sdiv_{2k+1}\phi)_{A_1\dots A_{2k+1}}
+(-2)^{-k}(\mathcal{G}_{2k+1}\phi)_{A_1\dots A_{2k+1}} \label{LaplacianAsGodd},
\end{align}
\end{subequations}
where the operators $\mathcal{F}_{2s}$ for $s\in \tfrac{1}{2}\mathbb{N}_0$ are defined via
\begin{align*}
(\mathcal{F}_{2s}\phi)_{A_1\dots A_{2s}}={}&
2^{-2s}\sum_{n=0}^{\lfloor s\rfloor}\sum_{m=0}^{\lfloor s\rfloor-n}\binom{2s+2}{2n+2m+2}(-2)^{n}\nonumber\\
&\quad\times \underbrace{D_{(A_1}{}^{B_1}\cdots D_{A_{2n}}{}^{B_{2n}}}_{2n} (\Delta^{\lfloor s\rfloor-n}_{2s}\phi)_{A_{2n+1}\dots A_{2s})B_1\dots B_{2n}}.
\end{align*}

The first operators are
\begin{align*}
(\mathcal{F}_0\phi)={}&\phi,\nonumber\\
(\mathcal{F}_1\phi)_{A}\equiv{}& \tfrac{3}{2}\phi_A,\\
(\mathcal{F}_2\phi)_{AB}={}& \tfrac{7}{4} (\Delta_{2} \phi)_{AB}
-\tfrac{1}{2}D_{(A}{}^{C}D_{B)}{}^{D}\phi_{CD}.
\end{align*}
See Lemma~\ref{lapspinor} in the appendix for the proof of \eqref{LaplacianAsGeven} and \eqref{LaplacianAsGodd}.

The operator $\scurl_2$ is the spinor equivalent to the operator $\ast \ud$ acting on 1-forms. The tensor equivalent of the operator $\mathcal{G}_4$ is the linearized Cotton-York tensor acting on symmetric trace-free 2-tensors. In the following section, a more detailed description of these relations is given.

\section{Integrability properties of spinor fields}\label{sec:Integrability}
A crucial part of this work relies on integrability properties for
spinors, that is to say proving that a spinor belongs to the image of a certain differential operator. In the case of spin-1, the operator under consideration is the curl operator; its integrability properties are well known
since it is described by the de Rham complex. For
spin-2, one has to resort to a generalization of the de Rham theory to trace free 2-tensors, which happened to have been studied in the context of
conformal deformations of the flat metric by Gasqui and Goldschmidt
\cite{Gasqui:1984vu}, whose 
results were 
extended by Beig \cite{Beig:1997wo}. 
On $\R^3$, we consider a generalization
of these elliptic complexes for arbitrary spin.

We present here the general picture of this integrability result for smooth
spinors. It is well known that for 1-forms the integrability conditions are
given by the following elliptic complex
\begin{equation*}
 C^\infty(\R^3,\R)\stackrel{\ud}{\longrightarrow} \Lambda^1\stackrel{\star \ud}{\longrightarrow} \Lambda^1  \stackrel{\coder}{\longrightarrow}C^\infty(\R^3,\R),
\end{equation*}
whose spinorial equivalent is
\begin{equation}\label{exactseq1}
 \mathcal{S}_{0}\stackrel{\stwist_0}{\longrightarrow} \mathcal{S}_{2} \stackrel{\scurl_2}{\longrightarrow} \mathcal{S}_{2}\stackrel{\sdiv_2}{\longrightarrow} \mathcal{S}_{0}.
\end{equation}
Gasqui and Goldschmidt were interested in the conformal deformation of a metric on a 3-manifold $M$. A deformation $g_t$  of a metric $g_0$ is said to be conformally rigid if there exist a family of diffeomorphisms $\phi_t^\star$ and of functions $u_t$ such that
$$
\phi^\star_t g_0 = e^{u_t}g_t.
$$
The infinitesimal equation corresponding to this deformation is given by the conformal Killing equation
\begin{equation}\label{confkill}
\mathcal{L}_X g_0-\frac13\text{Tr}_{g_0}(\mathcal{L}_X g_0)g_0=h
\end{equation}
where $X$ is a vector field on $M$ and $h$ is a trace free 2-tensor. The spinor equivalent of this equation is given by
$$
2D_{(AB}X_{CD)} = h_{ABCD}.
$$ 
Solving \eqref{confkill} requires that the 2-tensor $h$ satisfies the constraint equation. This is stated in \cite[Theorem 6.1, (2.24)]{Gasqui:1984vu} and in \cite{Beig:1997wo}.
\begin{theorem}[Gasqui-Goldschmidt]\label{thm:integrability2}
If $(M,g)$ is a conformally flat 3-dimensional manifold, then the following is an elliptic complex
$$
\Lambda^1(M) \stackrel{L}{\longrightarrow}  S_0^2(M, g) \stackrel{\mathcal{R}}{\longrightarrow} S_0^2(M, g) \stackrel{\diverg} {\longrightarrow}\Lambda^1(M), 
$$ 
where $\Lambda^1(M)$ is the space of 1-forms over $M$, $S_0^2(M, g)$ is the space of symmetric trace free 2-tensors and
\begin{eqnarray*}
(LW)_{ab}&=&D_{(a}W_{b)}-\frac{1}{3}g_{ab}D^cW_c\\
(\diverg t)_a&=&2g^{bc}D_ct_{ab}
\end{eqnarray*}
and \begin{equation*}
\begin{array}{c}
\mathcal{R}(\psi)_{ab} = \eps^{cd}{}_a D_{[c}\sigma_{d]b} \text{ where}\\
\sigma_{ab}=D_{(a} D^c \psi_{b)c} - \frac{1}{2} \Delta \psi_{ab} - \frac{1}{4} g_{ab}D^c D^d \psi_{cd}.
\end{array}
\end{equation*}
\end{theorem}
\begin{remark}
\begin{enumerate}
\item A consequence of Theorem \ref{thm:integrability2} is that equation \eqref{confkill} is integrable provided that 
$$
\mathcal{R}(h)_{ab}=0.
$$
\item In terms of spinors, the operator $\mathcal{R}_{ab}$ reads
$$
\mathcal{R}_{ab} = \mathcal{R}_{ABCD} = -\frac{i}{2\sqrt{2}} (\mathcal{G}_4)_{ABCD}.
$$
\end{enumerate}
\end{remark}

The spinorial equivalent of this sequence is the following elliptic complex
\begin{equation}\label{exactseq2}
\mathcal{S}_2 \stackrel{\stwist_2}{\longrightarrow}  \mathcal{S}_4 \stackrel{\mathcal{G}_4}{\longrightarrow} \mathcal{S}_4 \stackrel{\sdiv_4} {\longrightarrow}\mathcal{S}_2.
\end{equation}

We now state, using the fundamental operators $\stwist_{2s-2}$, $\mathcal{G}_{2s}$ and $\sdiv_{2s}$, a generalization of the elliptic complexes \eqref{exactseq1} and \eqref{exactseq2} for arbitrary spin.
\begin{lemma}\label{lem:ellipticcomplex}
The sequence
\begin{equation*} 
\mathcal{S}_{2s-2}\stackrel{\stwist_{2s-2}}{\longrightarrow} \mathcal{S}_{2s} \stackrel{\mathcal{G}_{2s}}{\longrightarrow} \mathcal{S}_{2s}\stackrel{\sdiv_{2s}}{\longrightarrow} \mathcal{S}_{2s-2},
\end{equation*}
is an elliptic complex.
\end{lemma}
\begin{proof} 
In view of \eqref{eq:divGproperty},
 the sequence is a differential complex. It is therefore enough to 
check that the symbol sequence is exact, \emph{i.e.} for a non-zero $\xi$  in $T^\star M$ 
$$
S_{2s-2} \stackrel{\sigma_\xi(\stwist_{2s-2})}{\longrightarrow} S_{2s} \stackrel{\sigma_\xi(\mathcal{G}_{2s})}{\longrightarrow} 
S_{2s}\stackrel{\sigma_\xi(\sdiv_{2s})}{\longrightarrow} S_{2s-2}.
$$
This follows from the vanishing properties \eqref{eq:divGproperty}
and the expression of powers of the Laplacian in terms of these operators, \emph{i.e.} \eqref{LaplacianAsGeven} and \eqref{LaplacianAsGodd}.  

Let $\xi$  be a fixed non-zero element of $T^\star M $. We first notice that 
\eqref{eq:divGproperty} implies
$$
\text{im}(\sigma_\xi(\mathcal{G}_{2s})) \subset \ker (\sigma_\xi(\sdiv_{2s}))
\quad \text{ and } \quad
\text{im}(\sigma_\xi(\stwist_{{2s}-2}))\subset \ker(\sigma_\xi(\mathcal{G}_{2s})).
$$
We then notice that the symbol of the Laplacian $\Delta^k_{2s}$ is an invertible symbol which is in the center of the algebra of symbols since its expression is
$$
\sigma_\xi(\Delta_{2s}^k) = |\xi_{AB}|^{2k} I.
$$

Furthermore, using the relations stated in Lemma~\ref{lapspinor}, we have\begin{align}
(\Delta^{s}_{2s}\phi)_{A_1\dots A_{2s}}&=(\stwist_{2s-2}\mathcal{F}_{2s-2}\sdiv_{2s}\phi)_{A_1\dots A_{2s}}-(-2)^{1-s}(\mathcal{G}_{2s}\scurl_{2s}\phi)_{A_1\dots A_{2s}}\text{ for } s\in \mathbb{N}_0,\label{lapeven}\\
 (\Delta^{s-1/2}_{2s}\phi)_{A_{1}\dots A_{2s+1}}&=
(\stwist_{2s-2}\mathcal{F}_{2s-2}\sdiv_{2s}\phi)_{A_1\dots A_{2s}}
+(-2)^{-1/2-s}(\mathcal{G}_{2s}\phi)_{A_1\dots A_{2s}} \text{ for } s\in \tfrac12+ \mathbb{N}_0\nonumber.
\end{align}

Assume now that the spin is an integer. The proof in the case when the spin is a half integer is left to the reader (the proof is almost identical). 
Let $Y$ be an element of $\ker(\sigma_\xi(\sdiv_{2s}))$. Using formula~\eqref{lapeven}, we get
\begin{eqnarray*}
Y &=& \sigma_\xi(\Delta_{2s}^s)^{-1}\sigma_\xi(\Delta_{2s}^s) Y\\
&=&\sigma_\xi(\Delta_{2s}^s)^{-1}\left(\sigma_\xi(\stwist_{2s-2})\sigma_\xi(\mathcal{F}_{2s-2})\sigma_\xi(\sdiv_{2s})-(-2)^{1-s}\sigma_\xi(\mathcal{G}_{2s})\sigma(\scurl_{2s})\right)Y\\
&=&-(-2)^{1-s}\sigma_\xi(\Delta_{2s}^s)^{-1}\sigma_\xi(\mathcal{G}_{2s})\sigma_\xi(\scurl_{2s})Y.
 \end{eqnarray*}
Since the symbol of the Laplacian commutes with all other symbols, we consequently get
$$
Y = -(-2)^{1-s}\sigma_\xi(\mathcal{G}_{2s})\sigma_\xi(\Delta_{2s}^s)^{-1}\sigma_\xi(\scurl_{2s}) Y,
$$
that is to say that $Y$ belongs to the image of $\sigma_\xi(\mathcal{G}_{2s})$.
If we now assume that $Y$ is in $\ker(\sigma_\xi(\mathcal{G}_{2s}))$. Using formula~\eqref{lapeven}, we get
\begin{eqnarray*}
Y &=& \sigma_\xi(\Delta_{2s}^s)^{-1}\sigma_\xi(\Delta_{2s}^s) Y\\
&=&\sigma_\xi(\Delta_{2s}^s)^{-1}\left(\sigma_\xi(\stwist_{2s-2})\sigma_\xi(\mathcal{F}_{2s-2})\sigma_\xi(\sdiv_{2s})-(-2)^{1-s}\sigma_\xi(\mathcal{G}_{2s})\sigma_\xi(\scurl_{2s})\right)Y.
 \end{eqnarray*}
Since $\mathcal{G}_{2s}$ and $\scurl_{2s}$ commute (Lemma~\ref{Gcurlexpansion}) and since the symbol of the Laplacian commutes with all other symbols, we get
$$
Y=\sigma_\xi(\stwist_{2s-2})\sigma_\xi(\Delta_{2s}^s)^{-1}\sigma_\xi(\mathcal{F}_{2s-2})\sigma_\xi(\sdiv_{2s}),
$$
that is to say that $Y$ belongs to the image of $\sigma_\xi(\stwist_{2s-2})$.
\end{proof}

Using the ellipticity of the sequence, it is finally possible to prove the existence of solutions of equations involving $\mathcal{G}_{2s}$ and $\stwist_{2s}$. This theorem is a direct consequence of \cite[Theorem~1.4]{Anonymous:1996ce}.
\begin{proposition}\label{prop:integrability3} For $x$ in $\R^3$, there exists an open neighborhood $U$ of $x$ such that the sequence
\begin{equation*}
\mathfrak{A}(U,S_{2s-2})\stackrel{\stwist_{2s-2}}{\longrightarrow} \mathfrak{A}(U,S_{2s}) \stackrel{\mathcal{G}_{2s}}{\longrightarrow} \mathfrak{A}(U,S_{2s})\stackrel{\sdiv_{2s}}{\longrightarrow} \mathfrak{A}(U,S_{2s-2}),
\end{equation*}
is exact, where $\mathfrak{A}(U,E)$ denotes the space of real analytic sections of $E$.
\end{proposition}
\begin{remark}\label{integrability4} We in fact only need this result in the context of polynomials. The problem will be to solve, for any real number $\delta$, the equations
$$
\stwist_{2s-2} \phi = \psi, \quad  \text{ when }\psi\in \mathcal{P}^{<\delta}(S_{2s}) ,
$$
and 
$$
\mathcal{G}_{2s} \xi = \zeta, \quad  \text{ when }\zeta\in \mathcal{P}^{<\delta}(S_{2s}).
$$
Proposition~\ref{prop:integrability3} ensures the local existence of solutions to these equations provided that
$$
\mathcal{G}_{2s}\psi =0 \text{ and }\sdiv_{2s} \zeta = 0.
$$
By integration, these solutions are necessarily polynomials.
\end{remark}
\begin{proof} The proof of Proposition~\ref{prop:integrability3} is a direct
  consequence of the fact that the fundamental operators $\stwist_{2s-2}$,
  $\mathcal{G}_{2s}$ and $\sdiv_{2s}$ are operators with constant
  coefficients, which consist only of higher order homogeneous terms. As a
  consequence, these operators are all \emph{sufficiently regular}
in the terminology of \cite{Anonymous:1996ce} 
  (since they have constant coefficients, cf. 
  \cite[Remark~1.16]{Anonymous:1996ce}) 
and \emph{formally integrable}
  (since they have only homogeneous terms of the highest possible order,
cf. \cite[Remark~1.21]{Anonymous:1996ce}). 
Proposition~\ref{prop:integrability3} is then
  a direct consequence of \cite[Theorem~1.4]{Anonymous:1996ce}. 
\end{proof}

\section{Construction of initial data for the potential} \label{Section:ConstructPotentials}

As seen in the introduction, one of the key points when constructing the initial data for the potential for the massless free field is the ability to solve, at the level of the initial data, the equation
$$
\varphi_{A\dots F} = (\mathcal{G}_{2s} \zeta)_{A\dots F}. 
$$ 
This requires a partial generalization of Proposition \ref{prop:integrability3} to weighted Sobolev spaces.
The main difficulty in the construction of $\zeta_{A\dots F}$ is to obtain the estimate
\begin{equation*}
\Vert\zeta_{A\dots F}\Vert_{j+2s-1,\delta+2s-1}\leq C \Vert\varphi_{A\dots F}\Vert_{j,\delta}.
\end{equation*}
This inequality is similar to those in standard elliptic theory. The main idea of this section is to construct a solution using the relations between the operator $\mathcal{G}_{2s}$ and powers of the Laplacian $\Delta$ as stated in Equations \eqref{LaplacianAsGeven} and  \eqref{LaplacianAsGodd}. The strategy of the proof is as follows:
\begin{enumerate}
\item using the elliptic properties of the Laplacian and its powers, a preimage $\theta_{A\dots F}$
 of $\varphi_{A\dots F}$ is constructed;
 \item using Equations \eqref{LaplacianAsGeven} and \eqref{LaplacianAsGodd}, the constraint equation satisfied by the initial datum, and the differential complex stated in Lemma \eqref{lem:ellipticcomplex}, we prove that the only non-vanishing term of Equations \eqref{LaplacianAsGeven} and \eqref{LaplacianAsGodd} is the one containing $\mathcal{G}_{2s}$.
\end{enumerate}
However, this schematic procedure works only for a certain range of weights. Outside this range, a more thorough discussion has to be performed. One of the key facts which is used is that the polynomial nature of the elements of the kernel of the Laplacian (and its powers) makes it possible to use Proposition \ref{prop:integrability3} (and Remark \ref{integrability4}).

This section deals strictly with the initial data both for the field and the potential. The corresponding Cauchy problems for higher spin fields and for the potential are described in Section \ref{sec:mainres}.  More precisely, the details regarding the relation between the Hertz potential and the field are given in Section \ref{sec:spacespinorsplit}, which is devoted to the 3+1 splitting of the potential equation. Furthermore, the representation Theorem \ref{th:repthm} for a massless field in terms of a Hertz potential, based on the uniqueness of solutions of the Cauchy problem, is given in Section \ref{sec:representation}. 

For $s=1/2$, we immediately get the desired solution by setting $\zeta_A=\varphi_A$. For higher spin, a more careful analysis is required. To simplify the presentation, the spin-1 case is discussed first in detail, followed by the general spin case.

The following lemma is a technical result describing the orthogonality properties of the range of $\stwist$ and the kernel of $\sdiv$.
\begin{lemma}\label{orthogonalitylemma}
Assume that $\varphi_{A\dots F}\in H^1_\delta(S_{2s})$ 
satisfies the constraint $(\sdiv_{2s}\varphi)_{C\dots F}=0$, and $\eta_{C\dots F}\in H^1_{-2-\delta}(S_{2s-2})$.
Then $\varphi_{A\dots F}$ is $L^2$ orthogonal to $(\stwist_{2s-2}\eta)_{A\dots F}$.
\end{lemma}
\begin{proof}
Let $\{\varphi^i_{A\dots F}\}_{i=0}^\infty\subset C_0^\infty(S_{2s})$ such that $\Vert \varphi^i_{A\dots F}-\varphi_{A\dots F}\Vert_{1,\delta}\rightarrow 0$ as $i\rightarrow \infty$. 
An integration by parts and Remark \ref{rem:DReal} give
\begin{equation*}
\int_{\R^3}\varphi^i_{A\dots F} \widehat{D^{(AB}\eta^{C\dots F)}}\ud\mu_{\R^3}
= \int_{\R^3}D^{AB}\varphi^i_{A\dots F}\widehat\eta^{C\dots F} \ud \mu_{\R^3},
\end{equation*}
that is
\begin{equation*}
\langle \varphi^i_{A\dots F}, (\stwist_{2s-2}\eta)_{A\dots F}\rangle_{L^2} =
\langle (\sdiv_{2s}\varphi^i)_{C\dots F},\eta_{C\dots F}\rangle_{L^2}.
\end{equation*}
Taking the limit as $i\rightarrow \infty$ on both sides gives 
$\langle \varphi_{A\dots F}, (\stwist_{2s-2}\eta)_{A\dots F}\rangle_{L^2}=0$.
\end{proof}

\begin{definition}\label{def:Esdelta}
Let $\delta$ be in $\mathbb{R}\setminus\mathbb{Z}$ and $s\in \tfrac{1}{2}\mathbb{N}$. 
Furthermore, let
\begin{equation*}
\mathbb{F}_{s,\delta}\equiv\ker \Delta^{\lfloor s\rfloor}_{2s}\cap L^2_{-3-\delta}(S_{2s}).
\end{equation*}
Define the space $\mathbb{E}_{s,\delta}$ to be the $L^2_{-3-\delta}$-orthogonal complement of $\mathbb{F}_{s,\delta}\cap \ker\mathcal{G}_{2s}$ in $\mathbb{F}_{s,\delta}$.
\end{definition}
\begin{remark}\label{rem:Etrivial}
If $\delta\geq -2s-2$, the space $\mathbb{E}_{s,\delta}$ is trivial. This follows from the fact that  $\mathbb{F}_{s,\delta}\subset \mathcal{P}^{<-3-\delta}(S_{2s})$ and that $\mathcal{G}_{2s}$ is a homogeneous order $2s-1$ operator. With $-3-\delta \leq 2s-1$ or equivalently $\delta\geq -2s-2$ we have $\mathbb{F}_{s,\delta}\subset \mathcal{P}^{<-3-\delta}(S_{2s})\subset \ker\mathcal{G}_{2s}$ and consequently $\mathbb{E}_{s,\delta}=\{0\}$.
\end{remark}

Before we prove existence of preimages under $\mathcal{G}_{2s}$, we prove a technical lemma that allows us to reduce the problem of finding a preimage of $\ker \sdiv_{2s} \cap H^j_\delta(S_{2s})$ under $\mathcal{G}_{2s}$ to that of finding a preimage of $\ker \sdiv_{2s} \cap H^j_\delta(S_{2s})$ orthogonal to $\mathbb{E}_{s,\delta}$.
\begin{lemma}\label{lemma:orthogonalpreimage}
Let $\delta$ be in $\mathbb{R}\setminus\mathbb{Z}$, $j>0$ integer, $\varphi_{A\dots F}$ in $H^j_\delta(S_{2s})$. If $\delta < -2s-2$ we can find a $\tilde\zeta_{A\dots F}\in H^{j+2s-1}_{\delta+2s-1}(S_{2s})$ and a constant $C$ depending only on $s$, $j$ and $\delta$ such that $\varphi_{A\dots F}+(\mathcal{G}_{2s}\tilde\zeta)_{A\dots F}$ is orthogonal to $\mathbb{E}_{s,\delta}$ and
\begin{align*}
\Vert\tilde\zeta_{A\dots F}\Vert_{j+2s-1,\delta+2s-1}&\leq C \Vert\varphi_{A\dots F}\Vert_{j,\delta}.
\end{align*}
\end{lemma}
\begin{proof}
Let $\{\mu^i_{A\dots F}\}_{i}$ be an $L^2_{-\delta-2s-2}(S_{2s})$ orthonormal basis for the finite dimensional space $\mathcal{G}_{2s}(\mathbb{E}_{s,\delta})\subset H^{j+2s-2}_{-\delta-2s-2}(S_{2s})$. 
Due to the splitting $\mathbb{F}_{s,\delta}=\mathbb{E}_{s,\delta}\oplus (\mathbb{F}_{s,\delta}\cap \ker\mathcal{G}_{2s})$, we have for each $\mu^i_{A\dots F}$ a unique $\xi^i_{A\dots F}\in \mathbb{E}_{s,\delta}$ such that $(\mathcal{G}_{2s}\xi^i)_{A\dots F}=\mu^i_{A\dots F}$, and $\{\xi^i_{A\dots F}\}_{i}$ span $\mathbb{E}_{s,\delta}$.
Let $\eta^i_{A\dots F}\equiv\lAngle r\rAngle^{2\delta+4s+1}\mu^i_{A\dots F}$.
By \cite[Lemma 5.2]{Cantor:1981bs} we have that $\eta^i_{A\dots F}\in H^{j+2s-1}_{\delta+2s-1}(S_{2s})$.
We can now approximate $\xi^i_{A\dots F}\in H^{j+4s-3}_{-\delta-3}$ by choosing a sequence 
$\{ \xi^{i,k}_{A\dots F} \}_{k \in \mathbb{N}} \subset C^\infty_0(S_{2s})$ such that $\Vert\xi^i_{A\dots F}-\xi^{i,k}_{A\dots F}\Vert_{j+4s-3,-\delta-3}\rightarrow 0$ for $k \to \infty$. 
Repeated integration by parts gives
\begin{align*}
\lAngle(\mathcal{G}_{2s}\eta^i)_{A\dots F},\xi^{j,k}_{A\dots F}\rAngle_{L^2}
={}&-\lAngle\eta^i_{A\dots F},(\mathcal{G}_{2s}\xi^{j,k})_{A\dots F}\rAngle_{L^2},
\end{align*}
where the boundary terms vanish due to the compact support of $\xi^{i,k}_{A\dots F}$.
In the limit $k\rightarrow \infty$, this gives, making use of the definition of the weighted spaces,
\begin{align*}
\lAngle(\mathcal{G}_{2s}\eta^i)_{A\dots F},\xi^j_{A\dots F}\rAngle_{L^2}
={}&-\lAngle\eta^i_{A\dots F},(\mathcal{G}_{2s}\xi^j)_{A\dots F}\rAngle_{L^2} \\
={}& -\lAngle\mu^i_{A\dots F},\mu^j_{A\dots F}\rAngle_{L^2_{-\delta-2s-2}}=-\delta^{ij}.
\end{align*}
Now, let 
\begin{align*}
\tilde \zeta_{A\dots F}\equiv{}&\sum_{i}\lAngle \varphi_{A\dots F},\xi^i_{A\dots F}\rAngle_{L^2}\eta^i_{A\dots F}.
\end{align*}
This gives that $\lAngle\varphi_{A\dots F}+(\mathcal{G}_{2s}\tilde\zeta)_{A\dots F},\xi^i_{A\dots F}\rAngle_{L^2}=0$ for all $i$. Hence, $\varphi_{A\dots F}+(\mathcal{G}_{2s}\tilde\zeta)_{A\dots F}$ is orthogonal to $\mathbb{E}_{s,\delta}$.

We also get the estimates
\begin{align}
\Vert\tilde \zeta_{A\dots F}\Vert_{j+2s-1,\delta+2s-1}\leq{}&\sum_{i}|\lAngle \varphi_{A\dots F},\xi^i_{A\dots F}\rAngle_{L^2}|\thinspace\Vert\eta^i_{A\dots F}\Vert_{j+2s-1,\delta+2s-1}\nonumber\\
\leq{}&\sum_{i}\Vert\varphi_{A\dots F}\Vert_{L^2_\delta}\Vert\xi^i_{A\dots F}\Vert_{L^2_{-3-\delta}}\Vert\eta^i_{A\dots F}\Vert_{j+2s-1,\delta+2s-1}\nonumber\\
\leq{}&C \Vert\varphi_{A\dots F}\Vert_{j,\delta},
\end{align}
where $C$ only depends on $s$, $j$ and $\delta$.
\end{proof}

\subsection{The spin-1 case}
To show how we can solve the equation $\varphi_{A\dots F}   
=(\mathcal{G}_{2s}\zeta)_{A\dots F}$, we begin with the spin-1 case to illustrate the idea. We then use the same ideas for general spin in Proposition~\ref{proprepspingeneral}. 
\begin{proposition}\label{proprepspin1}
Let $\delta$ be in  $\mathbb{R}\setminus\mathbb{Z}$, $j$ a positive integer, $\varphi_{AB}$ in $H^j_\delta(S_2)$ such that $D^{AB}\varphi_{AB}=0$. 
Then there exist a spinor field $\zeta_{AB}\in H^{j+1}_{\delta+1}(S_2)$ and a constant $C$ depending only on $\delta$ and $j$ such that
\begin{align*}
\varphi_{AB}&=(\mathcal{G}_2\zeta)_{AB},\\
\Vert\zeta_{AB}\Vert_{j+1,\delta+1}&\leq C \Vert\varphi_{AB}\Vert_{j,\delta}.
\end{align*}
\end{proposition}

\begin{proof}
By Remark~\ref{rem:Etrivial} we see that $\varphi_{AB}$ is automatically orthogonal to $\mathbb{E}_{1,\delta}$ if $\delta\geq -4$. If $\delta< -4$, we can use Lemma~\ref{lemma:orthogonalpreimage} to 
construct $\tilde\zeta_{AB}\in H^{j+1}_{\delta+1}(S_{2})$ such that $\varphi_{AB}+(\mathcal{G}_{2}\tilde\zeta)_{AB}\in \ker\sdiv_2\cap H^j_\delta(S_2)$ is orthogonal to $\mathbb{E}_{1,\delta}$. The estimates in Lemma~\ref{lemma:orthogonalpreimage} are of the right type, so if we can prove the proposition for $\varphi_{AB}+(\mathcal{G}_{2}\tilde\zeta)_{AB}$ instead of $\varphi_{AB}$, we are done. We can therefore in the rest of the proof without loss of generality assume that $\varphi_{AB}$ is orthogonal to $\mathbb{E}_{1,\delta}$. From now on, to make the link with the language of forms clear, we replace the operator $\mathcal G_2 = 2\scurl_2$ with the curl operator $\scurl_2$.

Let $\theta_{AB}$ be in $\mathbb{F}_{1,\delta}\cap \ker\scurl_2$. The field $\theta_{AB}$ is then in $\mathcal{P}^{<-3-\delta}(S_2)$ and therefore real analytic. Furthermore it is curl-free (\emph{i.e.} in $\ker \scurl_2$). Using Proposition \ref{prop:integrability3}, the sequence 
$$
\mathcal{P}^{<-4-\delta}(S_0)\stackrel{\stwist_0}{\longrightarrow} \mathcal{P}^{<-3-\delta}(S_2)  \stackrel{\scurl_2}{\longrightarrow} \mathcal{P}^{<-2-\delta}(S_2),
$$
is exact and, therefore, $\theta_{AB}$ can be written as a gradient $D_{AB}\eta = (\stwist_0\eta)_{AB} =\theta_{AB}$, where $\eta\in \mathcal{P}^{<-2-\delta}(S_2)\subset H^1_{-2-\delta}(S_2)$. Then, by Lemma~\ref{orthogonalitylemma}, $\varphi_{AB}$ is $L_2$-orthogonal to $\theta_{AB}$. As $\theta_{AB}$ was arbitrary in $\mathbb{F}_{1,\delta}\cap \ker\scurl_2$ and $\varphi_{AB}$ was by assumption orthogonal to $\mathbb{E}_{1,\delta}$, we have that $\varphi_{AB}$ is orthogonal to all of $\mathbb{F}_{1,\delta}$.

The Laplacian $\Delta_2: H^{j+2}_{\delta+2}(S_2)\rightarrow H^j_\delta(S_2)$ is formally self-adjoint and has closed range and finite dimensional kernel -- see \cite{Cantor:1981bs,McOwen:1980gz,Lockhart:1983ht} for details. By Fredholm's alternative, there exists a $\theta_{AB}\in H^{j+2}_{\delta+2}(S_2)$ such that $\varphi_{AB}=(\Delta_2\theta)_{AB}$. Using Proposition~\ref{prop:elliptictoolbox}, we can modify $\theta_{AB}$ within the class $\ker\Delta_2\cap H^{j+2}_{\delta+2}$ to obtain the estimate
\begin{equation*}
\Vert\theta_{AB}\Vert_{j+2,\delta+2}\leq C \Vert \varphi_{AB}\Vert_{j,\delta},
\end{equation*}
where $C$ only depends on $j$ and $\delta$.

Now, we can re-express the Laplacian $\Delta_2$ as
\begin{equation*}
\varphi_{AB}=(\Delta_2\theta)_{AB}=
 -2 (\scurl_2 \scurl_2 \theta)_{AB} + (\stwist_0 \sdiv_2 \theta)_{AB}.
\end{equation*}
We now want to show that $(\stwist_0 \sdiv_2 \theta)_{AB}$ vanishes (for $\delta <0$) or is in the image of $\scurl_2$ (for $\delta>0$).

By the constraint equation and commutations of the divergence and the Laplace operator, we have
$$
0=(\sdiv_2\varphi)=D^{AB}\varphi_{AB} = D^{AB}\left(\Delta \theta_{AB}\right) = \Delta \left(D^{AB}\varphi_{AB}\right) = (\Delta_0\sdiv_2\theta). 
$$
Hence, $(\sdiv_2\theta)\in \ker \Delta_0\cap L^2_{\delta+1}(S_0)$. 

If $\delta < 0$, we know that $\ker\Delta_0\cap L^2_{\delta+1}(S_0)$ only contains polynomials with degree $<1$, \emph{i.e.} constants, which means that they are in the kernel of the gradient operator $\stwist_0$. Hence, 
\begin{equation*}
\varphi_{AB}= -2 (\scurl_2 \scurl_2 \theta)_{AB}= -(\mathcal{G}_2 \scurl_2 \theta)_{AB}
\end{equation*}
and we can therefore choose $\zeta_{AB}=-(\scurl_2\theta)_{AB}$, and we get
\begin{equation*} 
\Vert\zeta_{AB}\Vert_{j+1,\delta+1}\leq \Vert\theta_{AB}\Vert_{j+2,\delta+2}\leq C\Vert\varphi_{AB}\Vert_{j,\delta}.
\end{equation*}

If $\delta > 0$, we need to be more careful. Let $\Omega\equiv \ker\Delta_0\cap L^2_{\delta+1}(S_0)$, \emph{i.e.} the set of harmonic polynomials with degree strictly smaller than $\delta+1$. Then $\stwist_0(\Omega)\subset L^2_{\delta}(S_2)$ is also a finite dimensional space of smooth fields. Since $\sdiv_2 \stwist_0 = \Delta_0$, we have the following diagram
\begin{equation*}
\xymatrix{
\Omega\subset\ker\left(\Delta_0\right)\subset \mathcal{S}_0\ar[r]^{\stwist_0}&\stwist_0\left(\Omega\right)\subset\mathcal{S}_2\ar[r]^{\sdiv_2}&\{0\}\subset\mathcal{S}_2\\
&(\scurl_2){}^{-1}\left(\stwist_0\left(\Omega\right)\right)\subset \mathcal{S}_2 \ar[u]^{\scurl_{2}}
      }.
\end{equation*}
Using the integrability condition stated in Proposition \ref{prop:integrability3}, and Remark~\ref{integrability4}, and more specifically by
$$
\mathcal{P}^{<\delta+1}(S_2)\stackrel{\scurl_2}{\longrightarrow} \mathcal{P}^{<\delta}(S_2) \stackrel{\sdiv_2}{\longrightarrow} \mathcal{P}^{<\delta-1}(S_0),
$$
we can define an a priori non unique linear mapping $\mathcal{T}: \stwist_0(\Omega)\rightarrow \mathcal{S}_2$, such that $\scurl_2\mathcal{T}$ acts as the identity on $\stwist_0\left(\Omega\right)$. As a linear operator from the finite dimensional space $\stwist_0(\Omega)\subset H^j_\delta(S_2)$ into $ H^{j+2}_{\delta+2}(S_2)$ (endowed with their respective induced Sobolev norms), $\mathcal{T}$ is bounded and, therefore, there exists a constant $C$, depending on the choice of the mapping $\mathcal{T}$, such that
\begin{align*}
\Vert(\mathcal{T}\stwist_0\sdiv_2\theta)_{AB}\Vert_{j+1,\delta+1}
&\leq C \Vert (\stwist_0\sdiv_2\theta)_{AB}\Vert_{j,\delta},\\
(\scurl_2\mathcal{T}\stwist_0\sdiv_2\theta)_{AB}
&=(\stwist_0\sdiv_2\theta)_{AB}.
\end{align*}
Now, let 
\begin{equation*}
\zeta_{AB}=-(\scurl_2\theta)_{AB}+(\mathcal{T}\stwist_0\sdiv_2\theta)_{AB}.
\end{equation*}
This gives the desired relations
\begin{align*}
(\mathcal{G}_2\zeta)_{AB}&=
2(\scurl_2\zeta)_{AB}=
-2(\scurl_2\scurl_2\theta)_{AB}+
2(\stwist_0\sdiv_4\theta)_{AB}=\varphi_{AB},\\
\Vert\zeta_{AB}\Vert_{j+1,\delta+1}&\leq 
\Vert\theta_{AB}\Vert_{j+2,\delta+2}
+C \Vert (\stwist_0\sdiv_2\theta)_{AB}\Vert_{j,\delta}
\leq C\Vert\varphi_{AB}\Vert_{j,\delta}.
\end{align*}
\end{proof}

\subsection{The spin-$s$ case}
\begin{proposition}\label{proprepspingeneral}
Let $\delta$ be in $\mathbb{R}\setminus\mathbb{Z}$, $j>0$ integer, $\varphi_{A\dots F}$ in $\ker \sdiv_{2s} \cap H^j_\delta(S_{2s})$. Let $m=\lfloor s\rfloor$, \emph{i.e.} the largest integer such that $m\leq s$. 
Then there exist a spinor field $\zeta_{A\dots F}\in H^{j+2s-1}_{\delta+2s-1}(S_{2s})$ and a constant $C$ depending only on $\delta$ and $j$ such that
\begin{align*}
\varphi_{A\dots F}&=(\mathcal{G}_{2s}\zeta)_{A\dots F},\\
\Vert\zeta_{A\dots F}\Vert_{j+2s-1,\delta+2s-1}&\leq C \Vert\varphi_{A\dots F}\Vert_{j,\delta}.
\end{align*}
\end{proposition}
\begin{proof}
By Remark~\ref{rem:Etrivial} we see that $\varphi_{A\dots F}$ is automatically orthogonal to $\mathbb{E}_{s,\delta}$ if $\delta\geq -2s-2$. If $\delta< -2s-2$, we can use Lemma~\ref{lemma:orthogonalpreimage} to 
construct $\tilde\zeta_{A\dots F}\in H^{j+2s-1}_{\delta+2s-1}(S_{2s})$ such that $\varphi_{A\dots F}+(\mathcal{G}_{2s}\tilde\zeta)_{A\dots F}\in \ker\sdiv_{2s}\cap H^j_\delta(S_{2s})$ is orthogonal to $\mathbb{E}_{s,\delta}$. The estimates in Lemma~\ref{lemma:orthogonalpreimage} are of the right type, so if we can prove the proposition for $\varphi_{A\dots F}+(\mathcal{G}_{2s}\tilde\zeta)_{A\dots F}$ instead of $\varphi_{A\dots F}$, we are done. Without loss of generality, we can therefore for the rest of the proof assume that $\varphi_{A\dots F}$ is orthogonal to $\mathbb{E}_{s,\delta}$.

Now, we will establish that $\varphi_{A\dots F}$ is orthogonal to $\mathbb{F}_{s,\delta}$ by using the constraint equation and the orthogonality to $\mathbb{E}_{s,\delta}$. The spinors in $\mathbb{F}_{s,\delta}\subset \mathcal{P}^{<-3-\delta}(S_{2s})$ are polynomial, so we can use Proposition~\ref{prop:integrability3} to conclude that $\mathbb{F}_{s,\delta}\cap \ker\mathcal{G}_{2s}=\mathbb{F}_{s,\delta}\cap \stwist_{2s-2}(\mathcal{P}^{<-2-\delta}(S_{2s-2}))$. But $\mathcal{P}^{<-2-\delta}(S_{2s-2})\subset H^1_{-2-\delta}(S_{2s-2})$, so Lemma~\ref{orthogonalitylemma} gives that $\lAngle \varphi_{A\dots F}, \zeta_{A\dots F}\rAngle_{L^2}=0$ for all $\zeta_{A\dots F}\in \mathbb{F}_{s,\delta}\cap \stwist_{2s-2}(\mathcal{P}^{<-2-\delta}(S_{2s-2}))=\mathbb{F}_{s,\delta}\cap \ker\mathcal{G}_{2s}$. By assumption $\varphi_{A\dots F}$ is orthogonal to $\mathbb{E}_{s,\delta}$ and therefore orthogonal to all of $\mathbb{F}_{s,\delta}$.

The operator $\Delta^m_{2s}: H^{j+2m}_{\delta+2m}(S_{2s})\rightarrow H^j_\delta(S_{2s})$ is formally self-adjoint and has closed range and finite dimensional kernel -- see \cite{Cantor:1981bs,McOwen:1980gz,Lockhart:1983ht} for details. By the Fredholm alternative, and the orthogonality, there exists a $\theta_{A\dots F}\in H^{j+2m}_{\delta+2m}(S_{2s})$ such that $\varphi_{A\dots F}=(\Delta^m_{2s}\theta)_{A\dots F}$. 
Using Proposition~\ref{prop:elliptictoolbox} we can modify $\theta_{A\dots F}$ within the class 
$\ker\Delta^m_{2s}\cap H^{j+2m}_{\delta+2m}(S_{2s})$ to obtain the estimate
\begin{equation*}
\Vert\theta_{A\dots F}\Vert_{j+2m,\delta+2m}\leq C \Vert \varphi_{A\dots F}\Vert_{j,\delta},
\end{equation*}
where $C$ only depends on $j$ and $\delta$.

For integer spin we can express the $\Delta^m_{2s}$ operator as
\begin{equation*}
(\Delta^{m}_{2s}\theta)_{A\dots F}=
(\stwist_{2s-2}\mathcal{F}_{2s-2}\sdiv_{2s}\theta)_{A\dots F}-(-2)^{1-m}(\mathcal{G}_{2s}\scurl_{2s}\theta)_{A\dots F}.
\end{equation*}
For half integer spin we can express the $\Delta^m_{2s}$ operator as
$2s=2m+1$
\begin{equation*}
 (\Delta^{m}_{2s}\theta)_{A\dots F}=
(\stwist_{2s-2}\mathcal{F}_{2s-2}\sdiv_{2s}\theta)_{A\dots F}
+(-2)^{-m}(\mathcal{G}_{2s}\theta)_{A\dots F} .
\end{equation*}

We now want to show that $(\stwist_{2s-2}\mathcal{F}_{2s-2}\sdiv_{2s}\theta)_{A\dots F}$ vanishes
 (for $\delta<0$) or is in the image of $\mathcal{G}_{2s}$ (for $\delta>0)$.

By the constraint equation and the fact that the divergence and the Laplace operator commute, we have
$$
0=(\sdiv_{2s}\varphi)_{C\dots F}=D^{AB}\varphi_{A\dots F} = D^{AB}\left(\Delta_{2s}^m 
\theta_{A\dots F}\right) = \Delta^m_{2s-2} \left(D^{AB}\varphi_{A\dots F}\right) = 
(\Delta^m_{2s-2} \sdiv_{2s}\theta)_{C\dots F}. 
$$
Hence, $(\sdiv_{2s}\theta)_{C\dots F}$ is in $\ker \Delta^m_{2s-2}\cap 
L^2_{\delta+2m-1}(S_{2s-2})$. 

If $\delta < 0$, we know that fields in $\ker \Delta^m_{2s-2}\cap L^2_{\delta+2m-1}(S_{2s-2})$ are 
in $\mathcal{P}^{<2m-1}(S_{2s-2})$, i. e. they are spanned by constant spinors times polynomials with maximal 
degree $2m-2$. They therefore belong to the kernel of the homogeneous order $2m-1$ operator 
$\stwist_{2s-2}\mathcal{F}_{2s-2}$. Hence, $(\stwist_{2s-2}\mathcal{F}_{2s-2}\sdiv_{2s}\theta)
_{A\dots F}=0$ and we get
\begin{equation*}
\varphi_{A\dots F}=
-(-2)^{1-m}(\mathcal{G}_{2s}\scurl_{2s}\theta)_{A\dots F},
\end{equation*}
for integer spin, and
\begin{equation*}
\varphi_{A\dots F}=(-2)^{-m}(\mathcal{G}_{2s}\theta)_{A\dots F},
\end{equation*}
for half integer spin.
For integer spin we can therefore choose $\zeta_{A\dots F}=-(-2)^{1-m}(\scurl_{2s}\theta)_{A\dots F}$, and we get
\begin{equation*} 
\Vert\zeta_{A\dots F}\Vert_{j+2s-1,\delta+2s-1}\leq C \Vert\theta_{A\dots F}\Vert_{j+2m,\delta+2m}\leq C\Vert\varphi_{A\dots F}\Vert_{j,\delta}.
\end{equation*}
For half integer spin we can choose $\zeta_{A\dots F}=(-2)^{-m}\theta_{A\dots F}$, and we get
\begin{equation*} 
\Vert\zeta_{A\dots F}\Vert_{j+2s-1,\delta+2s-1}=(-2)^{-m} \Vert\theta_{A\dots F}\Vert_{j+2m,\delta+2m}\leq C\Vert\varphi_{A\dots F}\Vert_{j,\delta}.
\end{equation*}

If $\delta > 0$, we need to be more careful. Let $\Omega\equiv \ker(\Delta^m_{2s-2})\cap\text{im}
(\sdiv_{2s})\cap L^2_{\delta+2m-1}(S_{2s-2})$. We know that it is a finite dimensional space of 
polynomial fields. $\stwist_{2s-2}\mathcal{F}_{2s-2}(\Omega)\subset L^2_{\delta}(S_{2s})$ is 
therefore also a finite dimensional space in $\mathcal{P}^{<\delta}(S_{2s})$.

Using the relations \eqref{eq:divGproperty}, \eqref{LaplacianAsGeven} and \eqref{LaplacianAsGodd}
we get 
\begin{equation*}
\Delta^m_{2s-2} \sdiv_{2s}=\sdiv_{2s}\Delta^m_{2s}
=\sdiv_{2s}\stwist_{2s-2}\mathcal{F}_{2s-2}\sdiv_{2s}.
\end{equation*}
Consequently, on $\Omega\subset\text{im}(\sdiv_{2s})$, the following relation holds
$$
\sdiv_{2s}\stwist_{2s-2}\mathcal{F}_{2s-2}\bigr |_\Omega=\Delta^m_{2s}\bigr |_\Omega=0.
$$
The relations between the considered operators can be summarized by
\begin{equation*}
\xymatrix{
\Omega\subset\ker\left(\Delta^m_{2s-2}\right)\subset \mathcal{S}_{2s-2}\ar[r]^{\stwist_{2s-2}\mathcal{F}_{2s-2}}&\stwist_{2s-2}\mathcal{F}_{2s-2}\left(\Omega\right)\subset\mathcal{S}_{2s}\ar[r]^{\sdiv_{2s}}&\{0\}\subset\mathcal{S}_{2s-2}\\
&(\mathcal{G}_{2s}){}^{-1}\left(\stwist_{2s-2}\mathcal{F}_{2s-2}\left(\Omega\right)\right)\subset 
\mathcal{S}_{2s}\ar[u]^{\mathcal{G}_{2s}}
     }.
\end{equation*}
Using the integrability condition stated by the exact sequence in Proposition~\ref{prop:integrability3} applied 
to polynomials (cf. Remark~\ref{integrability4}), and more specifically
$$
\mathcal{P}^{<\delta + 2s-1}(S_{2s})\stackrel{\mathcal{G}_{2s}}{\longrightarrow}  \mathcal{P}^{<\delta }(S_{2s}) \stackrel{\sdiv_{2s}}
{\longrightarrow}\mathcal{P}^{<\delta -1}(S_{2s-2}),
$$
we can define an a priori non unique linear mapping $\mathcal{T}: 
\stwist_{2s-2}\mathcal{F}_{2s-2}\left(\Omega\right) \rightarrow \mathcal{S}_{2s} $ such that 
$\mathcal{G}_{2s}\mathcal{T}$ is the identity operator on $\stwist_{2s-2}\mathcal{F}_{2s-2}(\Omega)$. As 
a linear operator on the finite dimensional space $\stwist_{2s-2}\mathcal{F}_{2s-2}(\Omega)\subset 
H^j_\delta(S_{2s})$ into $H^{j+2s-1}_{\delta+2s-1}(S_{2s})$ (endowed with their 
respective induced Sobolev norms), $\mathcal{T}$ is bounded and, therefore, there exists a constant $C$, 
depending on the choice of the operator $\mathcal{T}$, such that
\begin{align*}
\Vert(\mathcal{T}\stwist_{2s-2}\mathcal{F}_{2s-2}\sdiv_{2s}\theta)_{A\dots F}\Vert_{j+2s-1,\delta+2s-1}
&\leq C \Vert (\stwist_{2s-2}\mathcal{F}_{2s-2}\sdiv_{2s}\theta)_{A\dots F}\Vert_{j,\delta},\\
(\mathcal{G}_{2s}\mathcal{T}\stwist_{2s-2}\mathcal{F}_{2s-2}\sdiv_{2s}\theta)_{A\dots F}
&=(\stwist_{2s-2}\mathcal{F}_{2s-2}\sdiv_{2s}\theta)_{A\dots F}.
\end{align*}
Now, for integer spin we can therefore choose 
\begin{equation*}
\zeta_{A\dots F}=
(\mathcal{T}\stwist_{2s-2}\mathcal{F}_{2s-2}\sdiv_{2s}\theta)_{A\dots F}
-(-2)^{1-m}(\scurl_{2s}\theta)_{A\dots F}.
\end{equation*}
This gives the desired relations
\begin{align*}
(\mathcal{G}_{2s}\zeta)_{A\dots F}&=
(\stwist_{2s-2}\mathcal{F}_{2s-2}\sdiv_{2s}\theta)_{A\dots F}
-(-2)^{1-m}(\mathcal{G}_{2s}\scurl_{2s}\theta)_{A\dots F}\nonumber\\
&=(\Delta^m_{2s}\theta)_{A\dots F}=\varphi_{A\dots F},\\
\Vert\zeta_{A\dots F}\Vert_{j+2s-1,\delta+2s-1}&\leq 
C(\Vert\theta_{A\dots F}\Vert_{j+2m,\delta+2m}
+ \Vert (\stwist_{2s-2}\mathcal{F}_{2s-2}\sdiv_{2s}\theta)_{A\dots F}\Vert_{j,\delta})\nonumber\\
&\leq C\Vert\varphi_{A\dots F}\Vert_{j,\delta}.
\end{align*}

For half integer spin, we can choose 
\begin{equation*}
\zeta_{A\dots F}=
(\mathcal{T}\stwist_{2s-2}\mathcal{F}_{2s-2}\sdiv_{2s}\theta)_{A\dots F}
+(-2)^{-m}\theta_{A\dots F}.
\end{equation*}
This gives the desired relations
\begin{align*}
(\mathcal{G}_{2s}\zeta)_{A\dots F}&=
(\stwist_{2s-2}\mathcal{F}_{2s-2}\sdiv_{2s}\theta)_{A\dots F}
(-2)^{-m}(\mathcal{G}_{2s}\theta)_{A\dots F}\nonumber\\
&=(\Delta^m_{2s}\theta)_{A\dots F}=\varphi_{A\dots F},\\
\Vert\zeta_{A\dots F}\Vert_{j+2s-1,\delta+2s-1}&\leq 
C(\Vert\theta_{A\dots F}\Vert_{j+2m,\delta+2m}
+ \Vert (\stwist_{2s-2}\mathcal{F}_{2s-2}\sdiv_{2s}\theta)_{A\dots F}\Vert_{j,\delta})\nonumber\\
&\leq C\Vert\varphi_{A\dots F}\Vert_{j,\delta}.
\end{align*}
\end{proof}

\section{Estimates for solutions of the scalar wave equation with initial data with arbitrary weight} \label{sec:estsol}

This section contains complementary results for the study of the decay of the solution of the wave equation for the Cauchy problem such as the one stated in \cite{Klainerman:1985wn,Klainerman:1986wra} (using the vector field method), in \cite{MR862696} (using the integral representation), and \cite{DAncona:2001kr} (using Strichartz estimates). The purpose is to link precisely the asymptotic behavior of the initial data at $i^0$ and the asymptotic behavior in the future region $t\geq 0$.

It is important to remark that these different decay results for the wave equation hold for different regularities of the initial data $(f,g)$. In all the aforementioned results, the limiting factor is the use of the Sobolev embedding from $H^2$ into $L^\infty$, and more precisely the way it is used. If the Sobolev embedding is used at the level of both $f$ and $g$, we need $(f, g) \in H^j_\delta \times H^{j-1}_{\delta-1}$ for $j \geq 3$. This phenomenon occurs for instance in the work of Asakura \cite{MR862696}, who is relies on the integral representation of solutions. Energy methods can be used to give a result with weaker regularity assumtions. The Klainerman-Sobolev inequality \cite[Theorem 1]{Klainerman:1987hz} yields a decay estimate with weight $\delta=-3/2$ for $j=2$. Further, using the conformal compactification of Minkowski space to a subset of the Einstein cylinder, and the conformal transformation properties of the wave equation, gives using standard estimates for the wave equation in the Einstein cylinder the decay result for $\delta = -2$ and $j=2$, cf. \cite[Section 6.7]{MR1466700}. It appears to be an open problem to prove the corresponding decay result with $j=2$ for general weight $\delta$.

In this section we shall prove estimates for general weights $\delta$. Although, for $\delta > 0$, these are not actually \emph{decay} estimates, it will be convenient to refer to them using this term. It goes without saying that the most important applications are those with $\delta < 0$. In the following, we shall consider the Cauchy problem
\begin{equation} \label{cauchywave} \left\{
\begin{array}{l}
\square \phi =0 , \\
\phi|_{t=0}=f \in H^j_{\delta}(\R^3, \C) , \\ 
\partial_t \phi|_{t=0}=g \in H^{j-1}_{\delta-1}(\R^3,\C).
\end{array}\right.
\end{equation}

The following representation formula then holds (\cite{Evans:2010wj} on flat space-time or \cite[Theorem 5.3.3]{Friedlander:1975ub} for arbitrary curved background).
\begin{lemma}\label{integralwave}
 The solution of the Cauchy problem \eqref{cauchywave} is given by the representation formula
$$
\phi(t,x) = \frac{1}{4\pi} \left(\int_{\mathbb{S}^2}t\left(g(x+t\omega) + \partial_\omega f (x+t\omega)\right) + f(x+t\omega)   \ud \mu_{\mathbb{S}^2}\right),
$$
where $\mathbb{S}^2$ is the unit 2-sphere and $\partial_\omega$ is the derivative in the unit outer normal direction $\omega$ to $\mathbb{S}^2$.
\end{lemma}

Making use of the representation formula stated in Lemma~\ref{integralwave}, and of Proposition~\ref{prop:decaysob}, gives the following result.
\begin{proposition}\label{prop:decaywave} Let $j\geq 3$ and $\delta$ in $\R$, and let $\phi$ be a solution to the Cauchy problem \eqref{cauchywave}. The following inequality holds in the region $t > 0$. 
\begin{equation*}
 |\phi(t,x)|\leq  C \left(\Vert f\Vert_{3,\delta}+\Vert g\Vert_{2,\delta-1}\right)
\begin{cases}
 \lAngle v\rAngle^{-1}\lAngle u\rAngle^{1+\delta} &\text{ if } \delta <-1 , \\ 
 \dfrac{\log\lAngle v \rAngle-\log\lAngle u \rAngle}{\lAngle v \rAngle -\lAngle u \rAngle} &\text{ if } \delta =-1 , \\ 
 \lAngle v\rAngle^{\delta} &\text{ if } \delta >-1 .
\end{cases}
\end{equation*}
If, furthermore, $(k,l,m)$ is a triple of non-negative integers, $j\geq 3+k+l+m$, the following pointwise inequality holds, for all $t>0$ and $r>1$,
\begin{equation*}
| \partial_v^k\partial_u^l\snabla^m\phi |\leq C \left( \Vert f\Vert_{3+k+l+m, \delta} + \Vert g\Vert_{2+k+l+m, \delta-1}\right)
\begin{cases}
 \lAngle u\rAngle^{1+\delta-l} \lAngle v\rAngle^{-1-k-m} &\text{ if } \delta<l-1\\ 
\dfrac{\log\lAngle v \rAngle-\log\lAngle u \rAngle}{\lAngle v \rAngle^{l+m}\left(\lAngle v \rAngle -\lAngle u \rAngle\right)} &\text{ if } \delta = l-1\\ 
 \lAngle v\rAngle^{\delta -l-m-k} &\text{ if } \delta>l-1,
\end{cases}
\end{equation*}
where $\partial_u = \frac{1}{2} (\partial_t-\partial_r)$ and $ \partial_v = \frac{1}{2} (\partial_t+\partial_r)$ are respectively the outgoing and ingoing null directions.
\end{proposition}
\begin{remark}\label{rem:decaywave}
\begin{enumerate}
\item  For $\delta=-1$, it is important to notice that the following inequalities hold
$$
 \dfrac{\log\lAngle v \rAngle-\log\lAngle u \rAngle}{\lAngle v \rAngle -\lAngle u \rAngle} \leq \dfrac{1}{\lAngle u\rAngle } \text{ and }  \dfrac{\log\lAngle v \rAngle-\log\lAngle u \rAngle}{\lAngle v \rAngle -\lAngle u \rAngle} \leq C \dfrac{\log\lAngle v \rAngle}{\lAngle v \rAngle}.
$$
As a consequence, in the interior region $t>3r$, the decay result for $\phi$ is 
$$
\left|\phi(t,x)\right| \leq \dfrac{\widetilde C}{\lAngle v \rAngle },
$$
and, in the exterior region $3r>t >\frac{r}{3}$, 
$$
\left|\phi(t,x)\right| \leq \widetilde C\dfrac{\log\lAngle v \rAngle}{\lAngle v \rAngle }.
$$
The constant $\widetilde C$ depends on the norms of $f$ and $g$ as above.
\item It is important to note that, in the interior region, $\lAngle u \rAngle$ and $\lAngle v \rAngle$ are equivalent, and, as a consequence, the general estimate holds, for all weight $\delta$: 
\begin{equation*}
| \partial_v^k\partial_u^l\snabla^m\phi|\leq  \widetilde C  \lAngle v\rAngle^{\delta -l-m-k}.
\end{equation*}
\end{enumerate}
\end{remark}

For the proof we will need some integral estimates.
\begin{lemma}\label{integralextimateshypergeom}
For any $\delta$ in $\mathbb{R}$, we have the following integral estimates, for all $t>0$ and $x$ in $\R^3$,
$$
\int_{\mathbb{S}^2}\lAngle |x+t\omega|\rAngle^{\delta} \ud\mu_{\mathbb{S}^2}\leq 8\pi \max\left(1, \dfrac{1}{|2+\delta|}\right)\dfrac{\max\left(\lAngle u\rAngle^{\delta+2}, \lAngle v\rAngle^{\delta+2} \right)}{\lAngle v\rAngle \left(\lAngle u\rAngle  + \lAngle v\rAngle\right)} \text{ for }\delta\neq -2
$$
and
$$
\int_{\mathbb{S}^2}\lAngle |x+t\omega|\rAngle^{\delta} \ud\mu_{\mathbb{S}^2} = 8\pi\dfrac{\log \left(\dfrac{\lAngle u\rAngle}{\lAngle v\rAngle} \right)}{(\lAngle u\rAngle^2 -\lAngle v\rAngle^2)} \text{ for } \delta=-2
$$
\end{lemma}
\begin{remark}
\begin{enumerate}
\item For $\delta \neq -2$, the upper bounds become lower bounds if one replaces  $\max\left(1, \dfrac{1}{|2+\delta|}\right)$ by  $\min\left(1, \dfrac{1}{|2+\delta|}\right)$. 
\item  It is important to notice that, for $\delta \in (-2,-1]$, the estimate
$$
\int_{\mathbb{S}^2}\lAngle |x+t\omega|\rAngle^{\delta} \ud \mu_{\mathbb{S}^2} \leq  C \lAngle v\rAngle ^\delta
$$
is \emph{stronger} than 
$$
\int_{\mathbb{S}^2}\lAngle |x+t\omega|\rAngle^{\delta} \ud \mu_{\mathbb{S}^2} \leq  C \dfrac{\lAngle u\rAngle ^{1+\delta}}{\lAngle v\rAngle}
$$
since 
$$
\lAngle v \rAngle^\delta  = \lAngle u \rAngle ^{1+\delta} \lAngle v \rAngle^{-1} \lAngle v \rAngle^{\delta+1}\lAngle u \rAngle ^{-(1+\delta)}\leq \lAngle u \rAngle ^{1+\delta} \lAngle v \rAngle^{-1}
$$
(since $\delta+1<0$). It should also be noted that this result agrees with the estimate stated by Asakura in \cite{MR862696}. His assumptions on the initial data
$$
f = \mathcal{O}(\lAngle r\rAngle^\delta), g = \mathcal{O}(\lAngle r\rAngle^\delta),
$$
as well as their derivatives, implies that the integral $\int_{\mathbb{S}^2}\lAngle |x+t\omega|\rAngle^{\delta}\ud\mu_{\mathbb{S}^2}$ gives the asymptotic behavior of the solutions of the linear wave equation.
\end{enumerate}
\end{remark}

\begin{proof}
Let $(t,x)$ be fixed and consider the sphere $S(x,t)$ with center $x$ and radius $t$. Let $q$ be a point on the sphere $S(x,t)$. The coordinates of $q$ are then given by $(\theta,\phi)$ defined by:
\begin{itemize}
 \item in the 2-plane containing the origin $o$, the point $x$ and $q$, $\theta$ is the oriented angle
$$
\theta = (\vec{xo}, \vec{xq}) \in (0,\pi);
$$
\item in the plane orthogonal to $\vec{ox}$ and passing through $x$, one chooses a direction of origin. The direction of the orthogonal projection of $\vec{xq}$ on this plane is labeled by an angle $\phi$ belonging to $(0,2\pi)$. 
\end{itemize}
The integral can now be rewritten as
\begin{align*}
\int_{\mathbb{S}^2}\lAngle |x+t\omega|\rAngle^{\delta} \ud \mu_{\mathbb{S}^2} ={}& \int_0^{2\pi}\int_0^{\pi}(1+r^2 +t^2-2tr\cos\theta)^{\delta/2}\sin\theta\ud \theta\ud \phi \\
={}&
\begin{cases}
8\pi\dfrac{\lAngle u\rAngle^{2+\delta}-\lAngle v\rAngle^{2+\delta}}{(2+\delta)(\lAngle u\rAngle^2 -\lAngle v\rAngle^2)} &\text{ if } \delta \neq -2 \text{ and } \lAngle u\rAngle\neq \lAngle v\rAngle,\\
8\pi\dfrac{\log \left(\frac{\lAngle u\rAngle}{\lAngle v\rAngle} \right)}{(\lAngle u\rAngle^2 -\lAngle v\rAngle^2)} &\text{ if } \delta=-2 \text{ and } \lAngle u\rAngle\neq \lAngle v\rAngle,\\  
4\pi \lAngle v\rAngle^{\delta} &\text{ if } \lAngle u\rAngle = \lAngle v\rAngle.\\  
\end{cases}
\end{align*}
We note that $\lAngle u\rAngle$ and $\lAngle v \rAngle$ always satisfy
\begin{itemize}
\item $\lAngle u\rAngle\geq 1$ and $\lAngle v\rAngle\geq 1$;
\item for $t\geq 0$, $\lAngle u\rAngle\leq \lAngle v\rAngle$.
\end{itemize}

Let $F_\kappa$ denote the function defined, for $\kappa\geq 0$ and $z$ in $(0,1]$,
$$F_\kappa(z)=
\begin{cases}
\dfrac{1-z^{\kappa}}{\kappa(1-z)} & \text{ if } \kappa > 0 \text{ and } z\neq 1,\\
\dfrac{-\log z}{1-z} &\text{ if } \kappa = 0 \text{ and } z\neq 1,\\  
1 &\text{ if } z=1.
\end{cases}
$$
For $\kappa>0$, $F_\kappa$ is a continuous monotonic non-negative function on $(0,1]$. It is bounded from above by $C_\kappa = \max(1, \kappa^{-1})$ and from below by $c_\kappa = \min(1, \kappa^{-1})$ on $(0,1]$. For all $\delta$, the following identities hold,
$$
\dfrac{\lAngle u\rAngle^{2+\delta}-\lAngle v\rAngle^{2+\delta}}{(2+\delta)(\lAngle u\rAngle^2 -\lAngle v\rAngle^2)} = 
\begin{cases}
\dfrac{\lAngle u\rAngle^{\delta+2}}{\lAngle v\rAngle \left( \lAngle v\rAngle+\lAngle u\rAngle\right)}F_{-\delta-2}\left( \dfrac{\lAngle u\rAngle}{\lAngle v\rAngle}\right) & \text{ if } \delta+2<0.\\
\dfrac{\lAngle v\rAngle^{\delta+1}}{\lAngle v\rAngle+\lAngle u\rAngle}F_{\delta+2} \left(\dfrac{\lAngle u\rAngle}{\lAngle v\rAngle}\right) & \text{ if } \delta+2\geq 0. 
\end{cases}
$$
Using the bounds on $F_\kappa$, one gets
\begin{align*}
\dfrac{\lAngle u\rAngle^{2+\delta}-\lAngle v\rAngle^{2+\delta}}{(2+\delta)(\lAngle u\rAngle^2 -\lAngle v\rAngle^2)} &\leq 
\begin{cases}
C_{-\delta-2}\dfrac{\lAngle u\rAngle^{\delta+2}}{\lAngle v\rAngle \left( \lAngle v\rAngle+\lAngle u\rAngle\right)}& \text{ if } \delta<-2\\
C_{\delta+2}\dfrac{\lAngle v\rAngle^{\delta+1}}{\lAngle v\rAngle+\lAngle u\rAngle} & \text{ if } \delta>-2,
\end{cases}\\
\intertext{ and }\dfrac{\lAngle u\rAngle^{2+\delta}-\lAngle v\rAngle^{2+\delta}}{(2+\delta)(\lAngle u\rAngle^2 -\lAngle v\rAngle^2)} &\geq
\begin{cases} 
c_{-\delta-2}\dfrac{\lAngle u\rAngle^{\delta+2}}{\lAngle v\rAngle \left( \lAngle v\rAngle+\lAngle u\rAngle\right)}& \text{ if } \delta<-2\\
c_{\delta+2}\dfrac{\lAngle v\rAngle^{\delta+1}}{\lAngle v\rAngle+\lAngle u\rAngle} & \text{ if } \delta>-2.
\end{cases}
\end{align*}
\end{proof}

\begin{remark} These inequalities actually provide us with \emph{global} estimates for the solution of the wave equation, that is to say inequalities valid both on the exterior and the interior regions.
\end{remark}

\begin{proof}[Proof of Proposition~\ref{prop:decaywave}] Using Proposition~\ref{prop:decaysob}, one knows that if $f\in H^j_{\delta}$, $g\in H^{j-1}_{\delta-1}$, $j\geq 2+n$ and $j\geq 3+m$, there is constant $C$ such that
\begin{eqnarray*}
 |D^n f(y)| & \leq & C\lAngle y\rAngle^{\delta-n}\Vert f\Vert_{2+n,\delta}\\
|D^m g(y)| & \leq & C\lAngle y\rAngle^{\delta-1-m}\Vert g\Vert_{2+m,\delta-1}.
\end{eqnarray*}
Using the representation formula stated in Lemma \ref{integralwave}, one gets immediately
  \begin{gather*}
   |\phi(t,x)|\leq C\left(\Vert f\Vert_{3,\delta}+\Vert g\Vert_{2,\delta-1}\right)\int_{\mathbb{S}^2}\left(\lAngle|x+t\omega|\rAngle^{\delta} +t\lAngle|x+t\omega|\rAngle^{\delta-1}\right)\ud \mu_{\mathbb{S}^2}.
  \end{gather*}
We can use the estimate $t\leq (\lAngle u\rAngle^2+\lAngle v\rAngle^2)^{1/2}$ and Lemma~\ref{integralextimateshypergeom} to obtain global estimates for solutions of the wave equation. We compare the contributions from $f$ in $H^j_{\delta}$ and $g$ in $H^{j-1}_{\delta-1}$ in the following table. Its first column is the range of weights $\delta$ considered. The second column contains the asymptotic behavior of
$$
\int_{\mathbb{S}^2}\lAngle|x+t\omega|\rAngle^{\delta}\ud \mu_{\mathbb{S}^2}
$$
coming from $f$ in $H^j_{\delta}$. The third column gives the behaviour of
$$
\int_{\mathbb{S}^2} t\lAngle|x+t\omega|\rAngle^{\delta-1} \ud \mu_{\mathbb{S}^2}
$$
coming from $\partial f$ and $g$ in $H^{j-1}_{\delta-1}$. The estimates stated in these columns are written in such a way that the first factor gives the estimate which is multiplied by bounded quantities, except when considering logarithmic terms. Finally, the last column gives the estimate for the full solution of the wave equation obtained by summing the previous integrals. We use the function $F_0$ defined by
$$
F_{0}(z)=\frac{-\log(z)}{1-z} \text{ which satisfies } F_0(z)\leq z^{-1} \text{ for all } 0\leq z\leq 1. 
$$
\begin{equation*}
\begin{array}{|c|c|c|c|}
\hline
&f & t(\partial_\omega f+g) & f+t\partial_\omega f+tg\\
\hline
\delta<-2 & \dfrac{\lAngle u \rAngle ^{\delta+1}}{\lAngle v\rAngle}\cdot \dfrac{\lAngle u \rAngle}{( \lAngle v\rAngle+\lAngle u\rAngle)} & \dfrac{\lAngle u \rAngle ^{\delta+1}}{\lAngle v\rAngle }\cdot \dfrac{t}{\lAngle v\rAngle+\lAngle u\rAngle} & \dfrac{\lAngle u \rAngle^{\delta+1} }{\lAngle v\rAngle} \\
\hline
\delta=-2 & \dfrac{1}{{\lAngle u\rAngle+\lAngle v\rAngle}}\cdot \dfrac{F_0\left(\tfrac{\lAngle u\rAngle }{\lAngle v\rAngle}\right) }{\lAngle v\rAngle}  & \dfrac{1}{\lAngle u \rAngle\lAngle v\rAngle}\cdot\dfrac{t}{(\lAngle v\rAngle+\lAngle u\rAngle)} & \dfrac{\lAngle u \rAngle^{\delta+1} }{\lAngle v\rAngle}\\
\hline
\delta=-1 & \lAngle v\rAngle^\delta \cdot \dfrac{\lAngle v\rAngle}{\lAngle v\rAngle+\lAngle u\rAngle} & \dfrac{F_0\left(\tfrac{\lAngle u\rAngle }{\lAngle v\rAngle}\right) }{\lAngle v\rAngle} \cdot \dfrac{t}{{\lAngle u\rAngle+\lAngle v\rAngle}} &  \dfrac{F_0\left(\tfrac{\lAngle u\rAngle }{\lAngle v\rAngle}\right) }{\lAngle v\rAngle} 
\\

\hline

-2<\delta<-1 & \lAngle v\rAngle^{\delta}\cdot \dfrac{\lAngle v\rAngle}{\lAngle v\rAngle+\lAngle u\rAngle} & \dfrac{\lAngle u \rAngle^{\delta+1} }{\lAngle v\rAngle}\cdot \dfrac{t}{(\lAngle v\rAngle+\lAngle u\rAngle)} & \dfrac{\lAngle u \rAngle^{\delta+1} }{\lAngle v\rAngle}\\

\hline

\hline
\delta>-1 & \lAngle v\rAngle^{\delta}\cdot \dfrac{\lAngle v\rAngle}{\lAngle v\rAngle+\lAngle u\rAngle} & \lAngle v\rAngle^{\delta}\cdot \dfrac{t}{\lAngle v\rAngle+\lAngle u\rAngle} & \lAngle v \rAngle ^\delta\\
\hline
\end{array}
\end{equation*}

As a consequence, the following pointwise estimate holds for $\phi$: for all $t\geq 0$ and $x$ in $\R^3$, 
\begin{equation}\label{ineqwave}
|\phi(t,x)|\leq C\left(\Vert f\Vert_{3,\delta}+\Vert g\Vert_{2,\delta-1}\right)\begin{cases}
\lAngle v\rAngle^{-1}\lAngle u\rAngle^{1+\delta}&\text{ for }\delta <-1\\
\dfrac{\log\lAngle v\rAngle- \log\lAngle u\rAngle}{\lAngle v\rAngle-\lAngle u\rAngle} &\text{ for }\delta =-1\\
\lAngle v\rAngle^\delta&\text{ for }\delta >-1
\end{cases}.
\end{equation}

The same process can be applied to the derivatives of $\phi$ in the direction of $u$ and $v$. The integral representations of the derivatives are then
\begin{eqnarray*} 
 \partial_v \phi (t,x) &=&\frac{1}{8\pi}\int_{\mathbb{S}^2}\big(t\left(\partial_r g(x+t\omega)+\partial_\omega g(x+t\omega)\right)+ g(x+t\omega) \\ 
 && + t\left(\partial_r\partial_\omega f(x+t\omega)+\partial^2_\omega f(x+t\omega)\right)+\partial_rf(x+t\omega)+2\partial_\omega f(x+t\omega)\big)\ud \mu_{\mathbb{S}^2},\\
 \partial_u \phi (t,x) &=&\frac{1}{8\pi}\int_{\mathbb{S}^2}\big(t\left(-\partial_r g(x+t\omega)+\partial_\omega g(x+t\omega)\right)+ g(x+t\omega) \\ 
 && + t\left(-\partial_r\partial_\omega f(x+t\omega)+\partial^2_\omega f(x+t\omega)\right)-\partial_rf(x+t\omega)+2\partial_\omega f(x+t\omega)\big)\ud \mu_{\mathbb{S}^2}.
\end{eqnarray*}
Using Sobolev embeddings, one gets immediately
\begin{eqnarray*}
 |\partial_v \phi (t,x)| &\leq & C\left(\Vert f\Vert_{4,\delta}+\Vert g\Vert_{3,\delta-1}\right)\int_{\mathbb{S}^2}\left(\lAngle |x+t\omega|\rAngle^{\delta-1} +t\lAngle |x+t\omega|\rAngle^{\delta-2}\right)\ud \mu_{\mathbb{S}^2},\\
 |\partial_u \phi (t,x)| &\leq &C\left(\Vert f\Vert_{4,\delta}+\Vert g\Vert_{3,\delta-1}\right)\int_{\mathbb{S}^2}\left(\lAngle |x+t\omega|\rAngle^{\delta-1} +t\lAngle |x+t\omega|\rAngle^{\delta-2}\right)\ud \mu_{\mathbb{S}^2}.
\end{eqnarray*}
Again using Lemma~\ref{integralextimateshypergeom} and the same comparison procedure as in the previous table, one gets that, for all $t>0$ and $x$ in $\R^3$,
\begin{equation}\label{ineqwave0}
 \max\left(|\partial_v \phi (t,x)|,  |\partial_u \phi (t,x)|\right) \leq C\left(\Vert f\Vert_{4,\delta}+\Vert g\Vert_{3,\delta-1}\right)
 \begin{cases}
\lAngle v\rAngle^{-1}\lAngle u\rAngle^{\delta}&\text{ for }\delta <0\\
\dfrac{\log\lAngle v\rAngle- \log\lAngle u\rAngle}{\lAngle v\rAngle-\lAngle u\rAngle} &\text{ for }\delta =0\\
\lAngle v\rAngle^{\delta-1} &\text{ for }\delta >0
\end{cases}
\end{equation}

Using these results, one can now refine the estimates for the derivatives of the function $\phi$, using the commutators properties of the wave equation with the vector fields generating the symmetries of the metric. Introducing
$$
S=u\partial_u+v\partial_v \text{ which satisfies } [S,\square]=-2\square,
$$
the function $S\phi$ satisfies the Cauchy problem for the linear wave equation
\begin{equation*} \left\{
\begin{array}{l}
\square\left( S \phi\right) =0 \\
S\phi|_{t=0} \in H^{j-1}_{\delta}(\R^3)\\ 
\partial_t S\phi|_{t=0} \in H^{j-2}_{\delta-1}(\R^3).
\end{array}\right.
\end{equation*}
As a consequence, one can apply the decay results \eqref{ineqwave} and\eqref{ineqwave0} to $S\phi$. This gives
\begin{equation*}
|S\phi(t,x)|\leq
C \left(\Vert f\Vert_{4,\delta}+\Vert g\Vert_{3,\delta-1}\right)
\begin{cases}
\lAngle v\rAngle^{-1}\lAngle u\rAngle^{1+\delta}&\text{ for }\delta <-1\\
 (\lAngle v\rAngle-\lAngle u\rAngle)^{-1}\left(\log\lAngle v\rAngle- \log\lAngle u\rAngle\right) &\text{ for }\delta =-1\\
\lAngle v\rAngle^\delta&\text{ for }\delta >-1
\end{cases}.
\end{equation*}
To get the full decay result, we need to compare carefully the decay for $S\phi$, and the decay for $\partial_u \phi$ to get the full decay result for $\partial_v \phi$. We follow the same procedure as above to compare the decay of these terms in a table using 
$$
|u| \leq \lAngle u \rAngle, \quad v \leq \lAngle v \rAngle, \quad  \text{ and, for all }r>1, \quad 
\left |\dfrac{u}{v}\right | \leq C \dfrac{\lAngle  u\rAngle}{\lAngle v \rAngle}\leq C.
$$
\begin{equation*}
\begin{array}{|c|c|c|c|}
\hline
&\frac{  S\phi}{v} & \frac{u}{v}\partial_u\phi & \partial_v \phi \\
\hline
\delta<-1 & \dfrac{\lAngle u \rAngle ^{\delta+1}}{\lAngle v\rAngle}\cdot \dfrac{1}{v} & \dfrac{\lAngle u \rAngle^{\delta} }{\lAngle v\rAngle}\cdot \dfrac{u}{ v }& \dfrac{\lAngle u \rAngle^{\delta+1} }{\lAngle v\rAngle^2} \\
\hline
\delta=-1 & \dfrac{1}{v}\dfrac{F_0\left(\tfrac{\lAngle u\rAngle }{\lAngle v\rAngle}\right) }{\lAngle v\rAngle} & \dfrac{1}{\lAngle u \rAngle\lAngle v\rAngle}\cdot\dfrac{u}{v} & \dfrac{F_0\left(\tfrac{\lAngle u\rAngle }{\lAngle v\rAngle}\right) }{\lAngle v\rAngle^2}\\
\hline
-1<\delta<0 &  \dfrac{\lAngle v\rAngle^{\delta}}{v} & \dfrac{\lAngle u \rAngle^{\delta}}{\lAngle v\rAngle}\cdot \dfrac{u}{ v } & \lAngle v \rAngle ^{\delta-1}\\
\hline
\delta= 0 & \dfrac{\lAngle v\rAngle^{\delta}}{v}  & \dfrac{u}{v}\cdot\dfrac{F_0\left(\tfrac{\lAngle u\rAngle }{\lAngle v\rAngle}\right) }{\lAngle v\rAngle} &  \lAngle v \rAngle ^{\delta-1}\\
\hline
\delta>0 & \dfrac{\lAngle v\rAngle^{\delta}}{v} & \dfrac{u\lAngle v\rAngle^{\delta-1}}{v} & \lAngle v \rAngle ^{\delta-1}\\
\hline
\end{array}
\end{equation*}

 This consequently gives the following
\begin{equation*}
|\partial_v\phi(t,x)|\leq
C \left(\Vert f\Vert_{4,\delta}+\Vert g\Vert_{3,\delta-1}\right)
\begin{cases}
\lAngle v\rAngle^{-2} \lAngle u\rAngle^{\delta+1} &\text{ if } \delta<-1 , \\
  \dfrac{1}{\lAngle v\rAngle}\dfrac{\log\lAngle v\rAngle- \log\lAngle u\rAngle}{\lAngle v\rAngle-\lAngle u\rAngle} & \text{ if } \delta=-1, \\
\lAngle v\rAngle^{\delta-1} &\text{ if } \delta>-1.
\end{cases}
\end{equation*}

Finally, the fact that $\square$ commutes with the generators of $SO(3)$,
$$
x^i\partial_j-x^j\partial_i ,
$$
can be used to obtain, for $t>0$ and $r>1$,
$$
|\snabla\phi |\leq C \left(\Vert f\Vert_{4,\delta}+\Vert g\Vert_{3,\delta-1}\right)
\begin{cases}
\lAngle v\rAngle^{-2} \lAngle u\rAngle^{\delta+1} &\text{ if } \delta<-1 , \\
  \dfrac{1}{\lAngle v\rAngle}\dfrac{\log\lAngle v\rAngle- \log\lAngle u\rAngle}{\lAngle v\rAngle-\lAngle u\rAngle} & \text{ if } \delta=-1 , \\
\lAngle v\rAngle^{\delta-1} &\text{ if } \delta>-1.
\end{cases}
$$
The proof is completed by a recursion over the number of derivatives, which is not written here in detail. \end{proof}

\begin{remark}
\begin{enumerate}
\item The derivatives $\partial_u$ and $\partial_v$, $\snabla$ play different roles when considering the full scale of weights. This difference is at the origin of the failure of the peeling property for higher spin fields when the rate of decay at $i^0$ of the initial data is too low.
 \item This difference between derivatives can be explained by considering the derivatives of the fundamental solution of the wave equation. 
\end{enumerate}
\end{remark}

\section{Estimates for spinor fields represented by  potentials} \label{sec:estspin} 
Penrose, in his original paper on zero-rest mass fields \cite{MR0175590}, proved to the following two results:
\begin{itemize}
\item the existence for analytic massless fields of arbitrary spin of representation of the form
$$
\phi_{A\dots F}= \xi_1^{A'}\dots \xi_{2s}^{F'}\nabla_{AA'}\dots\nabla_{FF'} \chi,
$$
where the $\xi^{A'}$ are constant spinors and $\chi$ is a complex function satisfying the wave equation
$$
\square \chi=0.
$$ 
\item from a decay ansatz for $\chi$ along outgoing null light rays, he deduced the full peeling result for the considered field.
\end{itemize}
The purpose of this section is to give a similar result for massless field admitting a potential of the form considered by Penrose. The decay result for the solution of the wave equation which is used in this section is given by Proposition \ref{prop:decaywave}.

\subsection{Geometric background and preliminary lemmata}\label{geometricbackground}

The geometric framework and notations are introduced in this section. The geometric background is the Minkowski space-time. We consider on this space time the normalized null tetrad defined by

\begin{align*}
l^a  &=  \sqrt{2}\partial_v = \frac{1}{\sqrt{2}}\left( \frac{\partial}{\partial t} + \frac{\partial }{\partial r}\right), & 
m^a &= \frac{1}{r\sqrt{2}}\left( \frac{\partial}{\partial \theta} + \frac{i}{\sin\theta}\frac{\partial}{\partial \varphi} \right) , \\
n^a&=\sqrt{2}\partial_u = \frac{1}{\sqrt{2}}\left( \frac{\partial}{\partial t} - \frac{\partial}{\partial r}\right),  & \overline{m}^a & =  \frac{1}{r\sqrt{2}}\left( \frac{\partial}{\partial \theta} - \frac{i}{\sin\theta}\frac{\partial}{\partial \varphi} \right),
\end{align*}

so that $l_a n^a=1$ and $m_a\overline{m}^a=-1$, and 
the plane spanned by $l^a, n^a$ is orthogonal to that spanned by $m^a, \overline{m}^a$.
The derivatives in the directions $l^a, n^a, m^a, \overline{m}^a$ are denoted by $D, D', \delta, \delta'$ respectively. Consider finally a spin basis $(o^A,\iota^A)$ arising from this tetrad, i.e.,
$$
\begin{array}{cc}
l^a = o^A\overline{o}^{A'},&\quad m^a= o^A\overline{\iota}^{A'} ,\\
n^a= \iota ^A \overline{\iota}^{A'}, &\quad \overline{m}^a= \iota^A\overline{o}^{A'} .
\end{array}
$$
This basis satisfies
\begin{subequations}\label{basis}
\begin{align}
D o^A&=0 , & D \iota^A&=0 ,\\
D' o^A &= 0 ,  & D' \iota^A&=0 , \\
 \delta o^A&=\frac{\cot{\theta}}{2 r \sqrt{2}} o^A, &\delta \iota^A&=-\frac{\cot{\theta}}{2 r \sqrt{2}} \iota^A-\frac{1}{r\sqrt{2}} o^A , \\
 \delta' o^A&= - \frac{\cot{\theta}}{2 r \sqrt{2}}o^A+\frac{\iota^A}{r\sqrt{2}}, 
 & \delta' \iota^A&= \frac{\cot{\theta}}{2r \sqrt{2}}\iota^A .
\end{align}
\end{subequations}

\begin{lemma}[Commutators]\label{commutator}
The following commutator relations hold:
\begin{itemize}
\item $D$ and $D'$ commute.
\item consider the gradient $\snabla$ on the sphere of radius $r$,
$$
\snabla   = -\overline{m} \delta  -m\delta' ,
$$
then, for any positive integer $k$, 
\begin{align*}
\snabla^kD&= D\snabla^k + \frac{k}{r\sqrt{2}}  \snabla^k , \\
\snabla^kD ' &= D'\snabla^k - \frac{k}{r\sqrt{2}}  \snabla^k,
\end{align*}
where $\snabla^k$ is the $k$-th power of the operator $\snabla$.
\end{itemize}
\end{lemma}
\begin{remark}\label{rem:commute} One can directly infer from this lemma that the operator $r\snabla$ commutes with $D$ and $D'$.
\end{remark}

\begin{proof}
The proof follows directly from the commutator relations
\begin{equation*}
\delta D = D\delta + \frac{1}{r\sqrt{2}}\delta, \quad \delta D' = D\delta - \frac{1}{r\sqrt{2}}\delta,
\end{equation*}
and their complex conjugates, and on
$$
D m^a = D'm^a = 0,
$$
and their complex conjugates.
\end{proof}

\begin{lemma}[Asymptotic behavior of the decomposition of a constant spinor]
Let $\xi^A$ be a constant spinor over $\mathbb{M}$ and consider its decomposition over the basis $(o^A, \iota ^A)$:
$$
\xi^A= \alpha o^A+\beta \iota^A.
$$
Then, for any integer $n$, $\nabla^n \alpha$ and $\nabla^n\beta$ are smooth bounded functions on $\mathbb{M}\backslash \{ \R \times B(0,1)\}$. Furthermore, considering the derivatives in the null directions, the following estimates hold for $\theta \in [c, \pi - c]$ ($c>0$):

\begin{align*}
D\alpha&= D'\alpha=0,&D\beta= D'\beta=0\\
|\delta^n \alpha |&\leq \frac{C}{r^n},&|\delta^n \beta |\leq \frac{C}{r^n},\\
|\delta'^n \alpha |&\leq \frac{C}{r^n},&|\delta'^n \beta |\leq \frac{C}{r^n}.
\end{align*}

\end{lemma}
\begin{proof}
To prove that $\alpha$ and $\beta$ are bounded functions, it suffices to consider the decomposition of the real vector field $\xi{}^A\overline{\xi}{}^{A'}$ in Cartesian coordinates. The time component of the vector field is $|\alpha^2|+|\beta|^2$ and it is constant. As a consequence, $\alpha$ and $\beta$ are smooth bounded functions.

The second step consists in calculating the derivatives of the components in $\xi^{A}$. Since $o^A$ and $\iota^A$ are constant along outgoing and ingoing null rays, the following identities hold:
$$
D\alpha=D'\beta=D\alpha=D'\beta=0.
$$
For the angular derivatives, we have:
\begin{eqnarray*}
\left(\delta \alpha+\alpha \frac{\cot\theta}{2 r \sqrt{2}}-\frac{\beta}{r\sqrt{2}}\right)o^A+\left(\delta \beta - \beta \frac{\cot\theta}{2 r \sqrt{2}}\right)\iota^A=0,\\
\left(\delta'\alpha -\alpha\frac{\cot\theta}{2 r \sqrt{2}}\right)o^A+\left(  \delta' \beta+\beta \frac{\cot\theta}{2 r \sqrt{2}}+\frac{\alpha}{r\sqrt{2}}\right)\iota^A=0.
\end{eqnarray*}
An induction using these recursive relations gives the desired results.
\end{proof}

\subsection{Proof of the decay result}
We consider in this section a spin-$s$ field represented as
\begin{equation}\label{representation1}
\phi_{A\dots F}= \xi_1^{A'}\dots \xi_{2s}^{F'}\nabla_{AA'}\dots\nabla_{FF'} \chi,
\end{equation}
where $\chi$ is a complex scalar Hertz potential satisfying the wave equation
$$
\square \chi=0,
$$
and $\xi_1^{A'}, \dots, \xi_{2s}^{F'}$ are constants spinors.

The purpose of this section is to give a result which is similar to the one obtained for the wave (or spin-0) equation in order to retrieve similar decay estimates as in the pioneering work of Christodoulou-Klainerman \cite{Christodoulou:1990dd}.
\begin{proposition}[Decay estimates for arbitrary spin]\label{prop:decay1}
Let $(k,l,m)$ be a triple of non-negative integers and denote by $n$ their sum. We assume that the Hertz potential is a solution of the Cauchy problem, for $j> 2+2s+n$ and $\delta \notin \mathbb{Z}$
\begin{equation*}\left\{
\begin{array}{l}
\square \chi =0 , \\
\chi|_{t=0}\in H^{j}_\delta(\R^3, \mathbb{C}) , \\
\partial_t \chi|_{t=0}\in H^{j-1}_{\delta-1}(\R^3, \mathbb{C})  .
\end{array}\right.  
\end{equation*}
The norm of the initial data is denoted by $I_{j,\delta}$, 
$$
I_{j,\delta} = \Vert\chi|_{t=0}\Vert_{j, \delta} + \Vert\partial_t\chi|_{t=0}\Vert_{j-1, \delta-1}
$$  
Then, the following inequalities hold, for all $\bm{i}$ in $\{0,2s\}$:
\begin{enumerate}
\item for any $t\geq 0$, $x \in \R^3$, such that $t>3r$, that is to say in the interior region,
$$
|\nabla^n \phi_{A\dots F}|\leq c \lAngle t\rAngle^{\delta-2s-n} I_{n+2s+3, \delta} ,
$$
\item for $\bm{i}$ such that $1+\delta-l-\bm{i}<0$, for any $t\geq 0$, $x \in \R^3$, such that $3r>t>\tfrac{r}{3}$, that is to say in the exterior region,
$$
|D^kD'^l\snabla^m\phi_{\bm{i}}|\leq\frac{c \lAngle u\rAngle ^{1+\delta-\bm{i}-l}}{\lAngle v\rAngle^{1+2s-\bm{i}+k+m}}I_{n+2s+3,\delta} , 
$$
\item for $\bm{i}$ such that $1+\delta-l-\bm{i}>0$, for any $t\geq 0$, $x \in \R^3$, such that $3r>t>\tfrac{r}{3}$, that is to say in the exterior region,
$$
|D^kD'^l\snabla^m\phi_{\bm{i}}|\leq c \lAngle v\rAngle^{\delta -2s-l-k-m} I_{n+2s+3,\delta}.
$$
\end{enumerate}
\end{proposition}
\begin{remark} 
\begin{enumerate}
\item For the spin-$1$ case with $\delta=-1/2$, that is to say for initial data for the Maxwell fields lying in $H_{-5/2}$, which is the case considered in \cite{Christodoulou:1990dd}, one recovers the decay result stated in that paper. For the spin-$2$ case with $\delta=1/2$, that is to say for initial data in $H_{-7/2}$, which is the case considered by Christodoulou-Klainerman, their results are recovered. We need slightly higher regularity though due to our estimates for the wave equation.
\item It should be noted that in the case when the potential does not decay enough (for $\delta>-2s-2$), the decay rates of some components of the field cannot distinguished, and hence the peeling property fails to hold. 
\item The estimates stated in Proposition \ref{prop:decaywave} hold for all weight $\delta$ in $\mathbb{R}$. As a consequence, this should allow us to handle all weights for the initial data of the Hertz potential. Nonetheless, the case of integer weights has been put aside: firstly, the representation of massless free fields, as stated in Proposition~\ref{proprepspingeneral}, does not handle these weights; secondly, the appearance of such weights will lead to an unnecessary lengthy discussion in the proof.
\item The peeling result obtained by Penrose \cite{MR0175590} depends on the assumption that the Hertz potential decays as
$\chi \sim 1/r $
where $r$ is a parameter along the outgoing null rays. For such a decay result to hold, the initial data for the potential  have to lie in $H_{\delta}$ with $\delta<-1$. The peeling result by Mason-Nicolas \cite{Mason:2012jq}, which holds for the spins $1/2$ and $1$ on the Schwarzschild space-time, is for initial data lying in a Sobolev space whose weights are not equally distributed on the components. 
\end{enumerate}
\end{remark}
\begin{proof}
The proof is made by induction on the spin. The result for the spin-0 case is exactly the one obtained for the wave equation and is the base step of the induction.

We now make the following induction hypothesis for spin-$s$, with $s\in \tfrac{1}{2}\mathbb{N}_0$: for any triple $(k,l,m)$, and for any spin-$s$ field represented by
$$
\psi\underbrace{{}_{A\dots F}}_{\mathclap{2s \text{ indices}}}= \xi_1^{A'}\dots \xi_{2s}^{F'}\nabla_{AA'}\dots\nabla_{FF'} \chi,
$$
where $\chi$ is a potential whose initial data lie in $H^j_{\delta}(\R^3,\mathbb{C})\times H^{j-1}_{\delta-1}(\R^3,\mathbb{C})$, for $j>2+2s+n$ with $n=k+l+m$, the estimates stated in the theorem hold.

Let now $(k,l,m)$ be a triple of non-negative integers and consider a spin-$\left(s+1/2\right)$ field written
$$
\phi\underbrace{{}_{A\dots FG}}_{\mathclap{2s +1\text{ indices}}}= \xi_1^{A'}\dots \xi_{2s+1}^{G'}\nabla_{AA'}\dots\nabla_{GG'} \chi,
$$
where $\chi$ is a potential whose initial data are in $H^j_{\delta}(\R^3,\mathbb{C})\times H^{j-1}_{\delta-1}(\R^3,\mathbb{C})$ ($j>3+2s+n$). Consequently, the induction hypothesis is satisfied for the spin-$s$ field 
$$
\psi\underbrace{{}_{B\dots G}}_{\mathclap{2s \text{ indices}}}= \xi_2^{B'}\dots \xi_{2s+1}^{G'}\nabla_{BB'}\dots\nabla_{GG'} \chi,
$$
with the same $\chi$ whose initial data also lies in $H^p_{\delta}(\R^3)\times H^{p-1}_{\delta-1}(\R^3)$ ($p=j-1>2+2s+n$). It remains then to prove that
$$
\phi_{A\dots G}= \xi^{A'}\nabla_{AA'}\psi_{B\dots G}
$$
satisfies the appropriate decay result.

We first consider the interior decay. The result trivially follows from the interior decay result for the wave equation stated in Proposition \ref{prop:decaywave} and Remark \ref{rem:decaywave}. As a consequence, the following relation holds:
\begin{equation*}
|\nabla^n\phi_{A\dots G}|=|\xi^{A'}\nabla_{AA'}\left(\nabla^n\psi _{B\dots G}\right)|,
\end{equation*}
which is a derivative of order $n+1$ of a spinor field of valence $2s$ which satisfies the induction hypothesis. As a consequence, the following decay result is immediate, in the interior region $3t \leq r$:
$$
|\nabla^{n}\phi_{A\dots F}|\leq C\frac{I_{n+2s+3,\delta}}{\lAngle t\rAngle^{-\delta+2s+n+1}},
$$
where $C$ is a constant depending on $n$ and $s$. This closes the induction for the part concerning the interior decay.

Now we consider the problem of the exterior decay, that is to say the decay in the neighborhood of an outgoing light cone:
\begin{equation}\label{descriptionexterior}
 \frac{r}{3} \leq t\leq 3 r \Leftrightarrow |t-r|\leq \frac12 \left|t+r\right|.
\end{equation}
Recall that the components of the spinor $\psi_{B\dots F}$ are defined by
$$
\psi_{\bm{i}}= \underbrace{\iota^B\dots \iota^C}_{\bm{i}}\underbrace {o ^D \dots o^G}_{2s -\bm{i}}\psi_{B\dots G}.
$$
The proof in this region is done by induction as in the first part of the proof.  Let $(k,l,m)$ be a given triple of non negative integers and denote by $n$ their sum. The induction hypothesis is written as follows for spin-$s$:
\begin{quote}
For any spinor field of valence $2s$ $\psi_{B\dots G}$ satisfying:
$$
\psi_{B\dots G}= \xi_2^{B'}\dots \xi_{2s+1}^{G'}\nabla_{BB'}\dots \nabla_{GG'} \chi
$$
where $\chi$ is a complex scalar solution of the massless wave equation whose initial data lies in   $H^j_{\delta}(\R^3, \mathbb{C})\times H^{j-1}_{\delta-1}(\R^3, \mathbb{C})$ ($j>3+2s+n$), the following decay results holds, in the exterior region $\frac{t}{3}\leq r\leq 3t$, for all integer $k,l,m$:
\begin{itemize}
\item for $\bm{i}$ such that $1+\delta-l-\bm{i}<0$:
$$
|D^kD'^l\snabla^m\psi_{\bm{i}}|\leq\frac{c \lAngle u\rAngle ^{\delta+1-\bm{i}-l}}{\lAngle v\rAngle^{1+2s-\bm{i}+k+m}}I_{n+2s+3,\delta}.
$$
\item for $\bm{i}$ such that $1+\delta-l-\bm{i}>0$:
$$
|D^kD'^l\snabla^m\psi_{\bm{i}}|\leq C \lAngle v\rAngle^{\delta -2s-l-k-m} I_{n+2s+3,\delta},
$$
where the constant $C$ depends on the bounds of the exterior domain and the integers $k,l,m$.
\end{itemize}
\end{quote}

There is no need to prove the initial step since it is exactly the result for the standard wave equation. Assume that the induction hypothesis holds for spin-$s$ in $\frac12 \mathbb{N}_0$ and consider the field $\phi_{A\dots G}$ of spin-$\left(s+1/2\right)$ written as:
$$
\phi_{A\dots G}= \xi_1^{A'}\dots \xi_{2s+1}^{G'}\nabla_{AA'}\dots \nabla_{GG'} \chi
$$
where $\chi$ is a complex scalar solution of the massless wave equation whose initial data lies in $H^j_{\delta}(\R^3,\mathbb{C})\times H^{j-1}_{\delta+1}(\R^3,\mathbb{C})$ ($j>2+2s+n$). As a consequence, the spinor:
$$
\psi_{B\dots G}=\xi_2^{B'}\dots \xi_{2s+1}^{G'}\nabla_{BB'}\dots \nabla_{GG'} \chi
$$
is a spinor field of valence $s$ satisfying the requirements of the induction assumption.

To insure the proof of the induction assumption, a relation between the components of $\phi_{A\dots F}$ and the components of the field $\psi_{B\dots G}$ have to established. The components of these fields are related by the following Lemma. 
\begin{lemma}\label{relationscoef} Let $s \in \tfrac{1}{2} \mathbb{N}_0$. 
The components of $\phi_{A\dots G}$ of spin-$\left(s+1/2\right)$ and $\psi_{B \dots G}$ of spin-$s$, related by
$$
\phi_{A\dots G} = \xi^{A'}\nabla_{AA'} \psi_{B\dots G} \text{ where } \xi^{A'} = \alpha o^{A'} + \beta i^{A'} 
$$
are given by the following relations:
\begin{eqnarray}
\phi_0&= &\alpha D\psi_0 +\beta \delta \psi_0-s\beta\frac{\cot \theta}{r\sqrt{2}}\psi_0 ,\\
\phi_{\bm{i}}&= &\alpha\delta'\psi_{\bm{i}-1} +\beta D' \psi_{\bm{i}-1}
+\frac{\alpha}{r\sqrt{2}}\big((s+1-\bm{i}) \cot\theta\psi_{\bm{i}-1}-(2s+1-\bm{i})\psi_{\bm{i}}\big) , \label{relation2} 
\end{eqnarray}
for $\bm{i}>0$.
\end{lemma}
\begin{proof} The proof is carried out by using relations \eqref{basis} and is a basic calculation. We have 
\begin{eqnarray*}
\phi_0&=& o^A \dots o ^G \phi_{A\dots G} , \\
&=&  o^A \dots o ^G\xi^{A'} \nabla_{AA'} \psi_{B\dots G} , \\
&=& \alpha o^B \dots o^G D \psi_{B\dots G} + \beta o^B \dots o^G \delta \psi_{B\dots G} .
 \end{eqnarray*}
Since $D o ^A=0$ and $ \delta o^A=\frac{\cot{\theta}}{2 r \sqrt{2}} o^A$, we have
$$
o^B \dots o^G \delta \psi_{B\dots G}=  \delta \psi_{0} -2s \frac{\cot\theta}{2r \sqrt{2}}\psi_0 , 
$$
and hence
$$
\phi_0= \alpha D \psi_0 +\beta \delta \psi_0 - s \beta \frac{\cot\theta}{r \sqrt{2}}\psi_0.
$$
Consider now $\bm{i}>0$ fixed; we have
\begin{eqnarray*}
\phi_{\bm{i}}&=& \underbrace{\iota^A \dots \iota ^C}_{\bm{i} \text{ times}} \underbrace{o^D \dots o^G}_{2s + 1-\bm{i} }\phi_{A \dots G}\\
&=& \alpha \underbrace{\iota^B \dots \iota ^C}_{\bm{i}-1 } \underbrace{o^D \dots o^G}_{2s + 1-\bm{i} }\delta ' \psi_{B \dots G}+\beta  \underbrace{\iota^B \dots \iota ^C}_{\bm{i}-1 } \underbrace{o^D \dots o^G}_{2s + 1-\bm{i} }D ' \psi_{B \dots G}.
\end{eqnarray*}
Since $\delta' o^A= - \frac{\cot{\theta}}{2 r \sqrt{2}}o^A+\frac{\iota^A}{r\sqrt{2}}$
and $\delta' \iota^A= \frac{\cot{\theta}}{2r \sqrt{2}}\iota^A$ we have
\begin{align*}
 \underbrace{\iota^B \dots \iota ^C}_{\bm{i}-1 }& \underbrace{o^D \dots o^G}_{2s + 1-\bm{i} }\delta ' \psi_{B \dots G}
=
 \delta' \psi_{\bm{i}-1}
 -(\bm{i}-1)\left(\frac{\cot\theta}{2 r \sqrt{2}}  \underbrace{\iota^B \dots \iota ^C}_{\bm{i}-1 } \underbrace{o^D \dots o^G}_{2s + 1-\bm{i} }\psi_{B \dots G}\right)\nonumber\\
& - (2s+1-\bm{i})\left(- \frac{\cot\theta}{2 r \sqrt{2}}  \underbrace{\iota^B \dots \iota ^C}_{\bm{i}-1 } \underbrace{o^D \dots o^G}_{2s + 1-\bm{i}} \psi_{B\dots G} +  \frac{1}{ r \sqrt{2}}  \underbrace{\iota^B \dots \iota ^C}_{\bm{i} } \underbrace{o^D \dots o^G}_{2s-\bm{i}} \psi_{B\dots G} \right).
 \end{align*}
Consequently,  by the relations $D'\iota^A=D'o^A=0$, we get \eqref{relation2}.
\end{proof}

Using Lemma~\ref{relationscoef}, the proof of Proposition \ref{prop:decay1} can be continued. The two cases ($\bm{i}=0$ and $\bm{i}>0$) are treated separately although the method is the same. Here we present the case $\bm{i}=0$, the other case follows similarly.

The expression of the derivative $D^k D'^l\snabla^m\phi_0$ is calculated explicitly, one derivative at a time, using the Leibniz rule:
\begin{gather*}
\snabla^m \phi_0= \sum_{a=0}^m\binom{a}{m}\snabla^{a}\alpha\snabla^{m-a}D\psi_0
+ \sum_{a=0}^m\binom{a}{m}\snabla^{a}\beta\snabla^{m-a}\delta\psi_0\\
-2s\sum_{a+b+c=m}\frac{m!}{a!b!c!}\left(\frac{\partial_{\theta}^b\left(\cot\theta\right)}{2\sqrt{2}r^{b+1}}\right)\snabla^a\beta \partial_{\theta}^b \snabla^c\psi_0
\end{gather*}
since
$$
\snabla^b\cot \theta = \frac{1}{r^b}\frac{\partial^b \cot \theta}{\partial \theta^b}\partial_{\theta}^b, 
$$
the power on the vector field have to be understood as a symmetric tensor product.

We then apply simultaneously the derivatives $D$ and $D'$, using the Leibniz rule again. Notice first that $\snabla^a\alpha$ and $\snabla^a\beta$ depend on $r$ but $r^a\snabla^a \alpha$ and $r^a\snabla^a \beta$ do not, since  both $\alpha$ and $\beta$ are independent both of time and radius, and $r\snabla$ commutes with $D$ and $D'$ (cf. Remark \ref{rem:commute}). We have
\begin{gather*}
D^k D'^l \snabla^m \phi_0 =\\
 \sum_{d=0}^k \sum_{e=0}^l \sum_{a=0}^m\left[\binom{d}{k}\binom{e}{l}\binom{a}{m}\left(r^a\snabla^{a}\alpha \right)\left((-1)^dA_{a+d+e}^{d+e} \right)\right]\left\{\frac{D^{k-d}D'^{l-e}\snabla^{m-a}D\psi_0}{r^{a+d+e}}\right\}\\
 +
  \sum_{d=0}^k \sum_{e=0}^l \sum_{a=0}^m\left[\binom{d}{k}\binom{e}{l}\binom{a}{m}\left(r^a\snabla^{a}\beta \right)\left((-1)^d A_{a+d+e}^{d+e} \right)\right]\left\{\frac{D^{k-d}D'^{l-e}\snabla^{m-a}\delta\psi_0}{r^{a+d+e}}\right\}\\
 +2s \sum_{d=0}^k \sum_{e=0}^l \sum_{a+b+c=m}\left[\binom{d}{k}\binom{e}{l}
 \frac{m!\partial_{\theta}^b\left(\cot\theta\right)}{2\sqrt{2}a!b!c!} \left(r^a \snabla^a\beta\right)\left((-1)^dA_{1+a+d+e}^{d+e} \right)\right]\\
 \times\left\{\frac{D^{k-d}D'^{l-e}\snabla^c\psi_0}{r^{1+a+b+d+e}}\right\},
\end{gather*}
where 
$$
A_{m}^{n} = \frac{n!}{ (m-1)! }.
$$  
The factors in the square brackets are clearly bounded provided that $\theta$ lies in $[c, \pi- c]$ for a given (arbitrarily small) positive constant $c$. The singularity of the tetrad at the axis $\theta = \pm\pi$ prevents from covering the entire interval $[0,\pi]$. The problem can be easily solved by considering another tetrad associated to another spherical coordinate system. There exists consequently a constant $C$ depending on the spin, the $L^\infty$-bounds on the coefficients of the spinor field $\xi^{A'}$ and their derivatives, such that
\begin{gather}
|D^k D'^l \snabla^m \phi_0|\leq
 C \left (\sum_{d=0}^k \sum_{e=0}^l \sum_{a=0}^m\left|\frac{D^{k-d}D'^{l-e}\snabla^{m-a}D\psi_0}{r^{a+d+e}}\right|\right .\nonumber\\
 +  \left . \sum_{d=0}^k \sum_{e=0}^l \sum_{a=0}^m\left|\frac{D^{k-d}D'^{l-e}\snabla^{m-a}\delta\psi_0}{r^{a+d+e}}\right|
 + \sum_{d=0}^k \sum_{e=0}^l \sum_{a+b+c=m}  \left|\frac{D^{k-d}D'^{l-e}\snabla^c\psi_0}{r^{1+a+b+d+e}}\right| \right).\label{Derivativesphi0Estimate1}
\end{gather}
Each of these terms is treated separately.

The first term can be transformed to fit the induction hypothesis using Lemma \ref{commutator}:
\begin{align*}
D^{k-d}D'^{l-e}\snabla^{m-a}D\psi_0={}& D^{k-d+1}D'^{l-e}\snabla^{m-a} \psi_0+D^{k-d}D'^{l-e}\left(\frac{m-a}{r\sqrt{2}}\snabla^{m-a} \psi_0\right)\\
={}&D^{k-d+1}D'^{l-e}\snabla^{m-a} \psi_0\\
+\sum_{f=0}^{k-d}\sum_{g=0}^{l-e}\binom{f}{k-d}\binom{g}{l-e}&\frac{(-1)^f(m-a)(k-d+l-e)!}{r^{1+f+g}\sqrt{2}}D^{k-d-f}D'^{l-e-g}\snabla^{m-a}\psi_0.
\end{align*}

In order to use the decay result stated in the induction hypothesis, the number of derivatives in the ingoing direction has to be taken into account: 
\begin{enumerate}
\item[a)]{if $1+\delta -l<0$, then $$ |D^kD'^l \snabla^m \psi_0| \leq C \lAngle u\rAngle^{1+\delta - l}\lAngle v\rAngle^{-1-2s-k-m} I_{n+2s+3,\delta};$$}
\item[b)]{if $1+\delta -l>0$, then $$ |D^kD'^l \snabla^m \psi_0| \leq C \lAngle v\rAngle^{\delta-2s-k-l-m} I_{n+2s+3,\delta}. $$}
\end{enumerate}
In order to simplify the presentation of the proof, we deal specifically with the sum
$$
A = \sum_{d=0}^k \sum_{e=0}^l \sum_{a=0}^m\left|\frac{D^{k-d+1}D'^{l-e}\snabla^{m-a}\psi_0}{r^{a+d+e}}\right|.
$$
Assume first that $1+\delta-l<0$.

For this case the sum a priori contains terms of both type {\it a} and type {\it b}. We therefore split up the sum over $e$ into two parts corresponding to the different types. Let $e'$ be the largest integer such that $e'\leq l$ and $\delta-l+e'<0$. This means that, if $0\leq e\leq e'-1$, we have $1+\delta-(l-e)<0$ and if $e'\leq e\leq l$ we have $1+\delta-(l-e)>0$.
Using the induction hypothesis, the sum $A$ can then be bounded as follows: there exists a constant $C$ such that
\begin{align*}
A \leq{}& C \sum_{d=0}^k \sum_{a=0}^m \frac{I_{n+2s+4, \delta}}{r^{a+d}}\cdot\Big( 
\sum_{e=0}^{e'-1}\frac{\lAngle u\rAngle^{1+\delta-(l-e)}}{\lAngle v\rAngle^{2+2s+k-d+m-a}r^{e}}
+\sum_{e=e'}^{l}\frac{\lAngle v\rAngle^{\delta-2s-k+d-1-l+e-m+a}}{r^{e}}\Big )
\\
\leq{}& C \frac{\lAngle u\rAngle^{1+\delta-l}I_{n+2s+4, \delta}}{\lAngle v\rAngle^{2s+2+k+m}} \left(
\sum_{d=0}^k \sum_{a=0}^m \left(\frac{\lAngle v\rAngle}{r}\right)^{a+d}\right)\\
&\times \left(
\sum_{e=0}^{e'-1}\left(\frac{\lAngle u\rAngle}{\lAngle v\rAngle}\right)^{e}\left(\frac{\lAngle v\rAngle}{r}\right)^{e}
+\sum_{e=e'}^{l}\left(\frac{\lAngle u\rAngle}{\lAngle v\rAngle}\right)^{-(1+\delta-l)}\left(\frac{\lAngle v\rAngle}{r}\right)^{e}
 \right).
\end{align*}
Since we are considering the exterior region, that is to say the region defined by
$$
\frac{t}{3}\leq r \leq 3t,
$$
the following inequalities hold (assuming also $r>1$, which is not restrictive, when studying the asymptotic behavior),
$$
\frac{\lAngle v\rAngle}{r}\leq \sqrt{17} \text{ and } \frac{\lAngle u\rAngle}{\lAngle v\rAngle} \leq 1. 
$$ 
As a consequence, there exists a constant $C$ depending only on the considered region and of the number of derivatives such that
$$
A \leq C \lAngle u\rAngle^{1+\delta-l}\lAngle v\rAngle^{-1-2s-1-k-m}I_{n+2s+4,\delta}.
$$ 

In the case when $1+\delta-l>0$, all the indices $1+\delta-l+e$ are a fortiori positive and, as a consequence, the induction hypothesis gives immediately: there exists a constant $C$ depending on the number of derivatives and on the bounds of the derivatives of $\alpha$ and $\beta$ such that
$$
A\leq C \lAngle v\rAngle^{\delta -2s-1-k-l-m} I_{n+2s+4, \delta} \sum_{d=0}^k \sum_{e=0}^l \sum_{a=0}^m \left(\frac{\lAngle v\rAngle}{r}\right)^{a+e+d}.
$$ 
There exists consequently, as previously, a constant $C$ depending on the number of derivatives such that
$$
A\leq C \lAngle v\rAngle ^{\delta-2s-1-k-l-m}I_{n+2s+4, \delta}.
$$ 

The other terms in \eqref{Derivativesphi0Estimate1} can be studied in a similar way and details are left to the reader. Collecting all the inequalities obtained for these derivatives, one gets that there exists a constant $C$, depending only on the Sobolev embeddings and the number of derivatives such that
$$
|D^kD'^l\snabla^m \phi_0 |\leq CI_{n+2s+4,\delta}
\left\{\begin{array}{ll}
\lAngle u\rAngle^{1+\delta-l}\lAngle v\rAngle^{-1-2s-1-k-m}&\text{ if }1+\delta-l<0\\
\lAngle v\rAngle^{-2s-1-k-l-m}&\text{ if }1+\delta-l>0.\end{array}
\right.
$$

The other components $\psi_{\bm{i}}$ of the field can be studied in a similar way. The discussion will this time occur on the sign of $1+\delta-l-\bm{i}$. These complementary computations are left to the reader.

We have now proved that the induction hypothesis holds also for $s+1/2$. We can therefore conclude that it holds for all $s\in \tfrac{1}{2}\mathbb{N}_0$.
\end{proof}

\section{Asymptotic behavior of higher spin fields using Hertz potential} \label{sec:mainres} 
This section contains the main result of the paper, which consists, for arbitrary spin, in a decay result for solutions of the Cauchy problem with initial data in weighted Sobolev spaces. This extends the result contained in \cite{Christodoulou:1990dd} for the fixed weight $\delta = -s-3/2$ (for spin-$s$ fields with $s=1,2$) and clarifies the fact that peeling fails for the rapidly decaying components of the field. Furthermore, through Theorem~\ref{th:repthm}, we establish a full correspondence between the decay result of the wave equation and the peeling result for the higher spin fields.  

Consider a free massless spin-$s$ field $\phi_{A \dots F}$, \emph{\emph{i.e.}} a symmetric valence $2s$ spinor field on Minkowski space, which solves
\begin{equation}\label{cauchyprobmassless}
 \left\{
\begin{array}{l}
\nabla^{AA'}\phi_{A\dots F}=0, \\
\phi_{A\dots F}|_{t=0}=\varphi_{A\dots F} \in H^{j}_{\delta}(S_{2s}).
\end{array}\right.
\end{equation}
For $s\geq 1$, this Cauchy problem is consistent only when the geometric constraint 
\begin{equation}\label{geometricconstraint}
D^{AB} \varphi_{ABC\dots F}= (\sdiv_{2s} \varphi)_{C\dots F} = 0
\end{equation}
is satisfied.

In this section, we first investigate which spin-$s$ fields can be represented by a potential of the form
\begin{equation}\label{representation4D}
\phi_{A \dots F} = \nabla_{AA'}\cdots\nabla_{FF'} \widetilde\chi^{A'\dots F'},
\end{equation}
where the Hertz potential satisfies a Cauchy problem
\begin{equation} \label{cauchywavechieq1} \left\{
\begin{array}{l}
\square\chi_{A\dots F}=0, \\
\chi_{A\dots F}|_{t=0}=\xi_{A\dots F}\in H^{j+2s}_{\delta+2s}(S_{2s}),\\
\partial_t\chi_{A\dots F}|_{t=0}=\sqrt{2}\zeta_{A\dots F}\in H^{j+2s-1}_{\delta+2s-1}(S_{2s}),\\
\end{array}\right.
\end{equation}
with 
\begin{equation*}
\chi_{A\dots F}\equiv \tau_{AA'}\cdots \tau_{FF'}\widetilde\chi^{A'\dots F'}.
\end{equation*}
To achieve this, a 3+1 splitting of Eq.~\ref{representation4D} with respect to the Cauchy surfaces $\{t=\text{const.}\}$ is performed in Section~\ref{sec:spacespinorsplit} so that the initial data for the field $\varphi_{A\dots F}$ and for the potential $(\xi_{A\dots F},\zeta_{A\dots F})$ are related through the operator $\mathcal{G}_{2s}$. Theorem \ref{proprepspingeneral} is then used to construct initial data for the Hertz potential and control their Sobolev norms. In Section~\ref{sec:representation}, the uniqueness of the Cauchy problem for higher spin fields ensures that this field is represented by a Hertz potential satisfying the Cauchy problem~\eqref{cauchywavechieq1} with the constructed data. Finally, in Section~\ref{sec:maintheorem}, the asymptotic behavior of the field is derived from the decay result for the scalar wave equation stated in Proposition~\ref{prop:decaywave} through the technical Proposition~\ref{prop:decay1}.

\subsection{Space spinor splitting}\label{sec:spacespinorsplit}
A 3+1 splitting of the potential equation \eqref{representation4D} is now performed. Let $\tau_{AA'}=\sqrt{2}\nabla_{AA'}t$, which is covariantly constant. The operator $D_{AB}=\tau_{(A}{}^{A'}\nabla_{B)A'}$ is valid everywhere and it coincides with the intrinsic derivative on the slices $\{t=\text{const.}\}$. We therefore can consider it as an operator acting both on space-time spinors and on spatial spinors on a time slice. All other operators defined for fields on $\mathbb{R}^3$ also extend in this way to operators on fields on Minkowski space. With this view, we have the decomposition
$\tau_{B}{}^{A'}\nabla_{AA'}=D_{AB}+\tfrac{1}{\sqrt{2}}\epsilon_{AB}\partial_t$.

The evolution equations of the Cauchy problems \eqref{cauchyprobmassless} and \eqref{cauchywavechieq1} can be re-expressed as
\begin{align}
\partial_t \phi_{A\cdots F} &= \sqrt{2} (\scurl_{2s} \phi)_{A\cdots F},\label{evolutionphieq1}\\
\partial_t \partial_t \chi_{A\cdots F} &= - \Delta_{2s}\chi_{A\cdots F}.\label{wavechieq2}
\end{align}

For the spin-$1/2$ case, we immediately get
\begin{align*}
\phi_{A} &= (\scurl_1 \chi)_{A} + \tfrac{1}{\sqrt{2}}\partial_t \chi_{A} 
=(\mathcal{G}_1\scurl_1\chi)_{A}+\tfrac{1}{\sqrt{2}}(\mathcal{G}_1\partial_t\chi)_{A}.
\end{align*}
This simple pattern in fact generalizes to arbitrary spin:
\begin{proposition}\label{prop:splittingpotential}
The equation \eqref{representation4D} together with \eqref{wavechieq2} implies
\begin{equation*}
\phi_{A_1\dots A_{2s}}=(\mathcal{G}_{2s}\scurl_{2s}\chi)_{A_1\dots A_{2s}}
+\tfrac{1}{\sqrt{2}}(\mathcal{G}_{2s}\partial_t\chi)_{A_1\dots A_{2s}}.
\end{equation*}
\end{proposition}
\begin{remark}
 The property $\sdiv_{2s}\mathcal{G}_{2s}=0$ of the operators directly gives that the constraint $(\sdiv_{2s} \phi)_{C\cdots F} = 0$ is automatically satisfied for $s\geq 1$.
\end{remark}
\begin{proof}
See Proposition~\ref{splittingpotentialAppendix} in the appendix for a proof.
\end{proof}

\subsection{Representation by a Hertz potential}\label{sec:representation}
  
  We will investigate under which conditions on the initial datum $\varphi_{A\dots F}$ we can construct initial data for the potential. 
We immediately see from Proposition~\ref{prop:splittingpotential} that if a potential exists, then
$\varphi_{A\dots F}$ has to be in the image of $\mathcal{G}_{2s}$. 
Therefore, we can without loss of generality choose $\xi_{A\dots F}=0$. 
This means that we have to solve the equation $\varphi_{A\dots F}   
=(\mathcal{G}_{2s}\zeta)_{A\dots F}$. Since the integrability condition  \eqref{geometricconstraint} is satisfied, Proposition~\ref{proprepspingeneral} can be used to construct $\zeta_{A\dots F}$.

A key point which will be used later to prove that the field can be represented by a potential
is the uniqueness of the Cauchy problem for first order hyperbolic systems. For such a result, the reader can refer either to \cite{MR0080849} or \cite[Appendix 4]{MR2473363}. That the massless spin-$s$ field equation has a first order symmetric hyperbolic formulation follows immediately from Equation~\eqref{evolutionphieq1} and Lemma~\ref{symbolcurlreal}.

\begin{lemma}\label{uniqueness} 
Consider a spinor field $\varphi_{A\dots F}$ in $L^2_{\text{loc}}(S_{2s})$. Then the Cauchy problem \eqref{cauchyprobmassless} admits at most one solution in $C^0(\mathbb{R}, L^2_{\text{loc}}(S_{2s}))$. 
\end{lemma}
\begin{proof} This lemma is a direct consequence of the energy estimate. 
\end{proof}
\begin{remark} 
This lemma does not state existence of solutions to the Cauchy problem for the massless free 
fields with initial datum in weighted Sobolev spaces. However, one can use Theorem \ref{th:repthm} to obtain existence of solutions of this Cauchy problem from standard existence theorems 
for solutions of the wave equation with initial data in weighted Sobolev spaces.
\end{remark}

We can now use Lemma~\ref{uniqueness} and Proposition~\ref{prop:splittingpotential} to reduce the 
problem of constructing a Hertz potential to the level of initial data.
\begin{lemma}\label{lem:potentialconstrainteq1}Let $j\geq 2$ be an integer and $\varphi_{A\dots F}$ be a spinor field in $H^j_\delta(S_{2s})$ satisfying the constraint equation $(\sdiv_{2s}\varphi)_{C\dots F} = 0$.
Assume that there exist spinor fields $\xi_{A\dots F}\in H^{j+2s}_{\delta+2s}(S_{2s})$ and $\zeta_{A\dots F}\in H^{j+2s-1}_{\delta+2s-1}(S_{2s})$ satisfying
\begin{equation*}
\varphi_{A\dots F}   
=(\mathcal{G}_{2s}\scurl_{2s}\xi)_{A\dots F}
+(\mathcal{G}_{2s}\zeta)_{A\dots F}.
\end{equation*}
Then the only solution to the Cauchy problem \eqref{cauchyprobmassless} for massless free fields is given by
\begin{equation*}
\phi_{A \dots F} = \nabla_{AA'}\cdots\nabla_{FF'} \widetilde\chi^{A'\dots F'},
\end{equation*}
where $\widetilde \chi^{A'\dots F'}$ is obtained through the Cauchy problem \eqref{cauchywavechieq1} for $\chi_{A\dots F}$ with the initial data $(\xi_{A\dots F}, \sqrt{2}\zeta_{A\dots F})$.
\end{lemma}
\begin{proof} Let
\begin{equation*}
\widetilde\phi_{A \dots F} = \nabla_{AA'}\cdots\nabla_{FF'} \widetilde\chi^{A'\dots F'}.
\end{equation*}
It is a simple calculation to check that $\widetilde \phi_{A \dots F}$ satisfies the massless field equation of spin-$s$ (see \cite{MR0175590}, for instance). Furthermore, the restriction of $\phi_{A\dots F}$ and $\widetilde\phi_{A\dots F}$ agree on $\{t=0\}$ and are equal to $\varphi_{A\dots F}$ which lies in $H^j_\delta(S_{2s})$ and consequently in $L^2_{\delta}(S_{2s})$. Using the uniqueness stated in Lemma \ref{uniqueness}, we can conclude that both agree. 
\end{proof}
\begin{theorem}\label{th:repthm} Let $s$ be in $\tfrac{1}{2}\mathbb{N}$, $\delta$ be in $\R\setminus\mathbb{Z}$ and $j\geq2$ an integer. We consider $\varphi_{A\dots F}$ in $H^j_\delta(S_{2s})$ satisfying the constraint equation $D^{AB}\varphi_{A\dots F} =0$.
Then there exists a spinor field $\zeta_{A\dots F}$, solving the equation
$$
\varphi_{A\dots F} = (\mathcal{G}_{2s}\zeta)_{A\dots F}
$$
and satisfying the estimates
\begin{equation*}
\Vert\zeta_{A\dots F}\Vert_{j+2s-1,\delta+2s-1}\leq C \Vert\varphi_{A\dots F}\Vert_{j,\delta}.
\end{equation*}
Furthermore, the unique solution of the Cauchy problem for massless fields \eqref{cauchyprobmassless} with the initial datum $\varphi_{A\dots F}$ is given by
$$
\phi_{A\dots F} = \nabla_{AA'}\dots \nabla_{FF'}\widetilde{\chi}^{A'\dots F'},
$$
where the spinor field $\chi_{A\dots F}$, defined by
$$
\chi_{A\dots F}= \tau_{AA'}\cdots \tau_{FF'}\widetilde\chi^{A'\dots F'},
$$ satisfies the Cauchy problem \eqref{cauchywavechieq1} for the wave equation with initial data $(0,\sqrt{2}\zeta_{A\dots F})$.
\end{theorem}

\begin{proof} This result is a direct consequence of
  Lemma~\ref{lem:potentialconstrainteq1}, and of Propositions~\ref{proprepspin1} and \ref{proprepspingeneral}. 
\end{proof}

\subsection{Decay result for higher spin fields}\label{sec:maintheorem}

The notations adopted in the formulation of the main theorem is consistent with the ones which are adopted in Section~\ref{geometricbackground}.

The following decay result for higher spin fields recovers the decay result obtained by Christo\-doulou and Klainerman in \cite{Christodoulou:1990dd} for the spins 1 (corresponding to the weight $\delta = -5/2$) and 2 (corresponding to the weight  $\delta = -7/2$). The main difference is the regularity of the initial data: Christodoulou and Klainerman rely on weighted Sobolev embedding to obtain their decay result, so that only two derivatives of the initial data are required. The result which is presented here is based on the decay result of solutions of the wave equation stated in Section \ref{sec:estsol} and consequently requires at least three derivatives of the initial data. This restriction can be removed as soon as a decay result for the wave equation for initial data with arbitrary decay at spatial infinity, and regularity $H^2\times H^1$, is established. The following theorem will then also hold with one derivative less in the norms.

\begin{theorem}\label{maintheorem}
Let $s$ be in $\frac12\mathbb{N}$, $\delta$ in $\mathbb{R}\backslash \mathbb{Z}$, $j> 2$ an integer and consider the Cauchy problem for the massless free spin-$s$ fields
$$\left\{
\begin{array}{l}
\nabla^{AA'}\phi_{A\dots F}=0\\
\phi_{A\dots F}|_{t=0} = \varphi_{A\dots F} \in H^j_{\delta}(S_{2s})\\
D^{AB} \varphi_{A\dots F}=0.
\end{array}\right.
$$
We finally consider three nonnegative integers $k,l,m$ whose sum is denoted by $n\leq j-3$.
The following inequalities hold, for all $\bm{i}$ in $\{0,\dots,2s\}$: there exists a constant $C$ depending only on a choice of a constant dyad and of $k,l,m$ such that:
\begin{enumerate}
\item for any $t\geq 0$, $x \in \R^3$ such that $t>3r$, that is to say, in the interior region,
$$
|\nabla^n \phi_{A\dots F}|\leq C \lAngle t\rAngle^{\delta-n} \Vert\varphi_{A\dots F}\Vert_{3+n, \delta};
$$
\item for $\bm{i}$ such that $1+2s+\delta-l-\bm{i}<0$, for any $t\geq 0$, $x \in \R^3$,  such that $3r>t>\tfrac{r}{3}$, that is to say in the exterior region,
$$
|D^kD'^l\snabla^m\phi_{\bm{i}}|\leq\frac{C \lAngle u\rAngle ^{1+\delta+2s-l-\bm{i}}}{\lAngle v\rAngle^{1+2s-\bm{i}+k+m}}\Vert\varphi_{A\dots F}\Vert_{3+n, \delta};
$$
\item for $\bm{i}$ such that $1+2s+\delta-l-\bm{i}>0$, for any $t\geq 0$, $x \in \R^3$, such that $3r>t>\tfrac{r}{3}$, that is to say in the exterior region,
$$
|D^kD'^l\snabla^m\phi_{\bm{i}}|\leq C \lAngle v\rAngle^{\delta-n} \Vert\varphi_{A\dots F}\Vert_{3+n, \delta}.
$$
\end{enumerate}
\end{theorem}

\begin{proof} Let $s$ and $\delta$ be such as in the theorem and consider $\varphi_{A\dots F}$ a initial datum in $H^j_{\delta}(S_{2s})$ satisfying the constraints equation
$$
D^{AB}\varphi_{A\dots F}=0. 
$$
The initial datum $\varphi_{A\dots F}$ satisfies the assumptions stated in Theorem~\ref{th:repthm}, so that there exists a potential $\tilde{\chi}^{A'\dots F'}$ of order $2s$ such that the solution of the Cauchy problem with initial datum $\varphi_{A\dots F}$ is given by:
$$
\phi_{A\dots F} = \nabla_{AA'}\dots \nabla_{FF'} \tilde{\chi}^{A'\dots F'}
$$
and $\chi_{A\dots F} = \tau_{AA'}\dots \tau_{FF'} \tilde \chi^{A'\dots F'}$ satisfies the Cauchy problem
$$
\left\{
\begin{array}{l}
\square \chi_{A\dots F}=0\\
\chi_{A\dots F}|_{t=0}\in H^{j+2s}_{\delta+2s}(S_{2s})\\
\partial_t \chi_{A\dots F}|_{t=0}\in H^{j+2s-1}_{\delta+2s-1}(S_{2s}). 
\end{array}
\right.
$$
Furthermore, the norm of the potential is controlled by the norm of the initial data
$$
\Vert\chi_{A\dots F} \Vert_{3+n+2s,\delta+2s} \leq C\Vert\varphi_{A\dots F}\Vert_{3+n,\delta}.
$$

A constant dyad $(e^A_0, e^A_1)$ on the Minkowski space is chosen. The components of the field $\chi_{A\dots F}$ are then of the form $\chi \xi^1_{A}\dots \xi^{2s}_{F}$,
where the constant spinor $\xi_{A}^{i}$ (for $i\in\{1,\dots, 2s\}$) belongs to $\{e^0_{A}, e^{1}_A\}$ and $\chi$ is a complex function satisfying a Cauchy problem of the form
$$
\left\{
\begin{array}{l}
\square \chi=0 , \\
 \chi|_{t=0}\in H^{j+2s}_{\delta+2s}(\R^3, \mathbb{C}) , \\
\partial_t \chi|_{t=0} \in H^{j+2s-1}_{\delta+2s-1}(\R^3, \mathbb{C}). 
\end{array}
\right.
$$

Proposition~\ref{prop:decay1} can then be used, on each of the components of the field. All these components decay exactly in the same way and, consequently, the field $\phi_{A\dots F}$ decays exactly as the field under consideration in Proposition~\ref{prop:decay1}.
\end{proof}

\subsection*{Acknowledgments}
We thank Dietrich H\"afner, Jean-Philippe Nicolas and Lionel Mason for
helpful discussions.

\appendix

\section{Algebraic properties of the fundamental operators} \label{sec:algprop} 
To prove Proposition~\ref{prop:splittingpotential}, we need the following relation
\begin{lemma}\label{Gcurlexpansion}
The operators $\mathcal{G}_k$ and $\scurl_k$ commute and we have
\begin{eqnarray*}
(\mathcal{G}_k\scurl_k\phi)_{A_1\dots A_k}&=&(\scurl_k\mathcal{G}_k\phi)_{A_1\dots A_k}\\
&=&
\sum_{n=0}^{\lfloor \tfrac{k}{2}\rfloor}\binom{k}{2n}(-2)^{-n}\underbrace{D_{(A_1}{}^{B_1}\cdots D_{A_{k-2n}}{}^{B_{k-2n}}}_{k-2n}(\Delta^n_{k}\phi)_{A_{k-2n+1}\dots A_{k})B_1\dots B_{k-2n}}.
\end{eqnarray*}
\end{lemma}
\begin{proof}
We begin by proving that $(\mathcal{G}_k\scurl_k\phi)_{A_1\dots A_k}$ has the desired form.
By partially expanding the symmetry of the $\scurl_k$ operator we get
\begin{align*}
&\underbrace{D_{(A_1}{}^{B_1}\cdots D_{A_{k-2n-1}}{}^{B_{k-2n-1}}}_{k-2n-1}(\Delta^n_{k}\scurl_k\phi)_{A_{k-2n}\dots A_{k})B_1\dots B_{k-2n-1}}\nonumber\\
={}&\frac{2n+1}{k}\underbrace{D_{(A_1}{}^{B_1}\cdots D_{A_{k-2n}}{}^{B_{k-2n}}}_{k-2n}(\Delta^n_{k}\phi)_{A_{k-2n+1}\dots A_{k})B_1\dots B_{k-2n}}\nonumber\\
&+\frac{k-2n-1}{k}\underbrace{D_{(A_1}{}^{B_1}\cdots D_{A_{k-2n-1}}{}^{B_{k-2n-1}}}_{k-2n-1}D_{|B_{k-2n-1}|}{}^{B_{k-2n}}(\Delta^n_{k}\phi)_{A_{k-2n}\dots A_{k})B_1\dots B_{k-2n}}\nonumber\\
={}&\frac{2n+1}{k}\underbrace{D_{(A_1}{}^{B_1}\cdots D_{A_{k-2n}}{}^{B_{k-2n}}}_{k-2n}(\Delta^n_{k}\phi)_{A_{k-2n+1}\dots A_{k})B_1\dots B_{k-2n}}\nonumber\\
&-\frac{k-2n-1}{2k}\underbrace{D_{(A_1}{}^{B_1}\cdots D_{A_{k-2n-2}}{}^{B_{k-2n-2}}}_{k-2n-2}(\Delta^{n+1}_{k}\phi)_{A_{k-2n-1}\dots A_{k})B_1\dots B_{k-2n-2}}.
\end{align*}
Where we used $D_{A}{}^{C}D_{BC}=-\tfrac{1}{2}\epsilon_{AB}\Delta$ in the last step. We therefore get
\begin{align}
(\mathcal{G}_k&\scurl_k\phi)_{A_1\dots A_k}\nonumber\\
={}&
\sum_{n=0}^{\lfloor \tfrac{k-1}{2}\rfloor}\binom{k-1}{2n}(-2)^{-n}\underbrace{D_{(A_1}{}^{B_1}\cdots D_{A_{k-2n}}{}^{B_{k-2n}}}_{k-2n}(\Delta^n_{k}\phi)_{A_{k-2n+1}\dots A_{k})B_1\dots B_{k-2n}}\nonumber\\
&+\sum_{n=0}^{\lfloor \tfrac{k-1}{2}\rfloor}\binom{k-1}{2n+1}(-2)^{1-n}\underbrace{D_{(A_1}{}^{B_1}\cdots D_{A_{k-2n-2}}{}^{B_{k-2n-2}}}_{k-2n-2}(\Delta^{n+1}_{k}\phi)_{A_{k-2n-1}\dots A_{k})B_1\dots B_{k-2n-2}}\nonumber\\
={}&
\sum_{n=0}^{\lfloor \tfrac{k-1}{2}\rfloor}\binom{k-1}{2n}(-2)^{-n}\underbrace{D_{(A_1}{}^{B_1}\cdots D_{A_{k-2n}}{}^{B_{k-2n}}}_{k-2n}(\Delta^n_{k}\phi)_{A_{k-2n+1}\dots A_{k})B_1\dots B_{k-2n}}\nonumber\\
&+\sum_{n=1}^{\lfloor \tfrac{k+1}{2}\rfloor}\binom{k-1}{2n-1}(-2)^{-n}\underbrace{D_{(A_1}{}^{B_1}\cdots D_{A_{k-2n}}{}^{B_{k-2n}}}_{k-2n}(\Delta^{n}_{k}\phi)_{A_{k-2n+1}\dots A_{k})B_1\dots B_{k-2n}}.\label{Gkckexpr1}
\end{align}
Where we just changed $n\rightarrow n-1$ in the last sum. The Pascal triangle gives the algebraic identity
\begin{align*}
\sum_{n=0}^{\lfloor \tfrac{k-1}{2}\rfloor}\binom{k-1}{2n}A^{k-n}B^{n}
+\sum_{n=1}^{\lfloor \tfrac{k+1}{2}\rfloor}\binom{k-1}{2n-1}A^{k-n}B^{n}
=\sum_{n=0}^{\lfloor \tfrac{k}{2}\rfloor}\binom{k}{2n}A^{k-n}B^{n},
\end{align*}
which in turn gives the desired form for $(\mathcal{G}_k\scurl_k\phi)_{A_1\dots A_k}$. To handle $(\scurl_k\mathcal{G}_k\phi)_{A_1\dots A_k}$ we partially expand the symmetry in the following expression
\begin{align*}
&D_{A_1}{}^{C}\underbrace{D_{(A_2}{}^{B_2}\cdots D_{A_k-2n}{}^{B_{k-2n}}}_{k-2n-1}(\Delta_k^n\phi)_{A_{k-2n+1}\dots A_k C)B_2\dots B_{k-2n}}\nonumber\\
={}&\frac{k-2n-1}{k}D_{A_1}{}^{C}D_{C}{}^{B_2}\underbrace{D_{(A_2}{}^{B_3}\cdots D_{A_k-2n-1}{}^{B_{k-2n}}}_{k-2n-2}(\Delta_k^n\phi)_{A_{k-2n}\dots A_k)B_2\dots B_{k-2n}}\nonumber\\
&+\frac{2n+1}{k}D_{A_1}{}^{C}\underbrace{D_{(A_2}{}^{B_2}\cdots D_{A_k-2n}{}^{B_{k-2n}}}_{k-2n-1}(\Delta_k^n\phi)_{A_{k-2n+1}\dots A_k)C B_2\dots B_{k-2n}}\nonumber\\
={}&-\frac{k-2n-1}{2k}\underbrace{D_{(A_2}{}^{B_3}\cdots D_{A_k-2n-1}{}^{B_{k-2n}}}_{k-2n-2}(\Delta_k^{n+1}\phi)_{A_{k-2n}\dots A_k)A_1 B_3\dots B_{k-2n}}\nonumber\\
&+\frac{2n+1}{k}D_{A_1}{}^{B_1}\underbrace{D_{(A_2}{}^{B_2}\cdots D_{A_k-2n}{}^{B_{k-2n}}}_{k-2n-1}(\Delta_k^n\phi)_{A_{k-2n+1}\dots A_k)B_1\dots B_{k-2n}}.
\end{align*}
Where we in the last step again used  $D_{A}{}^{C}D_{BC}=-\tfrac{1}{2}\epsilon_{AB}\Delta$. Using this in the definition of $\mathcal{G}_k$ yields
\begin{align*}
D_{A_1}{}^C&(\mathcal{G}_k\phi)_{CA_2\dots A_k}\nonumber\\
={}&\sum_{n=0}^{\lfloor \tfrac{k-1}{2}\rfloor}\binom{k-1}{2n+1}(-2)^{1-n}\underbrace{D_{(A_2}{}^{B_3}\cdots D_{A_{k-2n-1}}{}^{B_{k-2n-2}}}_{k-2n-2}(\Delta^{n+1}_{k}\phi)_{A_{k-2n}\dots A_{k})A_1B_3\dots B_{k-2n}}\nonumber\\
&+\sum_{n=0}^{\lfloor \tfrac{k-1}{2}\rfloor}\binom{k-1}{2n}(-2)^{1-n}D_{A_1}{}^{B_1}\underbrace{D_{(A_2}{}^{B_2}\cdots D_{A_{k-2n}}{}^{B_{k-2n}}}_{k-2n-1}(\Delta^{n}_{k}\phi)_{A_{k-2n+1}\dots A_{k})B_1\dots B_{k-2n}}.
\end{align*}
After symmetrization we get that $(\scurl_k\mathcal{G}_k\phi)_{A_1\dots A_k}$ has an expansion identical to the one in the first equation in \eqref{Gkckexpr1}. This gives the desired result.
\end{proof}

\begin{proposition}\label{splittingpotentialAppendix}
The equation \eqref{representation4D} together with \eqref{wavechieq2} implies
\begin{equation*}
\phi_{A_1\dots A_{2s}}=(\mathcal{G}_{2s}\scurl_{2s}\chi)_{A_1\dots A_{2s}}
+\tfrac{1}{\sqrt{2}}(\mathcal{G}_{2s}\partial_t\chi)_{A_1\dots A_{2s}}.
\end{equation*}
\end{proposition}
\begin{proof}
Using $\tau_{B}{}^{A'}\nabla_{AA'}=D_{AB}+\tfrac{1}{\sqrt{2}}\epsilon_{AB}\partial_t$ we can write the potential equation in terms of $D_{AB}$ and $\partial_t$. We have 
\begin{align*}
\phi_{A_1\dots A_{2s}}
={}&\underbrace{\nabla_{A_1A'_1}\cdots \nabla_{A_{2s}A'_{2s}}}_{2s}\widetilde\chi^{A'_1\dots A'_{2s}}
=\underbrace{\tau^{B_1A'_1}\nabla_{A_1A'_1}\cdots \tau^{B_{2s}A'_{2s}}\nabla_{A_{2s}A'_{2s}}}_{2s}\chi_{B_1\dots B_{2s}}\nonumber\\
={}&\underbrace{(D_{A_1}{}^{B_1}+\tfrac{1}{\sqrt{2}}\epsilon_{A_1}{}^{B_1}\partial_t)\cdots (D_{A_{2s}}{}^{B_{2s}}+\tfrac{1}{\sqrt{2}}\epsilon_{A_{2s}}{}^{B_{2s}}\partial_t)}_{2s}\chi_{B_1\dots B_{2s}}\nonumber\\
={}&\sum_{n=0}^{2s}\binom{2s}{n}2^{-n/2}\underbrace{D_{(A_1}{}^{B_1}\cdots D_{A_{2s-n}}{}^{B_{2s-n}}}_{2s-n}\partial_t^n\chi_{A_{2s-n+1}\dots A_{2s})B_1\dots B_{2s-n}}.
\end{align*}
We can now use \eqref{wavechieq2} to eliminate all higher order time derivatives. This gives 
\begin{align*}
&\negthickspace\negthickspace\negthickspace\phi_{A_1\dots A_{2s}}\nonumber\\
={}&\sum_{n=0}^{\lfloor s\rfloor}\binom{2s}{2n}2^{-n}\underbrace{D_{(A_1}{}^{B_1}\cdots D_{A_{2s-2n}}{}^{B_{2s-2n}}}_{2s-2n}\partial_t^{2n}\chi_{A_{2s-2n+1}\dots A_{2s})B_1\dots B_{2s-2n}}\nonumber\\
&+\sum_{n=0}^{\lfloor s-\tfrac{1}{2}\rfloor}\binom{2s}{2n+1}2^{-n-1/2}\underbrace{D_{(A_1}{}^{B_1}\cdots D_{A_{2s-2n-1}}{}^{B_{2s-2n-1}}}_{2s-2n-1}\partial_t^{2n+1}\chi_{A_{2s-2n}\dots A_{2s})B_1\dots B_{2s-2n-1}}\nonumber\\
={}&\sum_{n=0}^{\lfloor s\rfloor}\binom{2s}{2n}(-2)^{-n}\underbrace{D_{(A_1}{}^{B_1}\cdots D_{A_{2s-2n}}{}^{B_{2s-2n}}}_{2s-2n}(\Delta^n_{2s}\chi)_{A_{2s-2n+1}\dots A_{2s})B_1\dots B_{2s-2n}}\nonumber\\
&+\tfrac{1}{\sqrt{2}}\sum_{n=0}^{\lfloor s-\tfrac{1}{2}\rfloor}\binom{2s}{2n+1}(-2)^{-n}\underbrace{D_{(A_1}{}^{B_1}\cdots D_{A_{2s-2n-1}}{}^{B_{2s-2n-1}}}_{2s-2n-1}(\Delta^n_{2s}\partial_t\chi)_{A_{2s-2n}\dots A_{2s})B_1\dots B_{2s-2n-1}}\nonumber\\
={}&(\mathcal{G}_{2s}\scurl_{2s}\chi)_{A_1\dots A_{2s}}
+\tfrac{1}{\sqrt{2}}(\mathcal{G}_{2s}\partial_t\chi)_{A_1\dots A_{2s}}.
\end{align*}
In the last step we used the definition of $\mathcal{G}_{2s}$ and Lemma~\ref{Gcurlexpansion}. In fact we have defined $\mathcal{G}_k$ to match the $\partial_t$ part of this expression.
\end{proof}

\begin{proposition}\label{divGproperty}
For $k\geq 2$, the operators $\mathcal{G}_k$ have the properties $\sdiv_k\mathcal{G}_k=0$ and $\mathcal{G}_k\stwist_{k-2}=0$.
\end{proposition}
\begin{proof}
First we prove that $\sdiv_k\mathcal{G}_k=0$.
By partially expanding the symmetrization in the definition of $\mathcal{G}_k$ and restricting the summation to non-vanishing terms we get
\begin{align*}(\mathcal{G}_k&\phi)_{A_1\dots A_k}\nonumber\\
={}& \sum_{n=0}^{\mathclap{\left\lfloor\tfrac{k-3}{2}\right\rfloor}}
\binom{k-2}{2n+1}(-2)^{-n} D_{A_1}{}^{B_1}D_{A_2}{}^{B_2}\underbrace{D_{(A_3}{}^{B_3}\cdots D_{A_{k-2n-1}}{}^{B_{k-2n-1}}}_{k-2n-3} (\Delta^n_k\phi)_{A_{k-2n}\dots A_k)B_1\dots B_{k-2n-1}}\nonumber\\
&+ \sum_{n=0}^{\mathclap{\left\lfloor\tfrac{k-1}{2}\right\rfloor}}
\binom{k-2}{2n}(-2)^{-n} D_{A_1}{}^{B_2}\underbrace{D_{(A_3}{}^{B_3}\cdots D_{A_{k-2n}}{}^{B_{k-2n}}}_{k-2n-2} (\Delta^n_k\phi)_{A_{k-2n+1}\dots A_k)A_2B_2\dots B_{k-2n}}\nonumber\\
&+ \sum_{n=0}^{\mathclap{\left\lfloor\tfrac{k-1}{2}\right\rfloor}}
\binom{k-2}{2n}(-2)^{-n} D_{A_2}{}^{B_2}\underbrace{D_{(A_3}{}^{B_3}\cdots D_{A_{k-2n}}{}^{B_{k-2n}}}_{k-2n-2} (\Delta^n_k\phi)_{A_{k-2n+1}\dots A_k)A_1B_2\dots B_{k-2n}}\nonumber\\
&+ \sum_{n=1}^{\mathclap{\left\lfloor\tfrac{k-1}{2}\right\rfloor}}
\binom{k-2}{2n-1}(-2)^{-n} \underbrace{D_{(A_3}{}^{B_3}\cdots D_{A_{k-2n+1}}{}^{B_{k-2n+1}}}_{k-2n-1} (\Delta^n_k\phi)_{A_{k-2n+2}\dots A_k)A_1A_2B_3\dots B_{k-2n+1}}.
\end{align*}
Using $D_{A}{}^{C}D_{BC}=-\tfrac{1}{2}\epsilon_{AB}\Delta$ we get $D^{A_1A_2}D_{A_1}{}^{B_1}D_{A_2}{}^{B_2}=\tfrac{1}{2}D^{B_1B_2}\Delta$, $D^{A_1(A_2}D_{A_1}{}^{B_2)}=0$ and
\begin{align*}
(\sdiv_k&\mathcal{G}_k\phi)_{A_1\dots A_k}\nonumber\\
={}& -\sum_{n=0}^{\mathclap{\left\lfloor\tfrac{k-3}{2}\right\rfloor}}
\binom{k-2}{2n+1}(-2)^{-n-1} D^{B_1B_2}\underbrace{D_{(A_3}{}^{B_3}\cdots D_{A_{k-2n-1}}{}^{B_{k-2n-1}}}_{k-2n-3} (\Delta^{n+1}_k\phi)_{A_{k-2n}\dots A_k)B_1\dots B_{k-2n-1}}\nonumber\\
&+ \sum_{n=1}^{\mathclap{\left\lfloor\tfrac{k-1}{2}\right\rfloor}}
\binom{k-2}{2n-1}(-2)^{-n}D^{B_1B_2} \underbrace{D_{(A_3}{}^{B_3}\cdots D_{A_{k-2n+1}}{}^{B_{k-2n+1}}}_{k-2n-1} (\Delta^n_k\phi)_{A_{k-2n+2}\dots A_k)B_1\dots B_{k-2n+1}}.
\end{align*}
The first sum is identical to the second sum after a variable change $n\rightarrow n-1$, hence $\sdiv_k\mathcal{G}_k=0$.

Now, we turn to the proof of $\mathcal{G}_k\stwist_{k-2}=0$. Partial expansion of the symmetrization in the definition of $\stwist_{k-2}$ gives
\begin{align*}
&\underbrace{D_{(A_1}^{B_1}\cdots D_{A_{k-2n-1}}{}^{B_{k-2n-1}}}_{k-2n-1}(\Delta^n_k\stwist_{k-2}\phi)_{A_{k-2n}\dots A_{k})B_1\dots B_{k-2n-1}}\nonumber\\
={}&\frac{(k-2n-1)(k-2n-2)}{k(k-1)}D_{B_1B_2}\underbrace{D_{(A_1}{}^{B_1}\cdots D_{A_{k-2n-1}}{}^{B_{k-2n-1}}}_{k-2n-1}(\Delta^n_{k-2}\phi)_{A_{k-2n}\dots A_{k})B_3\dots B_{k-2n-1}}\nonumber\\
&+\frac{2(k-2n-1)(2n+1)}{k(k-1)}D_{B_1(A_k}\underbrace{D_{A_1}{}^{B_1}\cdots D_{A_{k-2n-1}}{}^{B_{k-2n-1}}}_{k-2n-1}(\Delta^n_{k-2}\phi)_{A_{k-2n}\dots A_{k-1})B_2\dots B_{k-2n-1}}\nonumber\\
&+\frac{2n(2n+1)}{k(k-1)}D_{(A_{k-1}A_k}\underbrace{D_{A_1}{}^{B_1}\cdots D_{A_{k-2n-1}}{}^{B_{k-2n-1}}}_{k-2n-1}(\Delta^n_{k-2}\phi)_{A_{k-2n}\dots A_{k-2})B_1\dots B_{k-2n-1}}\nonumber\\
={}&\frac{(k-2n-1)(k-2n-2)}{2k(k-1)}D_{(A_1A_2}\underbrace{D_{A_3}{}^{B_3}\cdots D_{A_{k-2n-1}}{}^{B_{k-2n-1}}}_{k-2n-3}(\Delta^{n+1}_{k-2}\phi)_{A_{k-2n}\dots A_{k})B_3\dots B_{k-2n-1}}\nonumber\\
&+\frac{2n(2n+1)}{k(k-1)}D_{(A_{k-1}A_k}\underbrace{D_{A_1}{}^{B_1}\cdots D_{A_{k-2n-1}}{}^{B_{k-2n-1}}}_{k-2n-1}(\Delta^n_{k-2}\phi)_{A_{k-2n}\dots A_{k-2})B_1\dots B_{k-2n-1}}.
\end{align*}
Where we again used $D^{A_1A_2}D_{A_1}{}^{B_1}D_{A_2}{}^{B_2}=\tfrac{1}{2}D^{B_1B_2}\Delta$ and $D^{A_1(A_2}D_{A_1}{}^{B_2)}=0$.
We therefore get
\begin{align*}
(\mathcal{G}_k&\stwist_{k-2}\phi)_{A_1\dots A_k}\nonumber\\
={}&-\sum_{n=0}^{\mathclap{\left\lfloor\tfrac{k-3}{2}\right\rfloor}}\binom{k-2}{2n+1}(-2)^{1-n} D_{(A_1A_2}\underbrace{D_{A_3}{}^{B_3}\cdots D_{A_{k-2n-1}}{}^{B_{k-2n-1}}}_{k-2n-3}(\Delta^{n+1}_{k-2}\phi)_{A_{k-2n}\dots A_{k})B_3\dots B_{k-2n-1}}\nonumber\\
&\negthickspace\negmedspace{}+\sum_{n=1}^{\mathclap{\left\lfloor\tfrac{k-1}{2}\right\rfloor}}\binom{k-2}{2n-1}(-2)^{-n}D_{(A_{k-1}A_k}\underbrace{D_{A_1}{}^{B_1}\cdots D_{A_{k-2n-1}}{}^{B_{k-2n-1}}}_{k-2n-1}(\Delta^n_{k-2}\phi)_{A_{k-2n}\dots A_{k-2})B_1\dots B_{k-2n-1}}.
\end{align*}
The first sum is identical to the second sum after a variable change $n\rightarrow n-1$, hence $\mathcal{G}_k\stwist_{k-2}=0$.
\end{proof}

To use elliptic theory, we need well behaved elliptic operators. $\mathcal{G}_k$ is in general not elliptic but, through the following lemma, it can related to some power of the Laplacian -- which of course is elliptic.
\begin{lemma}\label{lapspinor}
The formulae \eqref{LaplacianAsGodd} and \eqref{LaplacianAsGeven} hold, that is to say
\begin{align*}
(\Delta^{k}_{2k}\phi)_{A_1\dots A_{2k}}={}&(\stwist_{2k-2}\mathcal{F}_{2k-2}\sdiv_{2k}\phi)_{A_1\dots A_{2k}}-(-2)^{1-k}(\mathcal{G}_{2k}\scurl_{2k}\phi)_{A_1\dots A_{2k}},\\
 (\Delta^{k}_{2k+1}\phi)_{A_{1}\dots A_{2k+1}}={}&
(\stwist_{2k-1}\mathcal{F}_{2k-1}\sdiv_{2k+1}\phi)_{A_1\dots A_{2k+1}}
+(-2)^{-k}(\mathcal{G}_{2k+1}\phi)_{A_1\dots A_{2k+1}} .
\end{align*}
\end{lemma}
\begin{proof}
For both formulae, we will use the following help quantity for the spin-$(k+j/2)$ case
\begin{align*}
I^{j,k}_m\equiv{}&\negmedspace\sum_{n=m}^{k-1}\binom{2k+j}{2n+j-2m}(-2)^{-n} \underbrace{D_{(A_1}{}^{B_1}\cdots D_{A_{2k-2n}}{}^{B_{2k-2n}}}_{2k-2n} (\Delta^{n}_{2k+j}\phi)_{A_{2k-2n+1}\dots A_{2k+j})B_1\dots B_{2k-2n}}
\end{align*}
Multiplying $D_{A_1}{}^{B_1}D_{C_1}{}^{B_2}\phi_{A_3\dots A_{k}C_2B_2}$ with 
$\epsilon_{A_2}{}^{C_1}\epsilon_{B_1}{}^{C_2}=\epsilon_{A_2B_1}\epsilon^{C_1C_2}+\epsilon_{A_2}{}^{C_2}\epsilon_{B_2}{}^{C_1}$ and using $D_{A}{}^{C}D_{BC}=-\tfrac{1}{2}\epsilon_{AB}\Delta$, we get
\begin{align*}
D_{A_1}{}^{B_1}D_{A_2}{}^{B_2}\phi_{A_3\dots A_{k}B_1B_2}
={}&
-\tfrac{1}{2}(\Delta_k\phi)_{A_1\dots A_k}+
D_{A_1 A_2}(\sdiv_k\phi)_{A_3\dots A_{k}}.
\end{align*}
Using this in the definition of $I^{j,k}_{m}$ gives
\begin{align}
I^{j,k}_m
={}&\sum_{n=m}^{k-1}\binom{2k+j}{2n+j-2m}(-2)^{-n-1}\nonumber\\
&\quad\times \underbrace{D_{(A_1}{}^{B_1}\cdots D_{A_{2k-2n-2}}{}^{B_{2k-2n-2}}}_{2k-2n-2} (\Delta^{n+1}_{2k+j}\phi)_{A_{2k-2n-1}\dots A_{2k+j})B_1\dots B_{2k-2n-2}}\nonumber\\
&+\sum_{n=m}^{k-1}\binom{2k+j}{2n+j-2m}(-2)^{-n} \nonumber\\
&\quad\times D_{(A_1A_2}\underbrace{D_{A_3}{}^{B_3}\cdots D_{A_{2k-2n}}{}^{B_{2k-2n}}}_{2k-2n-2} (\sdiv_{2k+j}\Delta^{n}_{2k+j}\phi)_{A_{2k-2n+1}\dots A_{2k+j})B_3\dots B_{2k-2n}}\nonumber\\
={}&I^{j,k}_{m+1}+\binom{2k+j}{2k+j-2m}(-2)^{-k}(\Delta^{k}_{2k+j}\phi)_{A_1\dots A_{2k+j}}
+\sum_{n=m}^{k-1}\binom{2k+j}{2n+j-2m}(-2)^{-n} \nonumber\\
&\quad\times D_{(A_1A_2}\underbrace{D_{A_3}{}^{B_3}\cdots D_{A_{2k-2n}}{}^{B_{2k-2n}}}_{2k-2n-2} (\sdiv_{2k+j}\Delta^{n}_{2k+j}\phi)_{A_{2k-2n+1}\dots A_{2k+j})B_3\dots B_{2k-2n}}.\label{IRecursion}
\end{align}
Here, we have changed $n\rightarrow n-1$ in the first sum, and identified that as $I^{j,k}_{m+1}$ plus the term where $n=k$, which gives us the $\Delta^{k}$-term. We can easily solve the recursion \eqref{IRecursion} and get
\begin{align*}
I^{j,k}_0
={}&\sum_{m=0}^{k-1}\binom{2k+j}{2k+j-2m}(-2)^{-k}(\Delta^{k}_{2k+j}\phi)_{A_1\dots A_{2k+j}}
+\sum_{m=0}^{k-1}\sum_{n=m}^{k-1}\binom{2k+j}{2n+j-2m}(-2)^{-n}\nonumber\\ 
&\quad \times D_{(A_1A_2}\underbrace{D_{A_3}{}^{B_3}\cdots D_{A_{2k-2n}}{}^{B_{2k-2n}}}_{2k-2n-2} (\sdiv_{2k+j}\Delta^{n}_{2k+j}\phi)_{A_{2k-2n+1}\dots A_{2k+j})B_3\dots B_{2k-2n}}\nonumber\\
={}&\sum_{m=0}^{k-1}\binom{2k+j}{2m}(-2)^{-k}(\Delta^{k}_{2k+j}\phi)_{A_1\dots A_{2k+j}}
+\sum_{n=0}^{k-1}\sum_{m=0}^{k-1-n}\binom{2k+j}{2n+2m+2}(-2)^{n+1-k}\nonumber\\ 
&\quad \times D_{(A_1A_2}\underbrace{D_{A_3}{}^{B_3}\cdots D_{A_{2n+2}}{}^{B_{2n+2}}}_{2n} (\Delta^{k-1-n}_{2k+j-2}\sdiv_{2k+j}\phi)_{A_{2n+3}\dots A_{2k+j})B_3\dots B_{2n+2}}\nonumber\\
={}&\sum_{m=0}^{k-1}\binom{2k+j}{2m}(-2)^{-k}(\Delta^{k}_{2k+j}\phi)_{A_1\dots A_{2k+j}}
-2^{j-1}(-2)^{k}(\stwist_{2k+j-2}\mathcal{F}_{2k+j-2}\sdiv_{2k+j}\phi)_{A_1\dots A_{2k+j}}.
\end{align*}
In the second sum we have changed the order of summation followed by the change $n\rightarrow k-n-1$. 

For the operators acting on an odd number of indices we have
\begin{align*}
(\mathcal{G}_{2k+1}&\phi)_{A_1\dots A_{2k+1}}\nonumber\\
={}& 
\sum_{n=0}^{k}\binom{2k+1}{2n+1}(-2)^{-n} \underbrace{D_{(A_1}{}^{B_1}\cdots D_{A_{2k-2n}}{}^{B_{2k-2n}}}_{2k-2n} (\Delta^n_{2k+1}\phi)_{A_{2k-2n+1}\dots A_{2k+1})B_1\dots B_{2k-2n}}\nonumber\\
={}&I^{1,k}_0+\binom{2k+1}{2k+1}(-2)^{-k} (\Delta^k_{2k+1}\phi)_{A_1\dots A_{2k+1}}\nonumber\\
={}&\sum_{m=0}^{k}\binom{2k+1}{2k+1-2m}(-2)^{-k}(\Delta^{k}_{2k+1}\phi)_{A_1\dots A_{2k+1}}
-(-2)^k(\stwist_{2k-1}\mathcal{F}_{2k-1}\sdiv_{2k+1}\phi)_{A_1\dots A_{2k+1}}\nonumber\\
={}&(-2)^k (\Delta^{k}_{2k+1}\phi)_{A_1\dots A_{2k+1}}
-(-2)^k(\stwist_{2k-1}\mathcal{F}_{2k-1}\sdiv_{2k+1}\phi)_{A_1\dots A_{2k+1}}.
\end{align*}
Hence, we have \eqref{LaplacianAsGodd}.

For the operators acting on an even number of indices we can use \eqref{Gcurlexpansion} to obtain
\begin{align*}
(\mathcal{G}_{2k}&\scurl_{2k}\phi)_{A_1\dots A_{2k}} \nonumber\\
={}&
\sum_{n=0}^{k}\binom{2k}{2n}(-2)^{-n} \underbrace{D_{(A_1}{}^{B_1}\cdots D_{A_{2k-2n}}{}^{B_{2k-2n}}}_{2k-2n} (\Delta^n_{2k+1}\phi)_{A_{2k-2n+1}\dots A_{2k})B_1\dots B_{2k-2n}}\nonumber\\
={}& I^{0,k}_0+\binom{2k}{2k}(-2)^{-k} (\Delta^k_{2k}\phi)_{A_1\dots A_{2k}}\nonumber\\
={}&\sum_{m=0}^{k}\binom{2k}{2k-2m}(-2)^{-k}(\Delta^{k}_{2k}\phi)_{A_1\dots A_{2k}}
+(-2)^k(\stwist_{2k-2}\mathcal{F}_{2k-2}\sdiv_{2k}\phi)_{A_1\dots A_{2k}}\nonumber\\
={}&-(-2)^{k-1}(\Delta^{k}_{2k}\phi)_{A_1\dots A_{2k}}
+(-2)^{k-1}(\stwist_{2k-2}\mathcal{F}_{2k-2}\sdiv_{2k}\phi)_{A_1\dots A_{2k}}.
\end{align*}

Hence, we have \eqref{LaplacianAsGeven}.
\end{proof}

\bibliographystyle{abbrv}
\bibliography{biblio1}
\end{document}